\documentclass[12pt]{book}
\usepackage[margin=1.0in]{geometry}
\usepackage{amsthm}
\usepackage{times}
\usepackage{newtxmath}

\newtheorem{remark}{Remark}
\newtheorem{notation}{Notation}

\newtheorem{definition}{Definition}
\newtheorem{assumption}{Assumption}

\newtheorem{proposition}{Proposition}
\newtheorem{corollary}{Corollary}
\newtheorem{reference}{Reference}
\newtheorem{conjecture}{Conjecture}

\begin{document}

\normalfont

\title{Topologization and Functional Analytification III}
\author{Xin Tong}
\date{}

\maketitle

\subsection*{Abstract}
\rm In this paper, we continue our study on the topologization and functional analytification in $\infty$-categorical and homotopical analytic geometry. As in our previous articles on the $\infty$-categorical extensions of certain analytic and topological contexts, we discuss the corresponding prismatic cohomological constructions after Bhatt-Lurie, Bhatt-Scholze and Drinfeld, and the corresponding Robba stacks and sheaves after Kedlaya-Liu.

\newpage

\tableofcontents

\newpage

\chapter{Introduction}
\section{Considerations}

\indent In \cite{T1} and \cite{T2} we have established a project on the corresponding functional analytic constructions and topological constructions for motivic contexts in some very general $(\infty,n)$-sense. What is really happening is that we not only discuss $(\infty,n)$-spaces and their cohomologies but also discuss very general representation $(\infty,n)$-spaces in some very general analytic geometric and representation theoretic point of view. \\

\indent Our current consideration would be the continuation of the discussion we made in \cite{T1} and \cite{T2}, namely what we are going to consider will be essentially the $\infty$-categorical and homotopical constructions for the corresponding derived stacks after \cite{BK}, \cite{BBK}, \cite{BBBK}, \cite{BBM}, \cite{KKM}, \cite{CS1}, \cite{CS2}, \cite{CS3}. We revisit many constructions from \cite{T2} after \cite{BS} and \cite{BL} closely on the corresponding prismatic complexes and the corresponding prismatic stacks. In some parallel fashion we discuss some extension to the corresponding context as in \cite{KL1} and \cite{KL2}. \\

\indent In the first chapter we discuss some notations as in \cite{T2} after \cite{BK}, \cite{BBK}, \cite{BBBK}, \cite{BBM}, \cite{KKM}. These categories are actually very crucial in our development as in \cite{M}, \cite{CS1},  \cite{CS2} and \cite{CS3}. In the chapter 4 we discuss the corresponding derived prismatic cohomology for commutative algebras and derived preperfectoidizations and derived perfectoidizations after \cite{BS}, \cite{Sch} and \cite{BL}. In the chapter 5 we discuss the corresponding derived prismatic cohomology for noncommutative algebras and derived preperfectoidizations and derived perfectoidizations after \cite{BS}, \cite{Sch} and \cite{BL}. In the chapter 6 we discuss the corresponding derived prismatic cohomology for $(\infty,n)$-ringed toposes after \cite{BS}, \cite{Sch} and \cite{BL}. In the chapter 7 we discuss the corresponding derived prismatic cohomology for $(\infty,n)$-ringed toposes but restricting to inductive systems of toposes after \cite{BS}, \cite{Sch} and \cite{BL}. In the chapter 8 we discuss the derived $\varphi$-modules, derived $B$-pairs and derived vector bundles over FF-curves closely after \cite{KL1} and \cite{KL2} in the context of commutative algebras. In the chapter 9 we discuss the derived $\varphi$-modules, derived $B$-pairs and derived vector bundles over FF-curves closely after \cite{KL1} and \cite{KL2} in the context of $(\infty,n)$-ringed toposes. In the chapter 10 we discuss the derived $\varphi$-modules, derived $B$-pairs and derived vector bundles over FF-curves closely after \cite{KL1} and \cite{KL2} restricting to inductive systems of toposes. We then after these chapters make further generalizations. To summarize, we have:

\begin{proposition}
For any ring $\mathcal{R}$ in the following $\infty$-categories:
\begin{align}
&\mathrm{sComm}\mathrm{Simplicial}\mathrm{Ind}\mathrm{Seminormed}^\mathrm{formalseriescolimitcomp}_R,\\
&\mathrm{sComm}\mathrm{Simplicial}\mathrm{Ind}^m\mathrm{Seminormed}^\mathrm{formalseriescolimitcomp}_R,\\
&\mathrm{sComm}\mathrm{Simplicial}\mathrm{Ind}\mathrm{Normed}^\mathrm{formalseriescolimitcomp}_R,\\
&\mathrm{sComm}\mathrm{Simplicial}\mathrm{Ind}^m\mathrm{Normed}^\mathrm{formalseriescolimitcomp}_R,\\
&\mathrm{sComm}\mathrm{Simplicial}\mathrm{Ind}\mathrm{Banach}^\mathrm{formalseriescolimitcomp}_R,\\
&\mathrm{sComm}\mathrm{Simplicial}\mathrm{Ind}^m\mathrm{Banach}^\mathrm{formalseriescolimitcomp}_R,	
\end{align}	
we have the desired derived functional analytic prismatic complex and the desired derived functional analytic prismatic stack in the functorial way:
\begin{align}
&\mathrm{Prism}_{-/P,\mathrm{BBM},\mathrm{analytification}}(\mathcal{R}),\\
&\mathrm{CW}_{-/P}(\mathcal{R}).
\end{align}
And we have the following functors: 
\begin{align}
\widetilde{\mathcal{C}}_{-/R}(.),{\mathbb{B}_e}_{-/R}(.),{\mathbb{B}_\mathrm{dR}^+}_{-/R}(.),{\mathbb{B}_\mathrm{dR}}_{-/R}(.),{FF}_{-/R}(.),V(.)
\end{align}
of certain period rings in $p$-adic Hodge theory and $\varphi$-modules, $B$-pairs and the vector bundles over Fargues-Fontaine stacks. And the construction could be promoted to certain $\infty$-ringed $\infty$-toposes. We have the comparison for $\varphi$-modules, $B$-pairs and the vector bundles over Fargues-Fontaine stacks over the rings in the above mentioned $\infty$-categories of $\mathbb{E}_\infty$-commutative algebras, which could also be promoted to certain $\infty$-ringed $\infty$-toposes.
\end{proposition}

\begin{proposition}
For any ring $\mathcal{R}$ in the following $\infty$-categories:
\begin{align}
&\mathrm{sNoncomm}\mathrm{Simplicial}\mathrm{Ind}\mathrm{Seminormed}^\mathrm{formalseriescolimitcomp}_R,\\
&\mathrm{sNoncomm}\mathrm{Simplicial}\mathrm{Ind}^m\mathrm{Seminormed}^\mathrm{formalseriescolimitcomp}_R,\\
&\mathrm{sNoncomm}\mathrm{Simplicial}\mathrm{Ind}\mathrm{Normed}^\mathrm{formalseriescolimitcomp}_R,\\
&\mathrm{sNoncomm}\mathrm{Simplicial}\mathrm{Ind}^m\mathrm{Normed}^\mathrm{formalseriescolimitcomp}_R,\\
&\mathrm{sNoncomm}\mathrm{Simplicial}\mathrm{Ind}\mathrm{Banach}^\mathrm{formalseriescolimitcomp}_R,\\
&\mathrm{sNoncomm}\mathrm{Simplicial}\mathrm{Ind}^m\mathrm{Banach}^\mathrm{formalseriescolimitcomp}_R,	
\end{align}	
we have the desired derived functional analytic topological Hochschild complex, topological period complex and topological cyclic complex in the functorial way:
\begin{align}
&\mathrm{THH}_{-/P,\mathrm{BBM},\mathrm{analytification}}(\mathcal{R}),\\
&\mathrm{TP}_{-/P,\mathrm{BBM},\mathrm{analytification}}(\mathcal{R}),\\
&\mathrm{TC}_{-/P,\mathrm{BBM},\mathrm{analytification}}(\mathcal{R}).\\
\end{align}
\end{proposition}

\begin{proposition}
The corresponding $\infty$-categories of $\varphi$-module functors, $B$-pair functors and vector bundles functors over FF functors are equivalent over:
\begin{align}
&\mathrm{sComm}\mathrm{Simplicial}\mathrm{Ind}\mathrm{Seminormed}^\mathrm{formalseriescolimitcomp}_R,\\
&\mathrm{sComm}\mathrm{Simplicial}\mathrm{Ind}^m\mathrm{Seminormed}^\mathrm{formalseriescolimitcomp}_R,\\
&\mathrm{sComm}\mathrm{Simplicial}\mathrm{Ind}\mathrm{Normed}^\mathrm{formalseriescolimitcomp}_R,\\
&\mathrm{sComm}\mathrm{Simplicial}\mathrm{Ind}^m\mathrm{Normed}^\mathrm{formalseriescolimitcomp}_R,\\
&\mathrm{sComm}\mathrm{Simplicial}\mathrm{Ind}\mathrm{Banach}^\mathrm{formalseriescolimitcomp}_R,\\
&\mathrm{sComm}\mathrm{Simplicial}\mathrm{Ind}^m\mathrm{Banach}^\mathrm{formalseriescolimitcomp}_R,	
\end{align}
or:
\begin{align}
\mathrm{AnalyticRings}^\mathrm{CS,formalcolimitclosure}_R.	
\end{align}	
\end{proposition}

\begin{remark}
This generalizes \cite{KL1} \cite{KL2} for instance to the situation over $(\infty,1)$-categorical derived rigid analytic affinoids, with comparison possibly in certain $(\infty,1)$-categories as well. 
\end{remark}

\newpage

\section{Notations}

\subsection{Commutative Algebras}

\indent We recall our notations in \cite{T2} as in the following from \cite{BK}, \cite{BBK}, \cite{BBBK}, \cite{BBM}, \cite{KKM}.

\begin{notation}\mbox{\rm{(Rings)}}
Recall we have the following six categories on the commutative algebras in the derived sense (let $R$ be a Banach ring or $\mathbb{F}_1$):
\begin{align}
&\mathrm{sComm}\mathrm{Simplicial}\mathrm{Ind}\mathrm{Seminormed}_R,\\
&\mathrm{sComm}\mathrm{Simplicial}\mathrm{Ind}^m\mathrm{Seminormed}_R,\\
&\mathrm{sComm}\mathrm{Simplicial}\mathrm{Ind}\mathrm{Normed}_R,\\
&\mathrm{sComm}\mathrm{Simplicial}\mathrm{Ind}^m\mathrm{Normed}_R,\\
&\mathrm{sComm}\mathrm{Simplicial}\mathrm{Ind}\mathrm{Banach}_R,\\
&\mathrm{sComm}\mathrm{Simplicial}\mathrm{Ind}^m\mathrm{Banach}_R.	
\end{align}
	
\end{notation}

\begin{notation}\mbox{\rm{(Prestacks in $\infty$-groupoid)}}
Recall we have the following six categories on the prestacks in $\infty$-groupoid in the derived sense endowed with homotopy epimorphism Grothendieck topology (let $R$ be a Banach ring or $\mathbb{F}_1$):
\begin{align}
&\mathrm{PreSta}_{\mathrm{sComm}\mathrm{Simplicial}\mathrm{Ind}\mathrm{Seminormed}_R,\mathrm{homotopyepi}},\\
&\mathrm{PreSta}_{\mathrm{sComm}\mathrm{Simplicial}\mathrm{Ind}^m\mathrm{Seminormed}_R,\mathrm{homotopyepi}},\\
&\mathrm{PreSta}_{\mathrm{sComm}\mathrm{Simplicial}\mathrm{Ind}\mathrm{Normed}_R,\mathrm{homotopyepi}},\\
&\mathrm{PreSta}_{\mathrm{sComm}\mathrm{Simplicial}\mathrm{Ind}^m\mathrm{Normed}_R,\mathrm{homotopyepi}},\\
&\mathrm{PreSta}_{\mathrm{sComm}\mathrm{Simplicial}\mathrm{Ind}\mathrm{Banach}_R,\mathrm{homotopyepi}},\\
&\mathrm{PreSta}_{\mathrm{sComm}\mathrm{Simplicial}\mathrm{Ind}^m\mathrm{Banach}_R,\mathrm{homotopyepi}}.	
\end{align}
	
\end{notation}

\begin{notation}\mbox{\rm{(Stacks in $\infty$-groupoid)}}
Recall we have the following six categories on the functors in $\infty$-groupoid in the derived sense endowed with homotopy epimorphism Grothendieck topology (let $R$ be a Banach ring or $\mathbb{F}_1$) satisfying the corresponding descent requirement for this given topology:
\begin{align}
&\mathrm{Sta}_{\mathrm{sComm}\mathrm{Simplicial}\mathrm{Ind}\mathrm{Seminormed}_R,\mathrm{homotopyepi}},\\
&\mathrm{Sta}_{\mathrm{sComm}\mathrm{Simplicial}\mathrm{Ind}^m\mathrm{Seminormed}_R,\mathrm{homotopyepi}},\\
&\mathrm{Sta}_{\mathrm{sComm}\mathrm{Simplicial}\mathrm{Ind}\mathrm{Normed}_R,\mathrm{homotopyepi}},\\
&\mathrm{Sta}_{\mathrm{sComm}\mathrm{Simplicial}\mathrm{Ind}^m\mathrm{Normed}_R,\mathrm{homotopyepi}},\\
&\mathrm{Sta}_{\mathrm{sComm}\mathrm{Simplicial}\mathrm{Ind}\mathrm{Banach}_R,\mathrm{homotopyepi}},\\
&\mathrm{Sta}_{\mathrm{sComm}\mathrm{Simplicial}\mathrm{Ind}^m\mathrm{Banach}_R,\mathrm{homotopyepi}}.	
\end{align}
	
\end{notation}

\begin{notation}\mbox{\rm{($\infty$-Ringed Toposes in $\infty$-groupoid)}}\\
Recall we have the following six categories on the $\infty$-ringed functors in $\infty$-groupoid in the derived sense endowed with homotopy epimorphism Grothendieck topology (let $R$ be a Banach ring or $\mathbb{F}_1$) satisfying the corresponding descent requirement for this given topology:
\begin{align}
&\mathrm{Sta}^\mathrm{derivedringed,\sharp}_{\mathrm{sComm}\mathrm{Simplicial}\mathrm{Ind}\mathrm{Seminormed}_R,\mathrm{homotopyepi}},\\
&\mathrm{Sta}^\mathrm{derivedringed,\sharp}_{\mathrm{sComm}\mathrm{Simplicial}\mathrm{Ind}^m\mathrm{Seminormed}_R,\mathrm{homotopyepi}},\\
&\mathrm{Sta}^\mathrm{derivedringed,\sharp}_{\mathrm{sComm}\mathrm{Simplicial}\mathrm{Ind}\mathrm{Normed}_R,\mathrm{homotopyepi}},\\
&\mathrm{Sta}^\mathrm{derivedringed,\sharp}_{\mathrm{sComm}\mathrm{Simplicial}\mathrm{Ind}^m\mathrm{Normed}_R,\mathrm{homotopyepi}},\\
&\mathrm{Sta}^\mathrm{derivedringed,\sharp}_{\mathrm{sComm}\mathrm{Simplicial}\mathrm{Ind}\mathrm{Banach}_R,\mathrm{homotopyepi}},\\
&\mathrm{Sta}^\mathrm{derivedringed,\sharp}_{\mathrm{sComm}\mathrm{Simplicial}\mathrm{Ind}^m\mathrm{Banach}_R,\mathrm{homotopyepi}}.	
\end{align}
Here $\sharp$ represents any category in the following:
\begin{align}
&\mathrm{sComm}\mathrm{Simplicial}\mathrm{Ind}\mathrm{Seminormed}_R,\\
&\mathrm{sComm}\mathrm{Simplicial}\mathrm{Ind}^m\mathrm{Seminormed}_R,\\
&\mathrm{sComm}\mathrm{Simplicial}\mathrm{Ind}\mathrm{Normed}_R,\\
&\mathrm{sComm}\mathrm{Simplicial}\mathrm{Ind}^m\mathrm{Normed}_R,\\
&\mathrm{sComm}\mathrm{Simplicial}\mathrm{Ind}\mathrm{Banach}_R,\\
&\mathrm{sComm}\mathrm{Simplicial}\mathrm{Ind}^m\mathrm{Banach}_R.	
\end{align}	
\end{notation}

\begin{notation}\mbox{\rm{(Quasicoherent Presheaves over $\infty$-Ringed Toposes in $\infty$-groupoid)}}\\
Recall we have the following six categories on the $\infty$-ringed functors in $\infty$-groupoid in the derived sense endowed with homotopy epimorphism Grothendieck topology (let $R$ be a Banach ring or $\mathbb{F}_1$) satisfying the corresponding descent requirement for this given topology:
\begin{align}
&\mathrm{Sta}^\mathrm{derivedringed,\sharp}_{\mathrm{sComm}\mathrm{Simplicial}\mathrm{Ind}\mathrm{Seminormed}_R,\mathrm{homotopyepi}},\\
&\mathrm{Sta}^\mathrm{derivedringed,\sharp}_{\mathrm{sComm}\mathrm{Simplicial}\mathrm{Ind}^m\mathrm{Seminormed}_R,\mathrm{homotopyepi}},\\
&\mathrm{Sta}^\mathrm{derivedringed,\sharp}_{\mathrm{sComm}\mathrm{Simplicial}\mathrm{Ind}\mathrm{Normed}_R,\mathrm{homotopyepi}},\\
&\mathrm{Sta}^\mathrm{derivedringed,\sharp}_{\mathrm{sComm}\mathrm{Simplicial}\mathrm{Ind}^m\mathrm{Normed}_R,\mathrm{homotopyepi}},\\
&\mathrm{Sta}^\mathrm{derivedringed,\sharp}_{\mathrm{sComm}\mathrm{Simplicial}\mathrm{Ind}\mathrm{Banach}_R,\mathrm{homotopyepi}},\\
&\mathrm{Sta}^\mathrm{derivedringed,\sharp}_{\mathrm{sComm}\mathrm{Simplicial}\mathrm{Ind}^m\mathrm{Banach}_R,\mathrm{homotopyepi}}.	
\end{align}
Here $\sharp$ represents any category in the following:
\begin{align}
&\mathrm{sComm}\mathrm{Simplicial}\mathrm{Ind}\mathrm{Seminormed}_R,\\
&\mathrm{sComm}\mathrm{Simplicial}\mathrm{Ind}^m\mathrm{Seminormed}_R,\\
&\mathrm{sComm}\mathrm{Simplicial}\mathrm{Ind}\mathrm{Normed}_R,\\
&\mathrm{sComm}\mathrm{Simplicial}\mathrm{Ind}^m\mathrm{Normed}_R,\\
&\mathrm{sComm}\mathrm{Simplicial}\mathrm{Ind}\mathrm{Banach}_R,\\
&\mathrm{sComm}\mathrm{Simplicial}\mathrm{Ind}^m\mathrm{Banach}_R.	
\end{align}	
We then have the corresponding $\infty$-categories of the corresponding quasicoherent presheaves of $\mathcal{O}$-modules:
\begin{align}
&\mathrm{Quasicoherentpresheaves,Sta}^\mathrm{derivedringed,\sharp}_{\mathrm{sComm}\mathrm{Simplicial}\mathrm{Ind}\mathrm{Seminormed}_R,\mathrm{homotopyepi}},\\
&\mathrm{Quasicoherentpresheaves,Sta}^\mathrm{derivedringed,\sharp}_{\mathrm{sComm}\mathrm{Simplicial}\mathrm{Ind}^m\mathrm{Seminormed}_R,\mathrm{homotopyepi}},\\
&\mathrm{Quasicoherentpresheaves,Sta}^\mathrm{derivedringed,\sharp}_{\mathrm{sComm}\mathrm{Simplicial}\mathrm{Ind}\mathrm{Normed}_R,\mathrm{homotopyepi}},\\
&\mathrm{Quasicoherentpresheaves,Sta}^\mathrm{derivedringed,\sharp}_{\mathrm{sComm}\mathrm{Simplicial}\mathrm{Ind}^m\mathrm{Normed}_R,\mathrm{homotopyepi}},\\
&\mathrm{Quasicoherentpresheaves,Sta}^\mathrm{derivedringed,\sharp}_{\mathrm{sComm}\mathrm{Simplicial}\mathrm{Ind}\mathrm{Banach}_R,\mathrm{homotopyepi}},\\
&\mathrm{Quasicoherentpresheaves,Sta}^\mathrm{derivedringed,\sharp}_{\mathrm{sComm}\mathrm{Simplicial}\mathrm{Ind}^m\mathrm{Banach}_R,\mathrm{homotopyepi}}.	
\end{align}
Here $\sharp$ represents any category in the following:
\begin{align}
&\mathrm{sComm}\mathrm{Simplicial}\mathrm{Ind}\mathrm{Seminormed}_R,\\
&\mathrm{sComm}\mathrm{Simplicial}\mathrm{Ind}^m\mathrm{Seminormed}_R,\\
&\mathrm{sComm}\mathrm{Simplicial}\mathrm{Ind}\mathrm{Normed}_R,\\
&\mathrm{sComm}\mathrm{Simplicial}\mathrm{Ind}^m\mathrm{Normed}_R,\\
&\mathrm{sComm}\mathrm{Simplicial}\mathrm{Ind}\mathrm{Banach}_R,\\
&\mathrm{sComm}\mathrm{Simplicial}\mathrm{Ind}^m\mathrm{Banach}_R.	
\end{align} 
\end{notation}

\begin{notation}\mbox{\rm{(Quasicoherent Sheaves over $\infty$-Ringed Toposes in $\infty$-groupoid)}}
Recall we have the following six categories on the $\infty$-ringed functors in $\infty$-groupoid in the derived sense endowed with homotopy epimorphism Grothendieck topology (let $R$ be a Banach ring or $\mathbb{F}_1$) satisfying the corresponding descent requirement for this given topology:
\begin{align}
&\mathrm{Sta}^\mathrm{derivedringed,\sharp}_{\mathrm{sComm}\mathrm{Simplicial}\mathrm{Ind}\mathrm{Seminormed}_R,\mathrm{homotopyepi}},\\
&\mathrm{Sta}^\mathrm{derivedringed,\sharp}_{\mathrm{sComm}\mathrm{Simplicial}\mathrm{Ind}^m\mathrm{Seminormed}_R,\mathrm{homotopyepi}},\\
&\mathrm{Sta}^\mathrm{derivedringed,\sharp}_{\mathrm{sComm}\mathrm{Simplicial}\mathrm{Ind}\mathrm{Normed}_R,\mathrm{homotopyepi}},\\
&\mathrm{Sta}^\mathrm{derivedringed,\sharp}_{\mathrm{sComm}\mathrm{Simplicial}\mathrm{Ind}^m\mathrm{Normed}_R,\mathrm{homotopyepi}},\\
&\mathrm{Sta}^\mathrm{derivedringed,\sharp}_{\mathrm{sComm}\mathrm{Simplicial}\mathrm{Ind}\mathrm{Banach}_R,\mathrm{homotopyepi}},\\
&\mathrm{Sta}^\mathrm{derivedringed,\sharp}_{\mathrm{sComm}\mathrm{Simplicial}\mathrm{Ind}^m\mathrm{Banach}_R,\mathrm{homotopyepi}}.	
\end{align}
Here $\sharp$ represents any category in the following:
\begin{align}
&\mathrm{sComm}\mathrm{Simplicial}\mathrm{Ind}\mathrm{Seminormed}_R,\\
&\mathrm{sComm}\mathrm{Simplicial}\mathrm{Ind}^m\mathrm{Seminormed}_R,\\
&\mathrm{sComm}\mathrm{Simplicial}\mathrm{Ind}\mathrm{Normed}_R,\\
&\mathrm{sComm}\mathrm{Simplicial}\mathrm{Ind}^m\mathrm{Normed}_R,\\
&\mathrm{sComm}\mathrm{Simplicial}\mathrm{Ind}\mathrm{Banach}_R,\\
&\mathrm{sComm}\mathrm{Simplicial}\mathrm{Ind}^m\mathrm{Banach}_R.	
\end{align}	
We then have the corresponding $\infty$-categories of the corresponding quasicoherent sheaves of $\mathcal{O}$-modules:
\begin{align}
&\mathrm{Quasicoherentpresheaves,Sta}^\mathrm{derivedringed,\sharp}_{\mathrm{sComm}\mathrm{Simplicial}\mathrm{Ind}\mathrm{Seminormed}_R,\mathrm{homotopyepi}},\\
&\mathrm{Quasicoherentpresheaves,Sta}^\mathrm{derivedringed,\sharp}_{\mathrm{sComm}\mathrm{Simplicial}\mathrm{Ind}^m\mathrm{Seminormed}_R,\mathrm{homotopyepi}},\\
&\mathrm{Quasicoherentpresheaves,Sta}^\mathrm{derivedringed,\sharp}_{\mathrm{sComm}\mathrm{Simplicial}\mathrm{Ind}\mathrm{Normed}_R,\mathrm{homotopyepi}},\\
&\mathrm{Quasicoherentpresheaves,Sta}^\mathrm{derivedringed,\sharp}_{\mathrm{sComm}\mathrm{Simplicial}\mathrm{Ind}^m\mathrm{Normed}_R,\mathrm{homotopyepi}},\\
&\mathrm{Quasicoherentpresheaves,Sta}^\mathrm{derivedringed,\sharp}_{\mathrm{sComm}\mathrm{Simplicial}\mathrm{Ind}\mathrm{Banach}_R,\mathrm{homotopyepi}},\\
&\mathrm{Quasicoherentpresheaves,Sta}^\mathrm{derivedringed,\sharp}_{\mathrm{sComm}\mathrm{Simplicial}\mathrm{Ind}^m\mathrm{Banach}_R,\mathrm{homotopyepi}}.	
\end{align}
Here $\sharp$ represents any category in the following:
\begin{align}
&\mathrm{sComm}\mathrm{Simplicial}\mathrm{Ind}\mathrm{Seminormed}_R,\\
&\mathrm{sComm}\mathrm{Simplicial}\mathrm{Ind}^m\mathrm{Seminormed}_R,\\
&\mathrm{sComm}\mathrm{Simplicial}\mathrm{Ind}\mathrm{Normed}_R,\\
&\mathrm{sComm}\mathrm{Simplicial}\mathrm{Ind}^m\mathrm{Normed}_R,\\
&\mathrm{sComm}\mathrm{Simplicial}\mathrm{Ind}\mathrm{Banach}_R,\\
&\mathrm{sComm}\mathrm{Simplicial}\mathrm{Ind}^m\mathrm{Banach}_R.	
\end{align} 
\end{notation}

\newpage
\subsection{Noncommutative Algebras}

\indent We recall our notations in \cite{T2} as in the following from \cite{BK}, \cite{BBK}, \cite{BBBK}, \cite{BBM}, \cite{KKM}.

\begin{notation}\mbox{\rm{(Rings)}}
Recall we have the following six categories on the noncommutative algebras in the derived sense (let $R$ be a Banach ring or $\mathbb{F}_1$):
\begin{align}
&\mathrm{sNoncomm}\mathrm{Simplicial}\mathrm{Ind}\mathrm{Seminormed}_R,\\
&\mathrm{sNoncomm}\mathrm{Simplicial}\mathrm{Ind}^m\mathrm{Seminormed}_R,\\
&\mathrm{sNoncomm}\mathrm{Simplicial}\mathrm{Ind}\mathrm{Normed}_R,\\
&\mathrm{sNoncomm}\mathrm{Simplicial}\mathrm{Ind}^m\mathrm{Normed}_R,\\
&\mathrm{sNoncomm}\mathrm{Simplicial}\mathrm{Ind}\mathrm{Banach}_R,\\
&\mathrm{sNoncomm}\mathrm{Simplicial}\mathrm{Ind}^m\mathrm{Banach}_R.	
\end{align}
	
\end{notation}

\begin{notation}\mbox{\rm{(Prestacks in $\infty$-groupoid)}}
Recall we have the following six categories on the prestacks in $\infty$-groupoid in the derived sense endowed with homotopy epimorphism Grothendieck topology (let $R$ be a Banach ring or $\mathbb{F}_1$):
\begin{align}
&\mathrm{PreSta}_{\mathrm{sNoncomm}\mathrm{Simplicial}\mathrm{Ind}\mathrm{Seminormed}_R,\mathrm{homotopyepi}},\\
&\mathrm{PreSta}_{\mathrm{sNoncomm}\mathrm{Simplicial}\mathrm{Ind}^m\mathrm{Seminormed}_R,\mathrm{homotopyepi}},\\
&\mathrm{PreSta}_{\mathrm{sNoncomm}\mathrm{Simplicial}\mathrm{Ind}\mathrm{Normed}_R,\mathrm{homotopyepi}},\\
&\mathrm{PreSta}_{\mathrm{sNoncomm}\mathrm{Simplicial}\mathrm{Ind}^m\mathrm{Normed}_R,\mathrm{homotopyepi}},\\
&\mathrm{PreSta}_{\mathrm{sNoncomm}\mathrm{Simplicial}\mathrm{Ind}\mathrm{Banach}_R,\mathrm{homotopyepi}},\\
&\mathrm{PreSta}_{\mathrm{sNoncomm}\mathrm{Simplicial}\mathrm{Ind}^m\mathrm{Banach}_R,\mathrm{homotopyepi}}.	
\end{align}
	
\end{notation}

\begin{notation}\mbox{\rm{(Stacks in $\infty$-groupoid)}}
Recall we have the following six categories on the functors in $\infty$-groupoid in the derived sense endowed with homotopy epimorphism Grothendieck topology (let $R$ be a Banach ring or $\mathbb{F}_1$) satisfying the corresponding descent requirement for this given topology:
\begin{align}
&\mathrm{Sta}_{\mathrm{sNoncomm}\mathrm{Simplicial}\mathrm{Ind}\mathrm{Seminormed}_R,\mathrm{homotopyepi}},\\
&\mathrm{Sta}_{\mathrm{sNoncomm}\mathrm{Simplicial}\mathrm{Ind}^m\mathrm{Seminormed}_R,\mathrm{homotopyepi}},\\
&\mathrm{Sta}_{\mathrm{sNoncomm}\mathrm{Simplicial}\mathrm{Ind}\mathrm{Normed}_R,\mathrm{homotopyepi}},\\
&\mathrm{Sta}_{\mathrm{sNoncomm}\mathrm{Simplicial}\mathrm{Ind}^m\mathrm{Normed}_R,\mathrm{homotopyepi}},\\
&\mathrm{Sta}_{\mathrm{sNoncomm}\mathrm{Simplicial}\mathrm{Ind}\mathrm{Banach}_R,\mathrm{homotopyepi}},\\
&\mathrm{Sta}_{\mathrm{sNoncomm}\mathrm{Simplicial}\mathrm{Ind}^m\mathrm{Banach}_R,\mathrm{homotopyepi}}.	
\end{align}
	
\end{notation}

\begin{notation}\mbox{\rm{($\infty$-Ringed Toposes in $\infty$-groupoid)}}\\
Recall we have the following six categories on the $\infty$-ringed functors in $\infty$-groupoid in the derived sense endowed with homotopy epimorphism Grothendieck topology (let $R$ be a Banach ring or $\mathbb{F}_1$) satisfying the corresponding descent requirement for this given topology:
\begin{align}
&\mathrm{Sta}^\mathrm{derivedringed,\sharp}_{\mathrm{sNoncomm}\mathrm{Simplicial}\mathrm{Ind}\mathrm{Seminormed}_R,\mathrm{homotopyepi}},\\
&\mathrm{Sta}^\mathrm{derivedringed,\sharp}_{\mathrm{sNoncomm}\mathrm{Simplicial}\mathrm{Ind}^m\mathrm{Seminormed}_R,\mathrm{homotopyepi}},\\
&\mathrm{Sta}^\mathrm{derivedringed,\sharp}_{\mathrm{sNoncomm}\mathrm{Simplicial}\mathrm{Ind}\mathrm{Normed}_R,\mathrm{homotopyepi}},\\
&\mathrm{Sta}^\mathrm{derivedringed,\sharp}_{\mathrm{sNoncomm}\mathrm{Simplicial}\mathrm{Ind}^m\mathrm{Normed}_R,\mathrm{homotopyepi}},\\
&\mathrm{Sta}^\mathrm{derivedringed,\sharp}_{\mathrm{sNoncomm}\mathrm{Simplicial}\mathrm{Ind}\mathrm{Banach}_R,\mathrm{homotopyepi}},\\
&\mathrm{Sta}^\mathrm{derivedringed,\sharp}_{\mathrm{sNoncomm}\mathrm{Simplicial}\mathrm{Ind}^m\mathrm{Banach}_R,\mathrm{homotopyepi}}.	
\end{align}
Here $\sharp$ represents any category in the following:
\begin{align}
&\mathrm{sNoncomm}\mathrm{Simplicial}\mathrm{Ind}\mathrm{Seminormed}_R,\\
&\mathrm{sNoncomm}\mathrm{Simplicial}\mathrm{Ind}^m\mathrm{Seminormed}_R,\\
&\mathrm{sNoncomm}\mathrm{Simplicial}\mathrm{Ind}\mathrm{Normed}_R,\\
&\mathrm{sNoncomm}\mathrm{Simplicial}\mathrm{Ind}^m\mathrm{Normed}_R,\\
&\mathrm{sNoncomm}\mathrm{Simplicial}\mathrm{Ind}\mathrm{Banach}_R,\\
&\mathrm{sNoncomm}\mathrm{Simplicial}\mathrm{Ind}^m\mathrm{Banach}_R.	
\end{align}	
\end{notation}

\begin{notation}\mbox{\rm{(Quasicoherent Presheaves over $\infty$-Ringed Toposes in $\infty$-groupoid)}}\\
Recall we have the following six categories on the $\infty$-ringed functors in $\infty$-groupoid in the derived sense endowed with homotopy epimorphism Grothendieck topology (let $R$ be a Banach ring or $\mathbb{F}_1$) satisfying the corresponding descent requirement for this given topology:
\begin{align}
&\mathrm{Sta}^\mathrm{derivedringed,\sharp}_{\mathrm{sNoncomm}\mathrm{Simplicial}\mathrm{Ind}\mathrm{Seminormed}_R,\mathrm{homotopyepi}},\\
&\mathrm{Sta}^\mathrm{derivedringed,\sharp}_{\mathrm{sNoncomm}\mathrm{Simplicial}\mathrm{Ind}^m\mathrm{Seminormed}_R,\mathrm{homotopyepi}},\\
&\mathrm{Sta}^\mathrm{derivedringed,\sharp}_{\mathrm{sNoncomm}\mathrm{Simplicial}\mathrm{Ind}\mathrm{Normed}_R,\mathrm{homotopyepi}},\\
&\mathrm{Sta}^\mathrm{derivedringed,\sharp}_{\mathrm{sNoncomm}\mathrm{Simplicial}\mathrm{Ind}^m\mathrm{Normed}_R,\mathrm{homotopyepi}},\\
&\mathrm{Sta}^\mathrm{derivedringed,\sharp}_{\mathrm{sNoncomm}\mathrm{Simplicial}\mathrm{Ind}\mathrm{Banach}_R,\mathrm{homotopyepi}},\\
&\mathrm{Sta}^\mathrm{derivedringed,\sharp}_{\mathrm{sNoncomm}\mathrm{Simplicial}\mathrm{Ind}^m\mathrm{Banach}_R,\mathrm{homotopyepi}}.	
\end{align}
Here $\sharp$ represents any category in the following:
\begin{align}
&\mathrm{sNoncomm}\mathrm{Simplicial}\mathrm{Ind}\mathrm{Seminormed}_R,\\
&\mathrm{sNoncomm}\mathrm{Simplicial}\mathrm{Ind}^m\mathrm{Seminormed}_R,\\
&\mathrm{sNoncomm}\mathrm{Simplicial}\mathrm{Ind}\mathrm{Normed}_R,\\
&\mathrm{sNoncomm}\mathrm{Simplicial}\mathrm{Ind}^m\mathrm{Normed}_R,\\
&\mathrm{sNoncomm}\mathrm{Simplicial}\mathrm{Ind}\mathrm{Banach}_R,\\
&\mathrm{sNoncomm}\mathrm{Simplicial}\mathrm{Ind}^m\mathrm{Banach}_R.	
\end{align}	
We then have the corresponding $\infty$-categories of the corresponding quasicoherent presheaves of $\mathcal{O}$-modules:
\begin{align}
&\mathrm{Quasicoherentpresheaves,Sta}^\mathrm{derivedringed,\sharp}_{\mathrm{sNoncomm}\mathrm{Simplicial}\mathrm{Ind}\mathrm{Seminormed}_R,\mathrm{homotopyepi}},\\
&\mathrm{Quasicoherentpresheaves,Sta}^\mathrm{derivedringed,\sharp}_{\mathrm{sNoncomm}\mathrm{Simplicial}\mathrm{Ind}^m\mathrm{Seminormed}_R,\mathrm{homotopyepi}},\\
&\mathrm{Quasicoherentpresheaves,Sta}^\mathrm{derivedringed,\sharp}_{\mathrm{sNoncomm}\mathrm{Simplicial}\mathrm{Ind}\mathrm{Normed}_R,\mathrm{homotopyepi}},\\
&\mathrm{Quasicoherentpresheaves,Sta}^\mathrm{derivedringed,\sharp}_{\mathrm{sNoncomm}\mathrm{Simplicial}\mathrm{Ind}^m\mathrm{Normed}_R,\mathrm{homotopyepi}},\\
&\mathrm{Quasicoherentpresheaves,Sta}^\mathrm{derivedringed,\sharp}_{\mathrm{sNoncomm}\mathrm{Simplicial}\mathrm{Ind}\mathrm{Banach}_R,\mathrm{homotopyepi}},\\
&\mathrm{Quasicoherentpresheaves,Sta}^\mathrm{derivedringed,\sharp}_{\mathrm{sNoncomm}\mathrm{Simplicial}\mathrm{Ind}^m\mathrm{Banach}_R,\mathrm{homotopyepi}}.	
\end{align}
Here $\sharp$ represents any category in the following:
\begin{align}
&\mathrm{sNoncomm}\mathrm{Simplicial}\mathrm{Ind}\mathrm{Seminormed}_R,\\
&\mathrm{sNoncomm}\mathrm{Simplicial}\mathrm{Ind}^m\mathrm{Seminormed}_R,\\
&\mathrm{sNoncomm}\mathrm{Simplicial}\mathrm{Ind}\mathrm{Normed}_R,\\
&\mathrm{sNoncomm}\mathrm{Simplicial}\mathrm{Ind}^m\mathrm{Normed}_R,\\
&\mathrm{sNoncomm}\mathrm{Simplicial}\mathrm{Ind}\mathrm{Banach}_R,\\
&\mathrm{sNoncomm}\mathrm{Simplicial}\mathrm{Ind}^m\mathrm{Banach}_R.	
\end{align} 
\end{notation}

\begin{notation}\mbox{\rm{(Quasicoherent Sheaves over $\infty$-Ringed Toposes in $\infty$-groupoid)}}
Recall we have the following six categories on the $\infty$-ringed functors in $\infty$-groupoid in the derived sense endowed with homotopy epimorphism Grothendieck topology (let $R$ be a Banach ring or $\mathbb{F}_1$) satisfying the corresponding descent requirement for this given topology:
\begin{align}
&\mathrm{Sta}^\mathrm{derivedringed,\sharp}_{\mathrm{sNoncomm}\mathrm{Simplicial}\mathrm{Ind}\mathrm{Seminormed}_R,\mathrm{homotopyepi}},\\
&\mathrm{Sta}^\mathrm{derivedringed,\sharp}_{\mathrm{sNoncomm}\mathrm{Simplicial}\mathrm{Ind}^m\mathrm{Seminormed}_R,\mathrm{homotopyepi}},\\
&\mathrm{Sta}^\mathrm{derivedringed,\sharp}_{\mathrm{sNoncomm}\mathrm{Simplicial}\mathrm{Ind}\mathrm{Normed}_R,\mathrm{homotopyepi}},\\
&\mathrm{Sta}^\mathrm{derivedringed,\sharp}_{\mathrm{sNoncomm}\mathrm{Simplicial}\mathrm{Ind}^m\mathrm{Normed}_R,\mathrm{homotopyepi}},\\
&\mathrm{Sta}^\mathrm{derivedringed,\sharp}_{\mathrm{sNoncomm}\mathrm{Simplicial}\mathrm{Ind}\mathrm{Banach}_R,\mathrm{homotopyepi}},\\
&\mathrm{Sta}^\mathrm{derivedringed,\sharp}_{\mathrm{sNoncomm}\mathrm{Simplicial}\mathrm{Ind}^m\mathrm{Banach}_R,\mathrm{homotopyepi}}.	
\end{align}
Here $\sharp$ represents any category in the following:
\begin{align}
&\mathrm{sNoncomm}\mathrm{Simplicial}\mathrm{Ind}\mathrm{Seminormed}_R,\\
&\mathrm{sNoncomm}\mathrm{Simplicial}\mathrm{Ind}^m\mathrm{Seminormed}_R,\\
&\mathrm{sNoncomm}\mathrm{Simplicial}\mathrm{Ind}\mathrm{Normed}_R,\\
&\mathrm{sNoncomm}\mathrm{Simplicial}\mathrm{Ind}^m\mathrm{Normed}_R,\\
&\mathrm{sNoncomm}\mathrm{Simplicial}\mathrm{Ind}\mathrm{Banach}_R,\\
&\mathrm{sNoncomm}\mathrm{Simplicial}\mathrm{Ind}^m\mathrm{Banach}_R.	
\end{align}	
We then have the corresponding $\infty$-categories of the corresponding quasicoherent sheaves of $\mathcal{O}$-modules:
\begin{align}
&\mathrm{Quasicoherentpresheaves,Sta}^\mathrm{derivedringed,\sharp}_{\mathrm{sNoncomm}\mathrm{Simplicial}\mathrm{Ind}\mathrm{Seminormed}_R,\mathrm{homotopyepi}},\\
&\mathrm{Quasicoherentpresheaves,Sta}^\mathrm{derivedringed,\sharp}_{\mathrm{sNoncomm}\mathrm{Simplicial}\mathrm{Ind}^m\mathrm{Seminormed}_R,\mathrm{homotopyepi}},\\
&\mathrm{Quasicoherentpresheaves,Sta}^\mathrm{derivedringed,\sharp}_{\mathrm{sNoncomm}\mathrm{Simplicial}\mathrm{Ind}\mathrm{Normed}_R,\mathrm{homotopyepi}},\\
&\mathrm{Quasicoherentpresheaves,Sta}^\mathrm{derivedringed,\sharp}_{\mathrm{sNoncomm}\mathrm{Simplicial}\mathrm{Ind}^m\mathrm{Normed}_R,\mathrm{homotopyepi}},\\
&\mathrm{Quasicoherentpresheaves,Sta}^\mathrm{derivedringed,\sharp}_{\mathrm{sNoncomm}\mathrm{Simplicial}\mathrm{Ind}\mathrm{Banach}_R,\mathrm{homotopyepi}},\\
&\mathrm{Quasicoherentpresheaves,Sta}^\mathrm{derivedringed,\sharp}_{\mathrm{sNoncomm}\mathrm{Simplicial}\mathrm{Ind}^m\mathrm{Banach}_R,\mathrm{homotopyepi}}.	
\end{align}
Here $\sharp$ represents any category in the following:
\begin{align}
&\mathrm{sNoncomm}\mathrm{Simplicial}\mathrm{Ind}\mathrm{Seminormed}_R,\\
&\mathrm{sNoncomm}\mathrm{Simplicial}\mathrm{Ind}^m\mathrm{Seminormed}_R,\\
&\mathrm{sNoncomm}\mathrm{Simplicial}\mathrm{Ind}\mathrm{Normed}_R,\\
&\mathrm{sNoncomm}\mathrm{Simplicial}\mathrm{Ind}^m\mathrm{Normed}_R,\\
&\mathrm{sNoncomm}\mathrm{Simplicial}\mathrm{Ind}\mathrm{Banach}_R,\\
&\mathrm{sNoncomm}\mathrm{Simplicial}\mathrm{Ind}^m\mathrm{Banach}_R.	
\end{align} 
\end{notation}

\begin{remark}
The corresponding noncommutative $\infty$-ringed structure over noncommutative $\infty$-toposes could be defined to be corresponding noncommutative $\infty$-ringed structure over commutative $\infty$-toposes. Certainly this will have its own interest if one would like to study the corresponding noncommutative deformation of the structure sheaves.	
\end{remark}

\chapter{$(\infty,1)$-Categorical Functional Analytification}

\newpage
\section{$(\infty,1)$-Categoricalization}

\indent Let $R$ be a Banach ring. We assume in our situation that $R$ itself is commutative. We will later on build up the corresponding foundations as in \cite{BBBK}, we consider the following categories:
\begin{align}
\mathrm{SemiNormed}_R, \mathrm{Normed}_R, \mathrm{Banach}_R
\end{align}
which are the corresponding semi-normed module over $R$, normed module over $R$, and finally the corresponding Banach module over $R$. One would like to construct the corresponding model categories with well-established model categorical structures. As in \cite{BBBK}, we consider the following construction:

\begin{definition}\mbox{\textbf{(\cite[Definition 3.1]{BBBK})}} Let $R$ be a Banach commutative algebra. Consider the following categories:
\begin{align}
\mathrm{SemiNormed}_R, \mathrm{Normed}_R, \mathrm{Banach}_R.
\end{align}
For each $C$ of these three categories we consider the inductive categories and monomorphic inductive categories associated to $C$, which will be denoted by:
\begin{align}
\mathrm{Ind}C,\mathrm{Ind}_\text{monomorphic}C.
\end{align}
\end{definition}

\begin{proposition}\mbox{\textbf{(\cite[Theorem 3.14]{BBBK})}}
Let $R$ be a Banach commutative algebra. Consider the following categories:
\begin{align}
\mathrm{SemiNormed}_R, \mathrm{Normed}_R, \mathrm{Banach}_R.
\end{align}
For each $C$ of these three categories we consider the inductive categories and monomorphic inductive categories associated to $C$, which will be denoted by:
\begin{align}
\mathrm{Ind}C,\mathrm{Ind}_\text{monomorphic}C.
\end{align}
Then all these categories can present certain $(\infty,1)$-categories. We then use the following notations to denote them:
\begin{align}
&D(\mathrm{Ind}\mathrm{SemiNormed}_R), D(\mathrm{Ind}\mathrm{Normed}_R), D(\mathrm{Ind}\mathrm{Banach}_R),\\
&D(\mathrm{Ind}_\text{monomorphic}\mathrm{SemiNormed}_R), D(\mathrm{Ind}_\text{monomorphic}\mathrm{Normed}_R), D(\mathrm{Ind}_\text{monomorphic}\mathrm{Banach}_R).
\end{align}
\end{proposition}

\begin{proof}
See \cite{BBBK}. The admissible model structures are sufficient to provide such structures.
\end{proof}

\newpage
\section{$(\infty,1)$-Sheafiness}

\indent The adic spaces in the sense of Huber as in \cite{Hu} essentially provide a framework for defining a specture in the anlytic situation associated to a Banach commutative ring $R$, but the sheafiness condition has to be required since the corresponding exact sequence in defining the sheafiness of the obvious structure ring structure might not be always holds for general $R$. But the work of \cite{BBBK} provides certain derived structure ring structure which solves the issue. The construction is made in \cite{BK}.

\begin{assumption}
Let $R$ be a general commutative Banach ring over a certain base $k$ as in \cite{BK}. For instance if $R/\mathbb{Q}_p$ is defined over $p$-adic number field $\mathbb{Q}_p$, then the discussion in the following will satisfy this requirement. The generality on the ring $R$ can be further generalized by applying the foundation in \cite{Ked1}.
\end{assumption}

\noindent The work of \cite{BK} defines a site associated to $R$, carrying the topology by taking the corresponding derived analytic rational localization, i.e. the corresponding Koszul complexes in the pure algebraic sense:
\begin{align}
\mathrm{Koszul}_{f,g}(R):=R/^\mathrm{derived}\{g/f\}
\end{align}
with certain induction to define the following:
\begin{align}
\mathrm{Koszul}_{f,g_1,...,g_n}(R).
\end{align}

Then we have the corresponding stack $(\mathrm{Spec}R,\mathcal{O}_{\mathrm{Spec}R})$ over the site 
\begin{align}
\underset{D(\mathrm{Ind}\mathrm{Banach}_k)}{\mathrm{CommutativeAlgebra}}
\end{align}
carrying Grothendieck topology by using the rational localization in the derived sense.

\begin{proposition} \mbox{\textbf{(Bambozzi-Kremnizer \cite[Definition 4.30, Proposition 4.33, Proposition 4.4]{BK})}}
Attached to $R$, we have a general $(\infty,1)$-fiber category/$(\infty,1)$-stack over the site 
\begin{align}
\underset{D(\mathrm{Ind}\mathrm{Banach}_k)}{\mathrm{CommutativeAlgebra}}.
\end{align}
The followings are $(\infty,1)$-Banach rings:
\begin{align}
\mathrm{Koszul}_{f,g}(R):=R/^\mathrm{derived}\{g/f\}
\end{align}
\begin{align}
\mathrm{Koszul}_{f,g_1,...,g_n}(R).
\end{align}
The obvious structure $(\infty,1)$-presheaf of $(\infty,1)$-ring $\mathcal{O}_{\mathrm{Spec}R}$ is actually an $(\infty,1)$-sheaf. 
\end{proposition}

\begin{corollary}\mbox{\textbf{(Bambozzi-Kremnizer \cite{BK})}}
Let $\mathrm{Mod}_\mathcal{O}$ be the $(\infty,1)$-category of all the quasicoherent $(\infty,1)$-presheaves of $\mathcal{O}$-modules. Then the finite projective objects with $\pi_0$ finite projective over $\pi_0\mathcal{O}$ in $\mathrm{Mod}_\mathcal{O}$ are indeed $(\infty,1)$-sheaves.
\end{corollary}

\begin{proof}
By applying the previous proposition.
\end{proof}

\newpage
\section{$(\infty,1)$-Analytic Stacks}

\indent The framework in \cite{BBBK} actually defined sufficiently general analytic stacks on the level of $(\infty,1)$-categories. Let $R$ be a general commutative Banach ring. From \cite{BBBK} we have the following result:

\begin{proposition}\mbox{\textbf{(Bambozzi-Ben-Bassat-Kremnizer)}}\\
\mbox{\textbf{(\cite[Definition 3.1, Theorem 3.14, Corollary 3.15, Remark 3.16]{BBBK})}}
Let $R$ be a Banach commutative algebra. Consider the following categories:
\begin{align}
\mathrm{SemiNormed}_R, \mathrm{Normed}_R, \mathrm{Banach}_R.
\end{align}
For each $C$ of these three categories we consider the inductive categories and monomorphic inductive categories associated to $C$, which will be denoted by:
\begin{align}
\mathrm{Ind}C,\mathrm{Ind}_\text{monomorphic}C.
\end{align}
Then all these categories can present certain $(\infty,1)$-categories. We then use the following notations to denote them:
\begin{align}
&D(\mathrm{Ind}\mathrm{SemiNormed}_R), D(\mathrm{Ind}\mathrm{Normed}_R), D(\mathrm{Ind}\mathrm{Banach}_R),\\
&D(\mathrm{Ind}_\text{monomorphic}\mathrm{SemiNormed}_R), D(\mathrm{Ind}_\text{monomorphic}\mathrm{Normed}_R), D(\mathrm{Ind}_\text{monomorphic}\mathrm{Banach}_R).
\end{align} 
Furthermore we have the following categories of simplicial commmutative rings:
\begin{align}
&\underset{D(\mathrm{Ind}\mathrm{SemiNormed}_R)}{\mathrm{CommutativeAlgebra}},\\ 
&\underset{D(\mathrm{Ind}\mathrm{Normed}_R)}{\mathrm{CommutativeAlgebra}},\\ 
&\underset{D(\mathrm{Ind}\mathrm{Banach}_R)}{\mathrm{CommutativeAlgebra}},\\
&\underset{D(\mathrm{Ind}_\text{monomorphic}\mathrm{SemiNormed}_R)}{\mathrm{CommutativeAlgebra}},\\ 
&\underset{D(\mathrm{Ind}_\text{monomorphic}\mathrm{Normed}_R)}{\mathrm{CommutativeAlgebra}},\\ 
&\underset{D(\mathrm{Ind}_\text{monomorphic}\mathrm{Banach}_R)}{\mathrm{CommutativeAlgebra}},
\end{align}
carrying the homotopical epimorphic Grothendieck $(\infty,1)$-topology.
\end{proposition}

\indent Over these sites we have from \cite[Definition 5.12]{BBBK} the definition of $(\infty,1)$-stacks, which are defined to be the $(\infty,1)$-fibered categories fibred in $\infty$-groupoids or certain sheaves valued in the corresponding in $\infty$-groupoids. We use the general notation $\mathrm{X}$ to denote such stack.

\chapter{$(\infty,1)$-Categorical Topologization}

\newpage
\section{$(\infty,1)$-Sheafiness}

\noindent \cite{CS1}, \cite{CS2}, \cite{CS3} defined the notation of the so-called analytic solid condensed rings. We use the notation $\mathrm{AnalyticSolidRings}$ to denote the $(\infty,1)$-category of all such rings. Also as in \cite{CS2} these are analytification of certain condensed solid rings in the $\infty$-categories of animation of condensed abelian groups $\mathrm{condensed}_\mathrm{abelian}$. We use the notation 
\begin{align}
{\underset{\mathrm{animation},\mathrm{condensed}_\mathrm{abelian}}{\mathrm{CommutativeRings}}}^{\mathrm{analytification}}
\end{align}
to denote this $(\infty,1)$-category. This $(\infty,1)$-category is stable under limits and colimits, carrying the corresponding solid tensor product $\otimes^\blacksquare$.

\indent We make the following parallel discussion in the condensed mathematics. The adic spaces in the sense of Huber as in \cite{Hu} essentially provide a framework for defining a specture in the anlytic situation associated to a Banach commutative ring $R$, but the sheafiness condition has to be required since the corresponding exact sequence in defining the sheafiness of the obvious structure ring structure might not be always holds for general $R$. But the work of \cite{CS2} provides certain derived structure ring structure which solves the issue. 

\begin{assumption}
Let $R$ be a general commutative Banach ring over a certain base $k$ as in \cite{BK}. For instance if $R/\mathbb{Q}_p$ is defined over $p$-adic number field $\mathbb{Q}_p$, then the discussion in the following will satisfy this requirement. The generality on the ring $R$ can be further generalized by applying the foundation in \cite{Ked1}.
\end{assumption}

\noindent The work of \cite{CS2} defines a site associated to $R$, carrying the topology by taking the corresponding derived analytic rational localization, i.e. the corresponding Koszul complexes in the pure algebraic sense:
\begin{align}
\mathrm{Koszul}_{f,g}(R):=R/^\mathrm{derived}\{g/f\}
\end{align}
with certain induction to define the following:
\begin{align}
\mathrm{Koszul}_{f,g_1,...,g_n}(R).
\end{align}

Then we have the corresponding stack $(\mathrm{Spec}R,\mathcal{O}_{\mathrm{Spec}R})$ over the site 
\begin{align}
{\underset{\mathrm{animation},\mathrm{condensed}_\mathrm{abelian}}{\mathrm{CommutativeRings}}}^{\mathrm{analytification}}
\end{align}
carrying Grothendieck topology by using the rational localization in the derived sense.

\begin{proposition} \mbox{\textbf{(Clausen-Scholze \cite[Proposition 12.18, Proposition 14.2, Proposition 14.7]{CS2})}}
Attached to $R$, we have a general $(\infty,1)$-fiber category/$(\infty,1)$-stack over the site 
\begin{align}
{\underset{\mathrm{animation},\mathrm{condensed}_\mathrm{abelian}}{\mathrm{CommutativeRings}}}^{\mathrm{analytification}}
\end{align}
The obvious structure $(\infty,1)$-presheaf of $(\infty,1)$-ring $\mathcal{O}_{\mathrm{Spec}R}$ is actually an $(\infty,1)$-sheaf. 
\end{proposition}

\begin{corollary} \mbox{\textbf{(Clausen-Scholze \cite[Remark 14.10]{CS2})}}
Let $\mathrm{Mod}_\mathcal{O}$ be the $(\infty,1)$-category of all the quasicoherent $(\infty,1)$-presheaves of $\mathcal{O}$-modules. Then the finite projective objects with $\pi_0$ finite projective over $\pi_0\mathcal{O}$ in $\mathrm{Mod}_\mathcal{O}$ are indeed $(\infty,1)$-sheaves.
\end{corollary}

\begin{proof}
By applying the previous proposition.
\end{proof}

\chapter{Derived Prismatic Cohomology for Commutative Algebras and Derived Preperfectoidization}

We now follow \cite{Grot1}, \cite{Grot2}, \cite{Grot3}, \cite{Grot4}, \cite{BK}, \cite{BBK}, \cite{BBBK}, \cite{BBM}, \cite{KKM}, \cite{T2}, \cite{Sch2}, \cite{BS}, \cite{BL}, \cite{Dr1}\footnote{One can consider the corresponding absolute prismatic complexes \cite{BS}, \cite{BL2}, \cite{BL}, \cite{Dr1} as well, though our presentation fix a corresponding base prism $(P,I)$ where $P/I$ is assumed to be Banach giving rise to the $p$-adic topology. And we assume the boundedness. } to revisit and discuss the corresponding derived prismatic cohomology for rings in the following:

\begin{notation}\mbox{\rm{(Rings)}}
Recall we have the following six categories on the commutative algebras in the derived sense (let $R$ be $P/I$):
\begin{align}
&\mathrm{sComm}\mathrm{Simplicial}\mathrm{Ind}\mathrm{Seminormed}_R,\\
&\mathrm{sComm}\mathrm{Simplicial}\mathrm{Ind}^m\mathrm{Seminormed}_R,\\
&\mathrm{sComm}\mathrm{Simplicial}\mathrm{Ind}\mathrm{Normed}_R,\\
&\mathrm{sComm}\mathrm{Simplicial}\mathrm{Ind}^m\mathrm{Normed}_R,\\
&\mathrm{sComm}\mathrm{Simplicial}\mathrm{Ind}\mathrm{Banach}_R,\\
&\mathrm{sComm}\mathrm{Simplicial}\mathrm{Ind}^m\mathrm{Banach}_R.	
\end{align}
	
\end{notation}

\begin{definition}
We now consider the rings:
\begin{align}
P/I\left<X_1,...,X_n\right>,n=0,1,2,...	
\end{align}
Then we take the corresponding homotopy colimit completion of these in the stable $\infty$-categories above:
\begin{align}
&\mathrm{sComm}\mathrm{Simplicial}\mathrm{Ind}\mathrm{Seminormed}_R,\\
&\mathrm{sComm}\mathrm{Simplicial}\mathrm{Ind}^m\mathrm{Seminormed}_R,\\
&\mathrm{sComm}\mathrm{Simplicial}\mathrm{Ind}\mathrm{Normed}_R,\\
&\mathrm{sComm}\mathrm{Simplicial}\mathrm{Ind}^m\mathrm{Normed}_R,\\
&\mathrm{sComm}\mathrm{Simplicial}\mathrm{Ind}\mathrm{Banach}_R,\\
&\mathrm{sComm}\mathrm{Simplicial}\mathrm{Ind}^m\mathrm{Banach}_R.	
\end{align}
The resulting $\infty$-categories will be denoted by:
\begin{align}
&\mathrm{sComm}\mathrm{Simplicial}\mathrm{Ind}\mathrm{Seminormed}^\mathrm{formalseriescolimitcomp}_R,\\
&\mathrm{sComm}\mathrm{Simplicial}\mathrm{Ind}^m\mathrm{Seminormed}^\mathrm{formalseriescolimitcomp}_R,\\
&\mathrm{sComm}\mathrm{Simplicial}\mathrm{Ind}\mathrm{Normed}^\mathrm{formalseriescolimitcomp}_R,\\
&\mathrm{sComm}\mathrm{Simplicial}\mathrm{Ind}^m\mathrm{Normed}^\mathrm{formalseriescolimitcomp}_R,\\
&\mathrm{sComm}\mathrm{Simplicial}\mathrm{Ind}\mathrm{Banach}^\mathrm{formalseriescolimitcomp}_R,\\
&\mathrm{sComm}\mathrm{Simplicial}\mathrm{Ind}^m\mathrm{Banach}^\mathrm{formalseriescolimitcomp}_R.	
\end{align}	
\end{definition}

\indent We then follow \cite{BS}, \cite{BL}, \cite{Dr1} to give the following definitions on the prismatic complexes $\Delta_{-/P}$ and the corresponding prismatic stacks as in \cite{BL}, which we will denote that by $\mathrm{CW}_{-/P}$.

\begin{definition}
Following \cite[Construction 7.6]{BS}, \cite[Definition 3.1, Variant 5.1]{BL} we give the following definition. For any ring
\begin{align}
\mathcal{R}=\underset{n}{\mathrm{homotopycolimit}}\mathcal{R}_n	
\end{align}
in the $\infty$-categories:
\begin{align}
&\mathrm{sComm}\mathrm{Simplicial}\mathrm{Ind}\mathrm{Seminormed}^\mathrm{formalseriescolimitcomp}_R,\\
&\mathrm{sComm}\mathrm{Simplicial}\mathrm{Ind}^m\mathrm{Seminormed}^\mathrm{formalseriescolimitcomp}_R,\\
&\mathrm{sComm}\mathrm{Simplicial}\mathrm{Ind}\mathrm{Normed}^\mathrm{formalseriescolimitcomp}_R,\\
&\mathrm{sComm}\mathrm{Simplicial}\mathrm{Ind}^m\mathrm{Normed}^\mathrm{formalseriescolimitcomp}_R,\\
&\mathrm{sComm}\mathrm{Simplicial}\mathrm{Ind}\mathrm{Banach}^\mathrm{formalseriescolimitcomp}_R,\\
&\mathrm{sComm}\mathrm{Simplicial}\mathrm{Ind}^m\mathrm{Banach}^\mathrm{formalseriescolimitcomp}_R,	
\end{align}	
we define the corresponding prismatic cohomology:
\begin{align}
\mathrm{Prism}_{-/P,\mathrm{BBM},\mathrm{analytification}}(\mathcal{R})
\end{align}
as:
\begin{align}
&\mathrm{Prism}_{-/P,\mathrm{BBM},\mathrm{analytification}}(\mathcal{R})\\
&:=[(\underset{n}{\mathrm{homotopycolimit}}~ \mathrm{Prism}_{-/P,,\mathrm{BBM},\mathrm{formalanalytification}}(\mathcal{R}_n))^\wedge_{p,I}]_{\mathrm{BBM},\mathrm{formalanalytification}}	
\end{align}
where the notation means we take the corresponding derived $(p,I)$-completion, then we take the corresponding formal series analytification from \cite[4.2]{BBM}.\\
For any ring
\begin{align}
\mathcal{R}=\underset{n}{\mathrm{homotopycolimit}}\mathcal{R}_n	
\end{align}
in the $\infty$-categories:
\begin{align}
&\mathrm{sComm}\mathrm{Simplicial}\mathrm{Ind}\mathrm{Seminormed}^\mathrm{formalseriescolimitcomp}_R,\\
&\mathrm{sComm}\mathrm{Simplicial}\mathrm{Ind}^m\mathrm{Seminormed}^\mathrm{formalseriescolimitcomp}_R,\\
&\mathrm{sComm}\mathrm{Simplicial}\mathrm{Ind}\mathrm{Normed}^\mathrm{formalseriescolimitcomp}_R,\\
&\mathrm{sComm}\mathrm{Simplicial}\mathrm{Ind}^m\mathrm{Normed}^\mathrm{formalseriescolimitcomp}_R,\\
&\mathrm{sComm}\mathrm{Simplicial}\mathrm{Ind}\mathrm{Banach}^\mathrm{formalseriescolimitcomp}_R,\\
&\mathrm{sComm}\mathrm{Simplicial}\mathrm{Ind}^m\mathrm{Banach}^\mathrm{formalseriescolimitcomp}_R,	
\end{align}	
we define the corresponding prismatic stack:
\begin{align}
\mathrm{CW}_{-/P}(\mathcal{R})
\end{align}
as:
\begin{align}
&\mathrm{CW}_{-/P}(\mathcal{R})\\
&:=[(\underset{n}{\mathrm{homotopycolimit}}~ \mathrm{CW}_{-/P}(\mathcal{R}_n).	
\end{align}
This as in \cite[Definition 3.1, Variant 5.1]{BL} carries the corresponding ringed topos structure 
\begin{align}
(\mathrm{CW}_{-/P}(\mathcal{R}),\mathcal{O}_{\mathrm{CW}_{-/P}(\mathcal{R})}). 
\end{align}
By \cite[Proposition 8.15]{BL} (also see \cite{Dr1}) we have that certain quasicoherent sheaves over this site will reflect completely the corresponding prismatic cohomological information. Therefore the resulting functor here $(\mathrm{CW}_{-/P},\mathcal{O}_{\mathrm{CW}_{-/P}})(-)$ will reflect the corresponding desired information for the functor $\mathrm{Prism}_{-/P}(-)$ as above.
\end{definition}

\indent Now we consider preperfectoidization constuctions:

\begin{definition}
Following \cite{BS}, we give the following definition. For any ring
\begin{align}
\mathcal{R}=\underset{n}{\mathrm{homotopycolimit}}\mathcal{R}_n	
\end{align}
in the $\infty$-categories:
\begin{align}
&\mathrm{sComm}\mathrm{Simplicial}\mathrm{Ind}\mathrm{Seminormed}^\mathrm{formalseriescolimitcomp}_R,\\
&\mathrm{sComm}\mathrm{Simplicial}\mathrm{Ind}^m\mathrm{Seminormed}^\mathrm{formalseriescolimitcomp}_R,\\
&\mathrm{sComm}\mathrm{Simplicial}\mathrm{Ind}\mathrm{Normed}^\mathrm{formalseriescolimitcomp}_R,\\
&\mathrm{sComm}\mathrm{Simplicial}\mathrm{Ind}^m\mathrm{Normed}^\mathrm{formalseriescolimitcomp}_R,\\
&\mathrm{sComm}\mathrm{Simplicial}\mathrm{Ind}\mathrm{Banach}^\mathrm{formalseriescolimitcomp}_R,\\
&\mathrm{sComm}\mathrm{Simplicial}\mathrm{Ind}^m\mathrm{Banach}^\mathrm{formalseriescolimitcomp}_R,	
\end{align}	
we define the corresponding prismatic cohomology:
\begin{align}
\mathrm{Prism}_{-/P,\mathrm{BBM},\mathrm{analytification}}(\mathcal{R})
\end{align}
as:
\begin{align}
&\mathrm{Prism}_{-/P,\mathrm{BBM},\mathrm{analytification}}(\mathcal{R})\\
&:=[(\underset{n}{\mathrm{homotopycolimit}}~ \mathrm{Prism}_{-/P,,\mathrm{BBM},\mathrm{formalanalytification}}(\mathcal{R}_n))^\wedge_{p,I}]_{\mathrm{BBM},\mathrm{formalanalytification}}	
\end{align}
where the notation means we take the corresponding derived $(p,I)$-completion, then we take the corresponding formal series analytification from \cite[4.2]{BBM}. Then as in \cite[Definition 8.2]{BS} we put the following preperfecdtoidization of any object $R$ to be:
\begin{align}
&R^\mathrm{preperfectoidization}:=\\
&\underset{i}{\mathrm{homotopycolimit}}(\mathrm{Prism}_{-/P,\mathrm{BBM},\mathrm{analytification}}(\mathcal{R})\rightarrow \mathrm{Fro}_*\mathrm{Prism}_{-/P,\mathrm{BBM},\mathrm{analytification}}(\mathcal{R})\\
&\rightarrow \mathrm{Fro}_*\mathrm{Fro}_*\mathrm{Prism}_{-/P,\mathrm{BBM},\mathrm{analytification}}(\mathcal{R})\rightarrow...)	
\end{align}
Then the perfectoidization of $R$ is just defined to be:
\begin{align}
R^\mathrm{preperfectoidization}\times P/I.	
\end{align}
For any ring
\begin{align}
\mathcal{R}=\underset{n}{\mathrm{homotopycolimit}}\mathcal{R}_n	
\end{align}
in the $\infty$-categories:
\begin{align}
&\mathrm{sComm}\mathrm{Simplicial}\mathrm{Ind}\mathrm{Seminormed}^\mathrm{formalseriescolimitcomp}_R,\\
&\mathrm{sComm}\mathrm{Simplicial}\mathrm{Ind}^m\mathrm{Seminormed}^\mathrm{formalseriescolimitcomp}_R,\\
&\mathrm{sComm}\mathrm{Simplicial}\mathrm{Ind}\mathrm{Normed}^\mathrm{formalseriescolimitcomp}_R,\\
&\mathrm{sComm}\mathrm{Simplicial}\mathrm{Ind}^m\mathrm{Normed}^\mathrm{formalseriescolimitcomp}_R,\\
&\mathrm{sComm}\mathrm{Simplicial}\mathrm{Ind}\mathrm{Banach}^\mathrm{formalseriescolimitcomp}_R,\\
&\mathrm{sComm}\mathrm{Simplicial}\mathrm{Ind}^m\mathrm{Banach}^\mathrm{formalseriescolimitcomp}_R,	
\end{align}	
we define the corresponding prismatic stack:
\begin{align}
\mathrm{CW}_{-/P}(\mathcal{R})
\end{align}
as:
\begin{align}
&\mathrm{CW}_{-/P}(\mathcal{R})\\
&:=[(\underset{n}{\mathrm{homotopycolimit}}~ \mathrm{CW}_{-/P}(\mathcal{R}_n).	
\end{align}
Then as in \cite[Definition 8.2]{BS} we put the following stacky preperfecdtoidization of any object $R$ to be:
\begin{align}
&R^\mathrm{stackypreperfectoidization}:=\\
&\underset{i}{\mathrm{homotopycolimit}}(\mathrm{CW}_{-/P}(\mathcal{R}))\rightarrow \mathrm{Fro}_*\mathrm{CW}_{-/P}(\mathcal{R})\\
&\rightarrow \mathrm{Fro}_*\mathrm{Fro}_*\mathrm{CW}_{-/P}(\mathcal{R})\rightarrow...)	
\end{align}
Then the perfectoidization of $R$ is just defined to be:
\begin{align}
R^\mathrm{stackypreperfectoidization}\times P/I.	
\end{align}

\end{definition}

\chapter{Derived Topological Hochschild Homology}

We now follow \cite{Grot1}, \cite{Grot2}, \cite{Grot3}, \cite{Grot4}, \cite{BK}, \cite{BBK}, \cite{BBBK}, \cite{BBM}, \cite{KKM}, \cite{T2}, \cite{Sch2}, \cite{BS}, \cite{BL}, \cite{Dr1}, \cite{NS}, \cite{BMS}, \cite{B}, \cite{BHM}\footnote{Our presentation fixes a corresponding base prism $(P,I)$ where $P/I$ is assumed to be Banach giving rise to the $p$-adic topology. And we assume the boundedness. } to revisit and discuss the corresponding derived topological Hochschild homology, topological period homology and topological cyclic homology for rings in the following:

\begin{notation}\mbox{\rm{(Rings)}}
Recall we have the following six categories on the noncommutative algebras in the derived sense (let $R$ be $P/I$):
\begin{align}
&\mathrm{sNoncomm}\mathrm{Simplicial}\mathrm{Ind}\mathrm{Seminormed}_R,\\
&\mathrm{sNoncomm}\mathrm{Simplicial}\mathrm{Ind}^m\mathrm{Seminormed}_R,\\
&\mathrm{sNoncomm}\mathrm{Simplicial}\mathrm{Ind}\mathrm{Normed}_R,\\
&\mathrm{sNoncomm}\mathrm{Simplicial}\mathrm{Ind}^m\mathrm{Normed}_R,\\
&\mathrm{sNoncomm}\mathrm{Simplicial}\mathrm{Ind}\mathrm{Banach}_R,\\
&\mathrm{sNoncomm}\mathrm{Simplicial}\mathrm{Ind}^m\mathrm{Banach}_R.	
\end{align}
	
\end{notation}

\begin{definition}
We now consider the rings\footnote{$Z_1,...,Z_n$ are just assumed to be free variables.}:
\begin{align}
P/I\left<Z_1,...,Z_n\right>,n=0,1,2,...	
\end{align}
Then we take the corresponding homotopy colimit completion of these in the stable $\infty$-categories above:
\begin{align}
&\mathrm{Noncomm}\mathrm{Simplicial}\mathrm{Ind}\mathrm{Seminormed}_R,\\
&\mathrm{Noncomm}\mathrm{Simplicial}\mathrm{Ind}^m\mathrm{Seminormed}_R,\\
&\mathrm{Noncomm}\mathrm{Simplicial}\mathrm{Ind}\mathrm{Normed}_R,\\
&\mathrm{Noncomm}\mathrm{Simplicial}\mathrm{Ind}^m\mathrm{Normed}_R,\\
&\mathrm{Noncomm}\mathrm{Simplicial}\mathrm{Ind}\mathrm{Banach}_R,\\
&\mathrm{Noncomm}\mathrm{Simplicial}\mathrm{Ind}^m\mathrm{Banach}_R.	
\end{align}
The resulting $\infty$-categories will be denoted by:
\begin{align}
&\mathrm{sNoncomm}\mathrm{Simplicial}\mathrm{Ind}\mathrm{Seminormed}^\mathrm{formalseriescolimitcomp}_R,\\
&\mathrm{sNoncomm}\mathrm{Simplicial}\mathrm{Ind}^m\mathrm{Seminormed}^\mathrm{formalseriescolimitcomp}_R,\\
&\mathrm{sNoncomm}\mathrm{Simplicial}\mathrm{Ind}\mathrm{Normed}^\mathrm{formalseriescolimitcomp}_R,\\
&\mathrm{sNoncomm}\mathrm{Simplicial}\mathrm{Ind}^m\mathrm{Normed}^\mathrm{formalseriescolimitcomp}_R,\\
&\mathrm{sNoncomm}\mathrm{Simplicial}\mathrm{Ind}\mathrm{Banach}^\mathrm{formalseriescolimitcomp}_R,\\
&\mathrm{sNoncomm}\mathrm{Simplicial}\mathrm{Ind}^m\mathrm{Banach}^\mathrm{formalseriescolimitcomp}_R.	
\end{align}	
\end{definition}
\

\indent We then follow \cite[Section 2.3]{BMS}, \cite[Chapter 3]{NS} to give the following definitions on the topological Hochschild complexes, topological period complexes and topological cyclic complexes
\begin{align}
 \mathrm{THH}_{-/P,\mathrm{BBM},\mathrm{analytification}},\\
 \mathrm{TP}_{-/P,\mathrm{BBM},\mathrm{analytification}},\\
 \mathrm{TC}_{-/P,\mathrm{BBM},\mathrm{analytification}}. 
\end{align}
 All the constructions are directly applications of functors in \cite[Section 2.3]{BMS}, \cite[Chapter 3]{NS}.

\begin{definition}
Following \cite[Section 2.3]{BMS} and \cite[Chapter 3]{NS} we give the following definition. For any ring
\begin{align}
\mathcal{R}=\underset{n}{\mathrm{homotopycolimit}}\mathcal{R}_n	
\end{align}
in the $\infty$-categories:
\begin{align}
&\mathrm{sNoncomm}\mathrm{Simplicial}\mathrm{Ind}\mathrm{Seminormed}^\mathrm{formalseriescolimitcomp}_R,\\
&\mathrm{sNoncomm}\mathrm{Simplicial}\mathrm{Ind}^m\mathrm{Seminormed}^\mathrm{formalseriescolimitcomp}_R,\\
&\mathrm{sNoncomm}\mathrm{Simplicial}\mathrm{Ind}\mathrm{Normed}^\mathrm{formalseriescolimitcomp}_R,\\
&\mathrm{sNoncomm}\mathrm{Simplicial}\mathrm{Ind}^m\mathrm{Normed}^\mathrm{formalseriescolimitcomp}_R,\\
&\mathrm{sNoncomm}\mathrm{Simplicial}\mathrm{Ind}\mathrm{Banach}^\mathrm{formalseriescolimitcomp}_R,\\
&\mathrm{sNoncomm}\mathrm{Simplicial}\mathrm{Ind}^m\mathrm{Banach}^\mathrm{formalseriescolimitcomp}_R,	
\end{align}	
we define the corresponding topological Hochschild complexes, topological period complexes and topological cyclic complexes $\mathrm{THH}_{-/P}$, $\mathrm{TP}_{-/P}$, $\mathrm{TC}_{-/P}$:
\begin{align}
 \mathrm{THH}_{-/P,\mathrm{BBM},\mathrm{analytification}}(\mathcal{R}),\\
 \mathrm{TP}_{-/P,\mathrm{BBM},\mathrm{analytification}}(\mathcal{R}),\\
 \mathrm{TC}_{-/P,\mathrm{BBM},\mathrm{analytification}}(\mathcal{R}). 
\end{align}
as:
\begin{align}
& \mathrm{THH}_{-/P,\mathrm{BBM},\mathrm{analytification}}(\mathcal{R})\\
&:=[(\underset{n}{\mathrm{homotopycolimit}}~  \mathrm{THH}_{-/P,\mathrm{BBM},\mathrm{analytification}}(\mathcal{R}_n))^\wedge_{p}]_{\mathrm{BBM},\mathrm{formalanalytification}}\\
& \mathrm{TP}_{-/P,\mathrm{BBM},\mathrm{analytification}}(\mathcal{R})\\
&:=[(\underset{n}{\mathrm{homotopycolimit}}~  \mathrm{TP}_{-/P,\mathrm{BBM},\mathrm{analytification}}(\mathcal{R}_n))^\wedge_{p}]_{\mathrm{BBM},\mathrm{formalanalytification}}\\
& \mathrm{TC}_{-/P,\mathrm{BBM},\mathrm{analytification}}(\mathcal{R})\\
&:=[(\underset{n}{\mathrm{homotopycolimit}}~  \mathrm{TC}_{-/P,\mathrm{BBM},\mathrm{analytification}}(\mathcal{R}_n))^\wedge_{p}]_{\mathrm{BBM},\mathrm{formalanalytification}}\\	
\end{align}
where the notation means we take the corresponding algebraic topological $p$-completion, then we take the corresponding formal series analytification from \cite[4.2]{BBM} in the corresponding analogy of the commutative situation.\\
\end{definition}

\indent Now we consider preperfectoidization constuctions:

\begin{definition}
Following \cite{BS}, we give the following definition. For any ring
\begin{align}
\mathcal{R}=\underset{n}{\mathrm{homotopycolimit}}\mathcal{R}_n	
\end{align}
in the $\infty$-categories:
\begin{align}
&\mathrm{sNoncomm}\mathrm{Simplicial}\mathrm{Ind}\mathrm{Seminormed}^\mathrm{formalseriescolimitcomp}_R,\\
&\mathrm{sNoncomm}\mathrm{Simplicial}\mathrm{Ind}^m\mathrm{Seminormed}^\mathrm{formalseriescolimitcomp}_R,\\
&\mathrm{sNoncomm}\mathrm{Simplicial}\mathrm{Ind}\mathrm{Normed}^\mathrm{formalseriescolimitcomp}_R,\\
&\mathrm{sNoncomm}\mathrm{Simplicial}\mathrm{Ind}^m\mathrm{Normed}^\mathrm{formalseriescolimitcomp}_R,\\
&\mathrm{sNoncomm}\mathrm{Simplicial}\mathrm{Ind}\mathrm{Banach}^\mathrm{formalseriescolimitcomp}_R,\\
&\mathrm{sNoncomm}\mathrm{Simplicial}\mathrm{Ind}^m\mathrm{Banach}^\mathrm{formalseriescolimitcomp}_R,	
\end{align}	
we define the corresponding topological Hochschild complexes, topological period complexes and topological cyclic complexes $\mathrm{THH}_{-/P}$, $\mathrm{TP}_{-/P}$, $\mathrm{TC}_{-/P}$:
\begin{align}
 \mathrm{THH}_{-/P,\mathrm{BBM},\mathrm{analytification}}(\mathcal{R}),\\
 \mathrm{TP}_{-/P,\mathrm{BBM},\mathrm{analytification}}(\mathcal{R}),\\
 \mathrm{TC}_{-/P,\mathrm{BBM},\mathrm{analytification}}(\mathcal{R}). 
\end{align}
as:
\begin{align}
& \mathrm{THH}_{-/P,\mathrm{BBM},\mathrm{analytification}}(\mathcal{R})\\
&:=[(\underset{n}{\mathrm{homotopycolimit}}~  \mathrm{THH}_{-/P,\mathrm{BBM},\mathrm{analytification}}(\mathcal{R}_n))^\wedge_{p}]_{\mathrm{BBM},\mathrm{formalanalytification}}\\
& \mathrm{TP}_{-/P,\mathrm{BBM},\mathrm{analytification}}(\mathcal{R})\\
&:=[(\underset{n}{\mathrm{homotopycolimit}}~  \mathrm{TP}_{-/P,\mathrm{BBM},\mathrm{analytification}}(\mathcal{R}_n))^\wedge_{p}]_{\mathrm{BBM},\mathrm{formalanalytification}}\\
& \mathrm{TC}_{-/P,\mathrm{BBM},\mathrm{analytification}}(\mathcal{R})\\
&:=[(\underset{n}{\mathrm{homotopycolimit}}~  \mathrm{TC}_{-/P,\mathrm{BBM},\mathrm{analytification}}(\mathcal{R}_n))^\wedge_{p}]_{\mathrm{BBM},\mathrm{formalanalytification}}\\	
\end{align}
where the notation means we take the corresponding algebraic topological $p$-completion, then we take the corresponding formal series analytification from \cite[4.2]{BBM} in the corresponding analogy of the commutative situation. Then as in \cite[Definition 8.2]{BS} we put the following preperfectoidizations of any object $R$ to be:
\begin{align}
&R^\mathrm{preperfectoidization,THH}:=\\
&\underset{i}{\mathrm{homotopycolimit}}(\mathrm{THH}_{-/P,\mathrm{BBM},\mathrm{analytification}}(\mathcal{R})\rightarrow \mathrm{Fro}_*\mathrm{THH}_{-/P,\mathrm{BBM},\mathrm{analytification}}(\mathcal{R})\\
&\rightarrow \mathrm{Fro}_*\mathrm{Fro}_*\mathrm{THH}_{-/P,\mathrm{BBM},\mathrm{analytification}}(\mathcal{R})\rightarrow...)	\\
&R^\mathrm{preperfectoidization,TP}:=\\
&\underset{i}{\mathrm{homotopycolimit}}(\mathrm{TP}_{-/P,\mathrm{BBM},\mathrm{analytification}}(\mathcal{R})\rightarrow \mathrm{Fro}_*\mathrm{TP}_{-/P,\mathrm{BBM},\mathrm{analytification}}(\mathcal{R})\\
&\rightarrow \mathrm{Fro}_*\mathrm{Fro}_*\mathrm{TP}_{-/P,\mathrm{BBM},\mathrm{analytification}}(\mathcal{R})\rightarrow...)	\\
&R^\mathrm{preperfectoidization,TC}:=\\
&\underset{i}{\mathrm{homotopycolimit}}(\mathrm{TC}_{-/P,\mathrm{BBM},\mathrm{analytification}}(\mathcal{R})\rightarrow \mathrm{Fro}_*\mathrm{TC}_{-/P,\mathrm{BBM},\mathrm{analytification}}(\mathcal{R})\\
&\rightarrow \mathrm{Fro}_*\mathrm{Fro}_*\mathrm{TC}_{-/P,\mathrm{BBM},\mathrm{analytification}}(\mathcal{R})\rightarrow...)	\\
\end{align}
Then the perfectoidizations of $R$ are just defined to be:
\begin{align}
R^\mathrm{preperfectoidization,\sharp}\times P/I.	
\end{align}
Here $\sharp$ represents one of $\mathrm{THH},\mathrm{TP},\mathrm{TC}$.
\end{definition}

\chapter{Derived Prismatic Cohomology for Ringed Toposes}

We now follow \cite{Grot1}, \cite{Grot2}, \cite{Grot3}, \cite{Grot4}, \cite{BK}, \cite{BBK}, \cite{BBBK}, \cite{BBM}, \cite{KKM}, \cite{T2}, \cite{Sch2}, \cite{BS}, \cite{BL}, \cite{Dr1}\footnote{One can consider the corresponding absolute prismatic complexes \cite{BS}, \cite{BL2}, \cite{BL}, \cite{Dr1} as well, though our presentation fix a corresponding base prism $(P,I)$ where $P/I$ is assumed to be Banach giving rise to the $p$-adic topology. And we assume the boundedness.} to revisit and discuss the corresponding derived prismatic cohomology for rings in the following. We first consider the following generating ringed spaces from \cite{BK}:
\begin{align}
(\mathrm{Spa}^\mathrm{BK}P/I\left<X_1,...,X_n\right>,\mathcal{R}_{\mathrm{Spa}^\mathrm{BK}P/I\left<X_1,...,X_n\right>}),n=0,1,2,...
\end{align}
\begin{definition}
We now consider the homotopy limit completion of 
\begin{align}
(\mathrm{Spa}^\mathrm{BK}P/I\left<X_1,...,X_n\right>,\mathcal{R}_{\mathrm{Spa}^\mathrm{BK}P/I\left<X_1,...,X_n\right>}),n=0,1,2,...
\end{align}
in the following $\infty$-categories:
\begin{align}
&\mathrm{Sta}^\mathrm{derivedringed,\sharp}_{\mathrm{sComm}\mathrm{Simplicial}\mathrm{Ind}\mathrm{Seminormed}_R,\mathrm{homotopyepi}},\\
&\mathrm{Sta}^\mathrm{derivedringed,\sharp}_{\mathrm{sComm}\mathrm{Simplicial}\mathrm{Ind}^m\mathrm{Seminormed}_R,\mathrm{homotopyepi}},\\
&\mathrm{Sta}^\mathrm{derivedringed,\sharp}_{\mathrm{sComm}\mathrm{Simplicial}\mathrm{Ind}\mathrm{Normed}_R,\mathrm{homotopyepi}},\\
&\mathrm{Sta}^\mathrm{derivedringed,\sharp}_{\mathrm{sComm}\mathrm{Simplicial}\mathrm{Ind}^m\mathrm{Normed}_R,\mathrm{homotopyepi}},\\
&\mathrm{Sta}^\mathrm{derivedringed,\sharp}_{\mathrm{sComm}\mathrm{Simplicial}\mathrm{Ind}\mathrm{Banach}_R,\mathrm{homotopyepi}},\\
&\mathrm{Sta}^\mathrm{derivedringed,\sharp}_{\mathrm{sComm}\mathrm{Simplicial}\mathrm{Ind}^m\mathrm{Banach}_R,\mathrm{homotopyepi}}.	
\end{align}
Here $\sharp$ represents any category in the following:
\begin{align}
&\mathrm{sComm}\mathrm{Simplicial}\mathrm{Ind}\mathrm{Seminormed}_R,\\
&\mathrm{sComm}\mathrm{Simplicial}\mathrm{Ind}^m\mathrm{Seminormed}_R,\\
&\mathrm{sComm}\mathrm{Simplicial}\mathrm{Ind}\mathrm{Normed}_R,\\
&\mathrm{sComm}\mathrm{Simplicial}\mathrm{Ind}^m\mathrm{Normed}_R,\\
&\mathrm{sComm}\mathrm{Simplicial}\mathrm{Ind}\mathrm{Banach}_R,\\
&\mathrm{sComm}\mathrm{Simplicial}\mathrm{Ind}^m\mathrm{Banach}_R.	
\end{align}	
Here $R=P/I$. The resulting sub $\infty$-categories are denoted by:
\begin{align}
&\mathrm{Proj}^\mathrm{formalspectrum}\mathrm{Sta}^\mathrm{derivedringed,\sharp}_{\mathrm{sComm}\mathrm{Simplicial}\mathrm{Ind}\mathrm{Seminormed}_R,\mathrm{homotopyepi}},\\
&\mathrm{Proj}^\mathrm{formalspectrum}\mathrm{Sta}^\mathrm{derivedringed,\sharp}_{\mathrm{sComm}\mathrm{Simplicial}\mathrm{Ind}^m\mathrm{Seminormed}_R,\mathrm{homotopyepi}},\\
&\mathrm{Proj}^\mathrm{formalspectrum}\mathrm{Sta}^\mathrm{derivedringed,\sharp}_{\mathrm{sComm}\mathrm{Simplicial}\mathrm{Ind}\mathrm{Normed}_R,\mathrm{homotopyepi}},\\
&\mathrm{Proj}^\mathrm{formalspectrum}\mathrm{Sta}^\mathrm{derivedringed,\sharp}_{\mathrm{sComm}\mathrm{Simplicial}\mathrm{Ind}^m\mathrm{Normed}_R,\mathrm{homotopyepi}},\\
&\mathrm{Proj}^\mathrm{formalspectrum}\mathrm{Sta}^\mathrm{derivedringed,\sharp}_{\mathrm{sComm}\mathrm{Simplicial}\mathrm{Ind}\mathrm{Banach}_R,\mathrm{homotopyepi}},\\
&\mathrm{Proj}^\mathrm{formalspectrum}\mathrm{Sta}^\mathrm{derivedringed,\sharp}_{\mathrm{sComm}\mathrm{Simplicial}\mathrm{Ind}^m\mathrm{Banach}_R,\mathrm{homotopyepi}}.	
\end{align}
Here $\sharp$ represents any category in the following:
\begin{align}
&\mathrm{sComm}\mathrm{Simplicial}\mathrm{Ind}\mathrm{Seminormed}_R,\\
&\mathrm{sComm}\mathrm{Simplicial}\mathrm{Ind}^m\mathrm{Seminormed}_R,\\
&\mathrm{sComm}\mathrm{Simplicial}\mathrm{Ind}\mathrm{Normed}_R,\\
&\mathrm{sComm}\mathrm{Simplicial}\mathrm{Ind}^m\mathrm{Normed}_R,\\
&\mathrm{sComm}\mathrm{Simplicial}\mathrm{Ind}\mathrm{Banach}_R,\\
&\mathrm{sComm}\mathrm{Simplicial}\mathrm{Ind}^m\mathrm{Banach}_R.	
\end{align}	
This means that any space $(\mathbb{X},\mathcal{R})$ in the full $\infty$-categories could be written as the following:
\begin{align}
(\mathbb{X},\mathcal{R})=\underset{n}{\mathrm{homotopylimit}}(\mathbb{X}_n,\mathcal{R}_n)	
\end{align}
where we have then:
\begin{align}
\mathcal{R}=\underset{n}{\mathrm{homotopycolimit}}\mathcal{R}_n	
\end{align}
as coherent sheaves over each $\mathbb{X}_n$.

\end{definition}

\indent We then follow \cite{BS}, \cite{BL}, \cite{Dr1} to give the following definitions on the prismatic cohomology presheaf $\Delta_{-/P}$ and the corresponding prismatic stack presheaf as in \cite{BL}, which we will denote that by $\mathrm{CW}_{-/P}$.

\begin{definition}
Following \cite[Construction 7.6]{BS}, \cite[Definition 3.1, Variant 5.1]{BL} we give the following definition. For any space
\begin{align}
(\mathbb{X},\mathcal{R})=\underset{n}{\mathrm{homotopylimit}}(\mathbb{X}_n,\mathcal{R}_n)	
\end{align}
in the $\infty$-categories:
\begin{align}
&\mathrm{Proj}^\mathrm{formalspectrum}\mathrm{Sta}^\mathrm{derivedringed,\sharp}_{\mathrm{sComm}\mathrm{Simplicial}\mathrm{Ind}\mathrm{Seminormed}_R,\mathrm{homotopyepi}},\\
&\mathrm{Proj}^\mathrm{formalspectrum}\mathrm{Sta}^\mathrm{derivedringed,\sharp}_{\mathrm{sComm}\mathrm{Simplicial}\mathrm{Ind}^m\mathrm{Seminormed}_R,\mathrm{homotopyepi}},\\
&\mathrm{Proj}^\mathrm{formalspectrum}\mathrm{Sta}^\mathrm{derivedringed,\sharp}_{\mathrm{sComm}\mathrm{Simplicial}\mathrm{Ind}\mathrm{Normed}_R,\mathrm{homotopyepi}},\\
&\mathrm{Proj}^\mathrm{formalspectrum}\mathrm{Sta}^\mathrm{derivedringed,\sharp}_{\mathrm{sComm}\mathrm{Simplicial}\mathrm{Ind}^m\mathrm{Normed}_R,\mathrm{homotopyepi}},\\
&\mathrm{Proj}^\mathrm{formalspectrum}\mathrm{Sta}^\mathrm{derivedringed,\sharp}_{\mathrm{sComm}\mathrm{Simplicial}\mathrm{Ind}\mathrm{Banach}_R,\mathrm{homotopyepi}},\\
&\mathrm{Proj}^\mathrm{formalspectrum}\mathrm{Sta}^\mathrm{derivedringed,\sharp}_{\mathrm{sComm}\mathrm{Simplicial}\mathrm{Ind}^m\mathrm{Banach}_R,\mathrm{homotopyepi}},	
\end{align}
we define the corresponding prismatic cohomology presheaf:
\begin{align}
\mathrm{Prism}_{-/P,\mathrm{BBM},\mathrm{analytification}}(\mathcal{R})
\end{align}
as:
\begin{align}
&\mathrm{Prism}_{-/P,\mathrm{BBM},\mathrm{analytification}}(\mathcal{R})\\
&:=[(\underset{n}{\mathrm{homotopycolimit}}~ \mathrm{Prism}_{-/P,,\mathrm{BBM},\mathrm{formalanalytification}}(\mathcal{R}_n))^\wedge_{p,I}]_{\mathrm{BBM},\mathrm{formalanalytification}}	
\end{align}
where the notation means we take the corresponding derived $(p,I)$-completion, then we take the corresponding formal series analytification from \cite[4.2]{BBM}.\\
For any space
\begin{align}
(\mathbb{X},\mathcal{R})=\underset{n}{\mathrm{homotopylimit}}(\mathbb{X}_n,\mathcal{R}_n)	
\end{align}
in the $\infty$-categories:
\begin{align}
&\mathrm{Proj}^\mathrm{formalspectrum}\mathrm{Sta}^\mathrm{derivedringed,\sharp}_{\mathrm{sComm}\mathrm{Simplicial}\mathrm{Ind}\mathrm{Seminormed}_R,\mathrm{homotopyepi}},\\
&\mathrm{Proj}^\mathrm{formalspectrum}\mathrm{Sta}^\mathrm{derivedringed,\sharp}_{\mathrm{sComm}\mathrm{Simplicial}\mathrm{Ind}^m\mathrm{Seminormed}_R,\mathrm{homotopyepi}},\\
&\mathrm{Proj}^\mathrm{formalspectrum}\mathrm{Sta}^\mathrm{derivedringed,\sharp}_{\mathrm{sComm}\mathrm{Simplicial}\mathrm{Ind}\mathrm{Normed}_R,\mathrm{homotopyepi}},\\
&\mathrm{Proj}^\mathrm{formalspectrum}\mathrm{Sta}^\mathrm{derivedringed,\sharp}_{\mathrm{sComm}\mathrm{Simplicial}\mathrm{Ind}^m\mathrm{Normed}_R,\mathrm{homotopyepi}},\\
&\mathrm{Proj}^\mathrm{formalspectrum}\mathrm{Sta}^\mathrm{derivedringed,\sharp}_{\mathrm{sComm}\mathrm{Simplicial}\mathrm{Ind}\mathrm{Banach}_R,\mathrm{homotopyepi}},\\
&\mathrm{Proj}^\mathrm{formalspectrum}\mathrm{Sta}^\mathrm{derivedringed,\sharp}_{\mathrm{sComm}\mathrm{Simplicial}\mathrm{Ind}^m\mathrm{Banach}_R,\mathrm{homotopyepi}},	
\end{align}
we define the corresponding prismatic stack presheaf (with stack values):
\begin{align}
\mathrm{CW}_{-/P}(\mathcal{R})
\end{align}
as:
\begin{align}
&\mathrm{CW}_{-/P}(\mathcal{R})\\
&:=[(\underset{n}{\mathrm{homotopycolimit}}~ \mathrm{CW}_{-/P}(\mathcal{R}_n).	
\end{align}
This as in \cite[Definition 3.1, Variant 5.1]{BL} carries the corresponding ringed topos structure 
\begin{align}
(\mathrm{CW}_{-/P}(\mathcal{R}),\mathcal{O}_{\mathrm{CW}_{-/P}(\mathcal{R})}). 
\end{align}
By \cite[Proposition 8.15]{BL} (also see \cite{Dr1}) we have that certain quasicoherent sheaves over this site will reflect completely the corresponding prismatic cohomological information. Therefore the resulting functor here $(\mathrm{CW}_{-/P},\mathcal{O}_{\mathrm{CW}_{-/P}})(-)$ will reflect the corresponding desired information for the functor $\mathrm{Prism}_{-/P}(-)$ as above.
\end{definition}

\chapter{Derived Prismatic Cohomology for Inductive Systems}

We now follow \cite{Grot1}, \cite{Grot2}, \cite{Grot3}, \cite{Grot4}, \cite{BK}, \cite{BBK}, \cite{BBBK}, \cite{BBM}, \cite{KKM}, \cite{T2}, \cite{Sch2}, \cite{BS}, \cite{BL}, \cite{Dr1}\footnote{One can consider the corresponding absolute prismatic complexes \cite{BS}, \cite{BL2}, \cite{BL}, \cite{Dr1} as well, though our presentation fix a corresponding base prism $(P,I)$ where $P/I$ is assumed to be Banach giving rise to the $p$-adic topology. And we assume the boundedness.  } to revisit and discuss the corresponding derived prismatic cohomology for rings in the following. We first consider the following generating ringed spaces from \cite{BK}:
\begin{align}
(\mathrm{Spa}^\mathrm{BK}P/I\left<X_1,...,X_n\right>,\mathcal{R}_{\mathrm{Spa}^\mathrm{BK}P/I\left<X_1,...,X_n\right>}),n=0,1,2,...
\end{align}
\begin{definition}
We now consider the homotopy colimit completion of 
\begin{align}
(\mathrm{Spa}^\mathrm{BK}P/I\left<X_1,...,X_n\right>,\mathcal{R}_{\mathrm{Spa}^\mathrm{BK}P/I\left<X_1,...,X_n\right>}),n=0,1,2,...
\end{align}
in the following $\infty$-categories:
\begin{align}
&\mathrm{Sta}^\mathrm{derivedringed,\sharp}_{\mathrm{sComm}\mathrm{Simplicial}\mathrm{Ind}\mathrm{Seminormed}_R,\mathrm{homotopyepi}},\\
&\mathrm{Sta}^\mathrm{derivedringed,\sharp}_{\mathrm{sComm}\mathrm{Simplicial}\mathrm{Ind}^m\mathrm{Seminormed}_R,\mathrm{homotopyepi}},\\
&\mathrm{Sta}^\mathrm{derivedringed,\sharp}_{\mathrm{sComm}\mathrm{Simplicial}\mathrm{Ind}\mathrm{Normed}_R,\mathrm{homotopyepi}},\\
&\mathrm{Sta}^\mathrm{derivedringed,\sharp}_{\mathrm{sComm}\mathrm{Simplicial}\mathrm{Ind}^m\mathrm{Normed}_R,\mathrm{homotopyepi}},\\
&\mathrm{Sta}^\mathrm{derivedringed,\sharp}_{\mathrm{sComm}\mathrm{Simplicial}\mathrm{Ind}\mathrm{Banach}_R,\mathrm{homotopyepi}},\\
&\mathrm{Sta}^\mathrm{derivedringed,\sharp}_{\mathrm{sComm}\mathrm{Simplicial}\mathrm{Ind}^m\mathrm{Banach}_R,\mathrm{homotopyepi}}.	
\end{align}
Here $\sharp$ represents any category in the following:
\begin{align}
&\mathrm{sComm}\mathrm{Simplicial}\mathrm{Ind}\mathrm{Seminormed}_R,\\
&\mathrm{sComm}\mathrm{Simplicial}\mathrm{Ind}^m\mathrm{Seminormed}_R,\\
&\mathrm{sComm}\mathrm{Simplicial}\mathrm{Ind}\mathrm{Normed}_R,\\
&\mathrm{sComm}\mathrm{Simplicial}\mathrm{Ind}^m\mathrm{Normed}_R,\\
&\mathrm{sComm}\mathrm{Simplicial}\mathrm{Ind}\mathrm{Banach}_R,\\
&\mathrm{sComm}\mathrm{Simplicial}\mathrm{Ind}^m\mathrm{Banach}_R.	
\end{align}	
Here $R=P/I$. The resulting sub $\infty$-categories are denoted by:
\begin{align}
&\mathrm{Ind}^\mathrm{formalspectrum}\mathrm{Sta}^\mathrm{derivedringed,\sharp}_{\mathrm{sComm}\mathrm{Simplicial}\mathrm{Ind}\mathrm{Seminormed}_R,\mathrm{homotopyepi}},\\
&\mathrm{Ind}^\mathrm{formalspectrum}\mathrm{Sta}^\mathrm{derivedringed,\sharp}_{\mathrm{sComm}\mathrm{Simplicial}\mathrm{Ind}^m\mathrm{Seminormed}_R,\mathrm{homotopyepi}},\\
&\mathrm{Ind}^\mathrm{formalspectrum}\mathrm{Sta}^\mathrm{derivedringed,\sharp}_{\mathrm{sComm}\mathrm{Simplicial}\mathrm{Ind}\mathrm{Normed}_R,\mathrm{homotopyepi}},\\
&\mathrm{Ind}^\mathrm{formalspectrum}\mathrm{Sta}^\mathrm{derivedringed,\sharp}_{\mathrm{sComm}\mathrm{Simplicial}\mathrm{Ind}^m\mathrm{Normed}_R,\mathrm{homotopyepi}},\\
&\mathrm{Ind}^\mathrm{formalspectrum}\mathrm{Sta}^\mathrm{derivedringed,\sharp}_{\mathrm{sComm}\mathrm{Simplicial}\mathrm{Ind}\mathrm{Banach}_R,\mathrm{homotopyepi}},\\
&\mathrm{Ind}^\mathrm{formalspectrum}\mathrm{Sta}^\mathrm{derivedringed,\sharp}_{\mathrm{sComm}\mathrm{Simplicial}\mathrm{Ind}^m\mathrm{Banach}_R,\mathrm{homotopyepi}}.	
\end{align}
Here $\sharp$ represents any category in the following:
\begin{align}
&\mathrm{sComm}\mathrm{Simplicial}\mathrm{Ind}\mathrm{Seminormed}_R,\\
&\mathrm{sComm}\mathrm{Simplicial}\mathrm{Ind}^m\mathrm{Seminormed}_R,\\
&\mathrm{sComm}\mathrm{Simplicial}\mathrm{Ind}\mathrm{Normed}_R,\\
&\mathrm{sComm}\mathrm{Simplicial}\mathrm{Ind}^m\mathrm{Normed}_R,\\
&\mathrm{sComm}\mathrm{Simplicial}\mathrm{Ind}\mathrm{Banach}_R,\\
&\mathrm{sComm}\mathrm{Simplicial}\mathrm{Ind}^m\mathrm{Banach}_R.	
\end{align}	
This means that any space $(\mathbb{X},\mathcal{R})$ in the full $\infty$-categories could be written as the following:
\begin{align}
(\mathbb{X},\mathcal{R})=\underset{n}{\mathrm{homotopycolimit}}(\mathbb{X}_n,\mathcal{R}_n)	
\end{align}
where we have then:
\begin{align}
\mathcal{R}=\underset{n}{\mathrm{homotopylimit}}\mathcal{R}_n	
\end{align}
as coherent sheaves over each $\mathbb{X}_n$.

\end{definition}

\indent We then follow \cite{BS}, \cite{BL}, \cite{Dr1} to give the following definitions on the prismatic cohomology presheaf $\Delta_{-/P}$ and the corresponding prismatic stack presheaf as in \cite{BL}, which we will denote that by $\mathrm{CW}_{-/P}$.

\begin{definition}
Following \cite[Construction 7.6]{BS}, \cite[Definition 3.1, Variant 5.1]{BL} we give the following definition. For any space
\begin{align}
(\mathbb{X},\mathcal{R})=\underset{n}{\mathrm{homotopycolimit}}(\mathbb{X}_n,\mathcal{R}_n)	
\end{align}
in the $\infty$-categories:
\begin{align}
&\mathrm{Ind}^\mathrm{formalspectrum}\mathrm{Sta}^\mathrm{derivedringed,\sharp}_{\mathrm{sComm}\mathrm{Simplicial}\mathrm{Ind}\mathrm{Seminormed}_R,\mathrm{homotopyepi}},\\
&\mathrm{Ind}^\mathrm{formalspectrum}\mathrm{Sta}^\mathrm{derivedringed,\sharp}_{\mathrm{sComm}\mathrm{Simplicial}\mathrm{Ind}^m\mathrm{Seminormed}_R,\mathrm{homotopyepi}},\\
&\mathrm{Ind}^\mathrm{formalspectrum}\mathrm{Sta}^\mathrm{derivedringed,\sharp}_{\mathrm{sComm}\mathrm{Simplicial}\mathrm{Ind}\mathrm{Normed}_R,\mathrm{homotopyepi}},\\
&\mathrm{Ind}^\mathrm{formalspectrum}\mathrm{Sta}^\mathrm{derivedringed,\sharp}_{\mathrm{sComm}\mathrm{Simplicial}\mathrm{Ind}^m\mathrm{Normed}_R,\mathrm{homotopyepi}},\\
&\mathrm{Ind}^\mathrm{formalspectrum}\mathrm{Sta}^\mathrm{derivedringed,\sharp}_{\mathrm{sComm}\mathrm{Simplicial}\mathrm{Ind}\mathrm{Banach}_R,\mathrm{homotopyepi}},\\
&\mathrm{Ind}^\mathrm{formalspectrum}\mathrm{Sta}^\mathrm{derivedringed,\sharp}_{\mathrm{sComm}\mathrm{Simplicial}\mathrm{Ind}^m\mathrm{Banach}_R,\mathrm{homotopyepi}},	
\end{align}
we define the corresponding prismatic cohomology presheaf:
\begin{align}
\mathrm{Prism}_{-/P,\mathrm{BBM},\mathrm{analytification}}(\mathcal{R})
\end{align}
as:
\begin{align}
&\mathrm{Prism}_{-/P,\mathrm{BBM},\mathrm{analytification}}(\mathcal{R})\\
&:=[(\underset{n}{\mathrm{homotopylimit}}~ \mathrm{Prism}_{-/P,,\mathrm{BBM},\mathrm{formalanalytification}}(\mathcal{R}_n))^\wedge_{p,I}]_{\mathrm{BBM},\mathrm{formalanalytification}}	
\end{align}
where the notation means we take the corresponding derived $(p,I)$-completion, then we take the corresponding formal series analytification from \cite[4.2]{BBM}.\\
For any space
\begin{align}
(\mathbb{X},\mathcal{R})=\underset{n}{\mathrm{homotopycolimit}}(\mathbb{X}_n,\mathcal{R}_n)	
\end{align}
in the $\infty$-categories:
\begin{align}
&\mathrm{Ind}^\mathrm{formalspectrum}\mathrm{Sta}^\mathrm{derivedringed,\sharp}_{\mathrm{sComm}\mathrm{Simplicial}\mathrm{Ind}\mathrm{Seminormed}_R,\mathrm{homotopyepi}},\\
&\mathrm{Ind}^\mathrm{formalspectrum}\mathrm{Sta}^\mathrm{derivedringed,\sharp}_{\mathrm{sComm}\mathrm{Simplicial}\mathrm{Ind}^m\mathrm{Seminormed}_R,\mathrm{homotopyepi}},\\
&\mathrm{Ind}^\mathrm{formalspectrum}\mathrm{Sta}^\mathrm{derivedringed,\sharp}_{\mathrm{sComm}\mathrm{Simplicial}\mathrm{Ind}\mathrm{Normed}_R,\mathrm{homotopyepi}},\\
&\mathrm{Ind}^\mathrm{formalspectrum}\mathrm{Sta}^\mathrm{derivedringed,\sharp}_{\mathrm{sComm}\mathrm{Simplicial}\mathrm{Ind}^m\mathrm{Normed}_R,\mathrm{homotopyepi}},\\
&\mathrm{Ind}^\mathrm{formalspectrum}\mathrm{Sta}^\mathrm{derivedringed,\sharp}_{\mathrm{sComm}\mathrm{Simplicial}\mathrm{Ind}\mathrm{Banach}_R,\mathrm{homotopyepi}},\\
&\mathrm{Ind}^\mathrm{formalspectrum}\mathrm{Sta}^\mathrm{derivedringed,\sharp}_{\mathrm{sComm}\mathrm{Simplicial}\mathrm{Ind}^m\mathrm{Banach}_R,\mathrm{homotopyepi}},	
\end{align}
we define the corresponding prismatic stack presheaf:
\begin{align}
\mathrm{CW}_{-/P}(\mathcal{R})
\end{align}
as:
\begin{align}
&\mathrm{CW}_{-/P}(\mathcal{R})\\
&:=[(\underset{n}{\mathrm{homotopylimit}}~ \mathrm{CW}_{-/P}(\mathcal{R}_n).	
\end{align}
This as in \cite[Definition 3.1, Variant 5.1]{BL} carries the corresponding ringed topos structure 
\begin{align}
(\mathrm{CW}_{-/P}(\mathcal{R}),\mathcal{O}_{\mathrm{CW}_{-/P}(\mathcal{R})}). 
\end{align}
By \cite[Proposition 8.15]{BL} (also see \cite{Dr1}) we have that certain quasicoherent sheaves over this site will reflect completely the corresponding prismatic cohomological information. Therefore the resulting functor here $(\mathrm{CW}_{-/P},\mathcal{O}_{\mathrm{CW}_{-/P}})(-)$ will reflect the corresponding desired information for the functor $\mathrm{Prism}_{-/P}(-)$ as above.
\end{definition}

\chapter{Robba Stacks in the Commutative Algebra Situations}

\begin{reference}
\cite{KL1}, \cite{KL2}, \cite{Sch1}, \cite{Sch}, \cite{Fon}, \cite{FF}, \cite{F1}, \cite{Ta}.
\end{reference}

Now we consider the construction from \cite{KL1} and \cite{KL2}, and apply the functors in \cite[Definition 9.3.3, Definition 9.3.5, Definition 9.3.11, Definition 9.3.9]{KL1} and \cite{KL2} to the rings and spaces in our current $\infty$-categorical context. Now let $R$ be any analytic field $\mathcal{K}$. Recall from \cite[Definition 9.3.3, Definition 9.3.5, Definition 9.3.11, Definition 9.3.9]{KL1} we have the following functors:
\begin{align}
\widetilde{\mathcal{C}}_{-/R}(.),{\mathbb{B}_e}_{-/R}(.),{\mathbb{B}_\mathrm{dR}^+}_{-/R}(.),{\mathbb{B}_\mathrm{dR}}_{-/R}(.),{FF}_{-/R}(.)	
\end{align}
on the following rings:
\begin{align}
R\left<X_1,...,X_n\right>,n=0,1,2,...	
\end{align}
We then have the situation to promote the functors of rings and stacks to the $\infty$-categorical context as above. Here recall that from \cite[Definition 9.3.3, Definition 9.3.5, Definition 9.3.11, Definition 9.3.9]{KL1} and \cite{KL2}:

\begin{definition}
For any $R\left<X_1,...,X_n\right>$, we have by taking the global section:
\begin{align}
\widetilde{\mathcal{C}}_{-/R}(.)(R\left<X_1,...,X_n\right>):=\widetilde{\mathcal{C}}_{\mathrm{Spa}R\left<X_1,...,X_n\right>/R,\text{pro\'et}}(\mathrm{Spa}R\left<X_1,...,X_n\right>/R,\text{pro\'et})\\
{\mathbb{B}_e}_{-/R}(.)(R\left<X_1,...,X_n\right>):={\mathbb{B}_e}_{\mathrm{Spa}R\left<X_1,...,X_n\right>/R,\text{pro\'et}}(\mathrm{Spa}R\left<X_1,...,X_n\right>/R,\text{pro\'et}),\\
{\mathbb{B}_\mathrm{dR}^+}_{-/R}(.)(R\left<X_1,...,X_n\right>):={\mathbb{B}_\mathrm{dR}^+}_{\mathrm{Spa}R\left<X_1,...,X_n\right>/R,\text{pro\'et}}(\mathrm{Spa}R\left<X_1,...,X_n\right>/R,\text{pro\'et}),\\
{\mathbb{B}_\mathrm{dR}}_{-/R}(.)(R\left<X_1,...,X_n\right>):={\mathbb{B}_\mathrm{dR}}_{\mathrm{Spa}R\left<X_1,...,X_n\right>/R,\text{pro\'et}}(\mathrm{Spa}R\left<X_1,...,X_n\right>/R,\text{pro\'et}),\\
{{FF}}_{-/R}(.)(R\left<X_1,...,X_n\right>):={FF}_{\mathrm{Spa}R\left<X_1,...,X_n\right>/R,\text{pro\'et}}(\mathrm{Spa}R\left<X_1,...,X_n\right>/R,\text{pro\'et}).
\end{align}
And from \cite[Definition 9.3.3, Definition 9.3.5, Definition 9.3.11, Definition 9.3.9]{KL1} and \cite{KL2} we have the notation of $\varphi$-modules, $B$-pairs and the vector bundles over the FF curves as above. We now use the notation $M$ to denote them. We then put $V(.)(R\left<X_1,...,X_n\right>):=M(\mathrm{Spa}R\left<X_1,...,X_n\right>/R,\text{pro\'et})$. For the stack $FF$, this $V$ will be a corresponding vector bundle at the end over $FF_{(R\left<X_1,...,X_n\right>)^\flat/R}$. 	
\end{definition}

\begin{definition}
Following \cite[Definition 9.3.3, Definition 9.3.5, Definition 9.3.11, Definition 9.3.9]{KL1}, \cite{KL2} we give the following definition. For any ring
\begin{align}
\mathcal{R}=\underset{n}{\mathrm{homotopycolimit}}\mathcal{R}_n	
\end{align}
in the $\infty$-categories:
\begin{align}
&\mathrm{sComm}\mathrm{Simplicial}\mathrm{Ind}\mathrm{Seminormed}^\mathrm{formalseriescolimitcomp}_R,\\
&\mathrm{sComm}\mathrm{Simplicial}\mathrm{Ind}^m\mathrm{Seminormed}^\mathrm{formalseriescolimitcomp}_R,\\
&\mathrm{sComm}\mathrm{Simplicial}\mathrm{Ind}\mathrm{Normed}^\mathrm{formalseriescolimitcomp}_R,\\
&\mathrm{sComm}\mathrm{Simplicial}\mathrm{Ind}^m\mathrm{Normed}^\mathrm{formalseriescolimitcomp}_R,\\
&\mathrm{sComm}\mathrm{Simplicial}\mathrm{Ind}\mathrm{Banach}^\mathrm{formalseriescolimitcomp}_R,\\
&\mathrm{sComm}\mathrm{Simplicial}\mathrm{Ind}^m\mathrm{Banach}^\mathrm{formalseriescolimitcomp}_R.	
\end{align}	
we define the corresponding $\infty$-functors:
\begin{align}
\widetilde{\mathcal{C}}_{-/R}(.),{\mathbb{B}_e}_{-/R}(.),{\mathbb{B}_\mathrm{dR}^+}_{-/R}(.),{\mathbb{B}_\mathrm{dR}}_{-/R}(.),{FF}_{-/R}(.),V(.)	
\end{align}
as:
\begin{align}
&\widetilde{\mathcal{C}}_{-/R}(.)(\mathcal{R}):=\underset{n}{\mathrm{homotopycolimit}}~\widetilde{\mathcal{C}}_{-/R}(\mathcal{R}_n),\\
&{\mathbb{B}_e}_{-/R}(.)(\mathcal{R}):=\underset{n}{\mathrm{homotopycolimit}}~{\mathbb{B}_e}_{-/R}(.)(\mathcal{R}_n),\\
&{\mathbb{B}_\mathrm{dR}^+}_{-/R}(.)(\mathcal{R}):=\underset{n}{\mathrm{homotopycolimit}}~{\mathbb{B}_\mathrm{dR}^+}_{-/R}(.)(\mathcal{R}_n),\\
&{\mathbb{B}_\mathrm{dR}}_{-/R}(.)(\mathcal{R}):=\underset{n}{\mathrm{homotopycolimit}}~{\mathbb{B}_\mathrm{dR}}_{-/R}(.)(\mathcal{R}_n),\\
&{{FF}}_{-/R}(.)(\mathcal{R}):=\underset{n}{\mathrm{homotopylimit}}~{{FF}}_{-/R}(.)(\mathcal{R}_n),\\	
&V(.)(\mathcal{R}):=\underset{n}{\mathrm{homotopycolimit}}~V(.)(\mathcal{R}_n).
\end{align}
Here the homotopy colimits are taken in the corresponding colimit completions of the categories where the rings and spaces are living. Here the homotopy limits are taken in the corresponding limit completions of the categories where the rings and spaces are living.
\end{definition}

\indent Now motivated also by \cite{M} after \cite{CS1}, \cite{CS2} and \cite{CS3} we consider the condensed mathematical version of the construction above. Namely we look at the corresponding Clausen-Scholze's animated enhancement of the corresponding analytic condensed solid commutative algebras:
\begin{align}
\mathrm{AnalyticRings}^\mathrm{CS}_R.	
\end{align}
And we consider the corresponding colimit closure of the formal series. We denote the corresponding $\infty$-category as:
\begin{align}
\mathrm{AnalyticRings}^\mathrm{CS,formalcolimitclosure}_R.	
\end{align}

\begin{definition}
Following \cite[Definition 9.3.3, Definition 9.3.5, Definition 9.3.11, Definition 9.3.9]{KL1}, \cite{KL2} we give the following definition. For any ring
\begin{align}
\mathcal{R}=\underset{n}{\mathrm{homotopycolimit}}\mathcal{R}_n	
\end{align}
in the $\infty$-category:
\begin{align}
\mathrm{AnalyticRings}^\mathrm{CS,formalcolimitclosure}_R,	
\end{align}	
we define the corresponding $\infty$-functors:
\begin{align}
\widetilde{\mathcal{C}}_{-/R}(.),{\mathbb{B}_e}_{-/R}(.),{\mathbb{B}_\mathrm{dR}^+}_{-/R}(.),{\mathbb{B}_\mathrm{dR}}_{-/R}(.),{FF}_{-/R}(.),V(.)	
\end{align}
as:
\begin{align}
&\widetilde{\mathcal{C}}_{-/R}(.)(\mathcal{R}):=\underset{n}{\mathrm{homotopycolimit}}~\widetilde{\mathcal{C}}_{-/R}(\mathcal{R}_n),\\
&{\mathbb{B}_e}_{-/R}(.)(\mathcal{R}):=\underset{n}{\mathrm{homotopycolimit}}~{\mathbb{B}_e}_{-/R}(.)(\mathcal{R}_n),\\
&{\mathbb{B}_\mathrm{dR}^+}_{-/R}(.)(\mathcal{R}):=\underset{n}{\mathrm{homotopycolimit}}~{\mathbb{B}_\mathrm{dR}^+}_{-/R}(.)(\mathcal{R}_n),\\
&{\mathbb{B}_\mathrm{dR}}_{-/R}(.)(\mathcal{R}):=\underset{n}{\mathrm{homotopycolimit}}~{\mathbb{B}_\mathrm{dR}}_{-/R}(.)(\mathcal{R}_n),\\
&{{FF}}_{-/R}(.)(\mathcal{R}):=\underset{n}{\mathrm{homotopylimit}}~{{FF}}_{-/R}(.)(\mathcal{R}_n),\\	
&V(.)(\mathcal{R}):=\underset{n}{\mathrm{homotopycolimit}}~V(.)(\mathcal{R}_n).
\end{align}
Here the homotopy colimits are taken in the corresponding colimit completions of the categories where the rings and spaces are living. Here the homotopy limits are taken in the corresponding limit completions of the categories where the rings and spaces are living.
\end{definition}

\begin{proposition}
The corresponding $\infty$-categories of $\varphi$-module functors, $B$-pair functors and vector bundles functors over FF functors are equivalent over:
\begin{align}
&\mathrm{sComm}\mathrm{Simplicial}\mathrm{Ind}\mathrm{Seminormed}^\mathrm{formalseriescolimitcomp}_R,\\
&\mathrm{sComm}\mathrm{Simplicial}\mathrm{Ind}^m\mathrm{Seminormed}^\mathrm{formalseriescolimitcomp}_R,\\
&\mathrm{sComm}\mathrm{Simplicial}\mathrm{Ind}\mathrm{Normed}^\mathrm{formalseriescolimitcomp}_R,\\
&\mathrm{sComm}\mathrm{Simplicial}\mathrm{Ind}^m\mathrm{Normed}^\mathrm{formalseriescolimitcomp}_R,\\
&\mathrm{sComm}\mathrm{Simplicial}\mathrm{Ind}\mathrm{Banach}^\mathrm{formalseriescolimitcomp}_R,\\
&\mathrm{sComm}\mathrm{Simplicial}\mathrm{Ind}^m\mathrm{Banach}^\mathrm{formalseriescolimitcomp}_R,	
\end{align}
or:
\begin{align}
\mathrm{AnalyticRings}^\mathrm{CS,formalcolimitclosure}_R.	
\end{align}	
\end{proposition}
\begin{proof}
This is the direct consequence of \cite[Theorem 9.3.12]{KL1}.	
\end{proof}

\chapter{Robba Stacks in the Ringed Topos Situations}

\begin{reference}
\cite{KL1}, \cite{KL2}, \cite{Sch1}, \cite{Sch}, \cite{Fon}, \cite{FF}, \cite{F1}, \cite{Ta}.
\end{reference}

Now we consider the construction from \cite[Definition 9.3.3, Definition 9.3.5, Definition 9.3.11, Definition 9.3.9]{KL1} and \cite{KL2}, and apply the functors in \cite[Definition 9.3.3, Definition 9.3.5, Definition 9.3.11, Definition 9.3.9]{KL1} and \cite{KL2} to the rings and spaces in our current $\infty$-categorical context. Now let $R$ be any analytic field $\mathcal{K}$. Recall from \cite[Definition 9.3.3, Definition 9.3.5, Definition 9.3.11, Definition 9.3.9]{KL1} we have the following functors:
\begin{align}
\widetilde{\mathcal{C}}_{-/R}(.),{\mathbb{B}_e}_{-/R}(.),{\mathbb{B}_\mathrm{dR}^+}_{-/R}(.),{\mathbb{B}_\mathrm{dR}}_{-/R}(.),{FF}_{-/R}(.)	
\end{align}
on the following ringed spaces from \cite{BK}:
\begin{align}
(\mathrm{Spa}^\mathrm{BK}R\left<X_1,...,X_n\right>,\mathcal{O}_{\mathrm{Spa}^\mathrm{BK}R\left<X_1,...,X_n\right>}),n=0,1,2,...	
\end{align}
We then have the situation to promote the functors of rings and stacks to the $\infty$-categorical context as above. Here recall that from \cite[Definition 9.3.3, Definition 9.3.5, Definition 9.3.11, Definition 9.3.9]{KL1} and \cite{KL2}:

\begin{definition}
For any $\mathrm{Spa}^\mathrm{BK}R\left<X_1,...,X_n\right>$, we have by taking the global section:
\begin{align}
\widetilde{\mathcal{C}}_{-/R}(.)(\mathcal{O}_{\mathrm{Spa}^\mathrm{BK}R\left<X_1,...,X_n\right>})&(\mathrm{Spa}^\mathrm{BK}R\left<X_1,...,X_m\right>):=\\
&\widetilde{\mathcal{C}}_{\mathrm{Spa}R\left<X_1,...,X_m\right>/R,\text{pro\'et}}(\mathrm{Spa}R\left<X_1,...,X_m\right>/R,\text{pro\'et})\\
{\mathbb{B}_e}_{-/R}(.)(\mathcal{O}_{\mathrm{Spa}^\mathrm{BK}R\left<X_1,...,X_n\right>})&(\mathrm{Spa}^\mathrm{BK}R\left<X_1,...,X_m\right>):=\\
&{\mathbb{B}_e}_{\mathrm{Spa}R\left<X_1,...,X_m\right>/R,\text{pro\'et}}(\mathrm{Spa}R\left<X_1,...,X_m\right>/R,\text{pro\'et}),\\
{\mathbb{B}_\mathrm{dR}^+}_{-/R}(.)(\mathcal{O}_{\mathrm{Spa}^\mathrm{BK}R\left<X_1,...,X_n\right>})&(\mathrm{Spa}^\mathrm{BK}R\left<X_1,...,X_m\right>):=\\
&{\mathbb{B}_\mathrm{dR}^+}_{\mathrm{Spa}R\left<X_1,...,X_m\right>/R,\text{pro\'et}}(\mathrm{Spa}R\left<X_1,...,X_m\right>/R,\text{pro\'et}),\\
{\mathbb{B}_\mathrm{dR}}_{-/R}(.)(\mathcal{O}_{\mathrm{Spa}^\mathrm{BK}R\left<X_1,...,X_n\right>})&(\mathrm{Spa}^\mathrm{BK}R\left<X_1,...,X_m\right>):=\\
&{\mathbb{B}_\mathrm{dR}}_{\mathrm{Spa}R\left<X_1,...,X_m\right>/R,\text{pro\'et}}(\mathrm{Spa}R\left<X_1,...,X_m\right>/R,\text{pro\'et}),\\
{{FF}}_{-/R}(.)(\mathcal{O}_{\mathrm{Spa}^\mathrm{BK}R\left<X_1,...,X_n\right>})&(\mathrm{Spa}^\mathrm{BK}R\left<X_1,...,X_m\right>):=\\
&{FF}_{\mathrm{Spa}R\left<X_1,...,X_m\right>/R,\text{pro\'et}}(\mathrm{Spa}R\left<X_1,...,X_m\right>/R,\text{pro\'et}).
\end{align}
And from \cite[Definition 9.3.3, Definition 9.3.5, Definition 9.3.11, Definition 9.3.9]{KL1} and \cite{KL2} we have the notation of $\varphi$-modules, $B$-pairs and the vector bundles over the FF curves as above. We now use the notation $M$ to denote them. We then put $V(.)(\mathcal{O}_{\mathrm{Spa}^\mathrm{BK}R\left<X_1,...,X_n\right>})(\mathrm{Spa}^\mathrm{BK}R\left<X_1,...,X_m\right>):=M(\mathrm{Spa}R\left<X_1,...,X_m\right>/R,\text{pro\'et})$. For the stack $FF$, this $V$ will be a corresponding vector bundle at the end over $FF_{(R\left<X_1,...,X_n\right>)^\flat/R}$. 	
\end{definition}

\begin{definition}
Following \cite[Definition 9.3.3, Definition 9.3.5, Definition 9.3.11, Definition 9.3.9]{KL1}, \cite{KL2} we give the following definition. For any space
\begin{align}
(\mathbb{X},\mathcal{R})=\underset{n}{\mathrm{homotopylimit}}(\mathbb{X}_n,\mathcal{R}_n)	
\end{align}
in the $\infty$-categories:
\begin{align}
&\mathrm{Proj}^\mathrm{formalspectrum}\mathrm{Sta}^\mathrm{derivedringed,\sharp}_{\mathrm{sComm}\mathrm{Simplicial}\mathrm{Ind}\mathrm{Seminormed}_R,\mathrm{homotopyepi}},\\
&\mathrm{Proj}^\mathrm{formalspectrum}\mathrm{Sta}^\mathrm{derivedringed,\sharp}_{\mathrm{sComm}\mathrm{Simplicial}\mathrm{Ind}^m\mathrm{Seminormed}_R,\mathrm{homotopyepi}},\\
&\mathrm{Proj}^\mathrm{formalspectrum}\mathrm{Sta}^\mathrm{derivedringed,\sharp}_{\mathrm{sComm}\mathrm{Simplicial}\mathrm{Ind}\mathrm{Normed}_R,\mathrm{homotopyepi}},\\
&\mathrm{Proj}^\mathrm{formalspectrum}\mathrm{Sta}^\mathrm{derivedringed,\sharp}_{\mathrm{sComm}\mathrm{Simplicial}\mathrm{Ind}^m\mathrm{Normed}_R,\mathrm{homotopyepi}},\\
&\mathrm{Proj}^\mathrm{formalspectrum}\mathrm{Sta}^\mathrm{derivedringed,\sharp}_{\mathrm{sComm}\mathrm{Simplicial}\mathrm{Ind}\mathrm{Banach}_R,\mathrm{homotopyepi}},\\
&\mathrm{Proj}^\mathrm{formalspectrum}\mathrm{Sta}^\mathrm{derivedringed,\sharp}_{\mathrm{sComm}\mathrm{Simplicial}\mathrm{Ind}^m\mathrm{Banach}_R,\mathrm{homotopyepi}},	
\end{align}
we define the corresponding $\infty$-functors:
\begin{align}
\widetilde{\mathcal{C}}_{-/R}(.),{\mathbb{B}_e}_{-/R}(.),{\mathbb{B}_\mathrm{dR}^+}_{-/R}(.),{\mathbb{B}_\mathrm{dR}}_{-/R}(.),{FF}_{-/R}(.),V(.)	
\end{align}
as:
\begin{align}
&\widetilde{\mathcal{C}}_{-/R}(.)(\mathcal{R}):=\underset{n}{\mathrm{homotopycolimit}}~\widetilde{\mathcal{C}}_{-/R}(\mathcal{R}_n),\\
&{\mathbb{B}_e}_{-/R}(.)(\mathcal{R}):=\underset{n}{\mathrm{homotopycolimit}}~{\mathbb{B}_e}_{-/R}(.)(\mathcal{R}_n),\\
&{\mathbb{B}_\mathrm{dR}^+}_{-/R}(.)(\mathcal{R}):=\underset{n}{\mathrm{homotopycolimit}}~{\mathbb{B}_\mathrm{dR}^+}_{-/R}(.)(\mathcal{R}_n),\\
&{\mathbb{B}_\mathrm{dR}}_{-/R}(.)(\mathcal{R}):=\underset{n}{\mathrm{homotopycolimit}}~{\mathbb{B}_\mathrm{dR}}_{-/R}(.)(\mathcal{R}_n),\\
&{{FF}}_{-/R}(.)(\mathcal{R}):=\underset{n}{\mathrm{homotopylimit}}~{{FF}}_{-/R}(.)(\mathcal{R}_n),\\	
&V(.)(\mathcal{R}):=\underset{n}{\mathrm{homotopycolimit}}~V(.)(\mathcal{R}_n).
\end{align}
Here the homotopy colimits are taken in the corresponding colimit completions of the categories where the rings and spaces are living. Here the homotopy limits are taken in the corresponding limit completions of the categories where the rings and spaces are living.
\end{definition}

\chapter{Robba Stacks in the Inductive System Situations}

\begin{reference}
\cite{KL1}, \cite{KL2}, \cite{Sch1}, \cite{Sch}, \cite{Fon}, \cite{FF}, \cite{F1}, \cite{Ta}.
\end{reference}

Now we consider the construction from \cite[Definition 9.3.3, Definition 9.3.5, Definition 9.3.11, Definition 9.3.9]{KL1} and \cite{KL2}, and apply the functors in \cite[Definition 9.3.3, Definition 9.3.5, Definition 9.3.11, Definition 9.3.9]{KL1} and \cite{KL2} to the rings and spaces in our current $\infty$-categorical context. Now let $R$ be any analytic field $\mathcal{K}$. Recall from \cite[Definition 9.3.3, Definition 9.3.5, Definition 9.3.11, Definition 9.3.9]{KL1} we have the following functors:
\begin{align}
\widetilde{\mathcal{C}}_{-/R}(.),{\mathbb{B}_e}_{-/R}(.),{\mathbb{B}_\mathrm{dR}^+}_{-/R}(.),{\mathbb{B}_\mathrm{dR}}_{-/R}(.),{FF}_{-/R}(.)	
\end{align}
on the following ringed spaces from \cite{BK}:
\begin{align}
(\mathrm{Spa}^\mathrm{BK}R\left<X_1,...,X_n\right>,\mathcal{O}_{\mathrm{Spa}^\mathrm{BK}R\left<X_1,...,X_n\right>}),n=0,1,2,...	
\end{align}
We then have the situation to promote the functors of rings and stacks to the $\infty$-categorical context as above. Here recall that from \cite[Definition 9.3.3, Definition 9.3.5, Definition 9.3.11, Definition 9.3.9]{KL1} and \cite{KL2}:

\begin{definition}
For any $\mathrm{Spa}^\mathrm{BK}R\left<X_1,...,X_n\right>$, we have by taking the global section:
\begin{align}
\widetilde{\mathcal{C}}_{-/R}(.)(\mathcal{O}_{\mathrm{Spa}^\mathrm{BK}R\left<X_1,...,X_n\right>})&(\mathrm{Spa}^\mathrm{BK}R\left<X_1,...,X_m\right>):=\\
&\widetilde{\mathcal{C}}_{\mathrm{Spa}R\left<X_1,...,X_m\right>/R,\text{pro\'et}}(\mathrm{Spa}R\left<X_1,...,X_m\right>/R,\text{pro\'et})\\
{\mathbb{B}_e}_{-/R}(.)(\mathcal{O}_{\mathrm{Spa}^\mathrm{BK}R\left<X_1,...,X_n\right>})&(\mathrm{Spa}^\mathrm{BK}R\left<X_1,...,X_m\right>):=\\
&{\mathbb{B}_e}_{\mathrm{Spa}R\left<X_1,...,X_m\right>/R,\text{pro\'et}}(\mathrm{Spa}R\left<X_1,...,X_m\right>/R,\text{pro\'et}),\\
{\mathbb{B}_\mathrm{dR}^+}_{-/R}(.)(\mathcal{O}_{\mathrm{Spa}^\mathrm{BK}R\left<X_1,...,X_n\right>})&(\mathrm{Spa}^\mathrm{BK}R\left<X_1,...,X_m\right>):=\\
&{\mathbb{B}_\mathrm{dR}^+}_{\mathrm{Spa}R\left<X_1,...,X_m\right>/R,\text{pro\'et}}(\mathrm{Spa}R\left<X_1,...,X_m\right>/R,\text{pro\'et}),\\
{\mathbb{B}_\mathrm{dR}}_{-/R}(.)(\mathcal{O}_{\mathrm{Spa}^\mathrm{BK}R\left<X_1,...,X_n\right>})&(\mathrm{Spa}^\mathrm{BK}R\left<X_1,...,X_m\right>):=\\
&{\mathbb{B}_\mathrm{dR}}_{\mathrm{Spa}R\left<X_1,...,X_m\right>/R,\text{pro\'et}}(\mathrm{Spa}R\left<X_1,...,X_m\right>/R,\text{pro\'et}),\\
{{FF}}_{-/R}(.)(\mathcal{O}_{\mathrm{Spa}^\mathrm{BK}R\left<X_1,...,X_n\right>})&(\mathrm{Spa}^\mathrm{BK}R\left<X_1,...,X_m\right>):=\\
&{FF}_{\mathrm{Spa}R\left<X_1,...,X_m\right>/R,\text{pro\'et}}(\mathrm{Spa}R\left<X_1,...,X_m\right>/R,\text{pro\'et}).
\end{align}
And from \cite[Definition 9.3.3, Definition 9.3.5, Definition 9.3.11, Definition 9.3.9]{KL1} and \cite{KL2} we have the notation of $\varphi$-modules, $B$-pairs and the vector bundles over the FF curves as above. We now use the notation $M$ to denote them. We then put $V(.)(\mathcal{O}_{\mathrm{Spa}^\mathrm{BK}R\left<X_1,...,X_n\right>})(\mathrm{Spa}^\mathrm{BK}R\left<X_1,...,X_m\right>):=M(\mathrm{Spa}R\left<X_1,...,X_m\right>/R,\text{pro\'et})$. For the stack $FF$, this $V$ will be a corresponding vector bundle at the end over $FF_{(R\left<X_1,...,X_n\right>)^\flat/R}$. 	
\end{definition}

\begin{definition}
Following \cite[Definition 9.3.3, Definition 9.3.5, Definition 9.3.11, Definition 9.3.9]{KL1}, \cite{KL2} we give the following definition. For any space
\begin{align}
(\mathbb{X},\mathcal{R})=\underset{n}{\mathrm{homotopycolimit}}(\mathbb{X}_n,\mathcal{R}_n)	
\end{align}
in the $\infty$-categories:
\begin{align}
&\mathrm{Ind}^\mathrm{formalspectrum}\mathrm{Sta}^\mathrm{derivedringed,\sharp}_{\mathrm{sComm}\mathrm{Simplicial}\mathrm{Ind}\mathrm{Seminormed}_R,\mathrm{homotopyepi}},\\
&\mathrm{Ind}^\mathrm{formalspectrum}\mathrm{Sta}^\mathrm{derivedringed,\sharp}_{\mathrm{sComm}\mathrm{Simplicial}\mathrm{Ind}^m\mathrm{Seminormed}_R,\mathrm{homotopyepi}},\\
&\mathrm{Ind}^\mathrm{formalspectrum}\mathrm{Sta}^\mathrm{derivedringed,\sharp}_{\mathrm{sComm}\mathrm{Simplicial}\mathrm{Ind}\mathrm{Normed}_R,\mathrm{homotopyepi}},\\
&\mathrm{Ind}^\mathrm{formalspectrum}\mathrm{Sta}^\mathrm{derivedringed,\sharp}_{\mathrm{sComm}\mathrm{Simplicial}\mathrm{Ind}^m\mathrm{Normed}_R,\mathrm{homotopyepi}},\\
&\mathrm{Ind}^\mathrm{formalspectrum}\mathrm{Sta}^\mathrm{derivedringed,\sharp}_{\mathrm{sComm}\mathrm{Simplicial}\mathrm{Ind}\mathrm{Banach}_R,\mathrm{homotopyepi}},\\
&\mathrm{Ind}^\mathrm{formalspectrum}\mathrm{Sta}^\mathrm{derivedringed,\sharp}_{\mathrm{sComm}\mathrm{Simplicial}\mathrm{Ind}^m\mathrm{Banach}_R,\mathrm{homotopyepi}},	
\end{align}

we define the corresponding $\infty$-functors:
\begin{align}
\widetilde{\mathcal{C}}_{-/R}(.),{\mathbb{B}_e}_{-/R}(.),{\mathbb{B}_\mathrm{dR}^+}_{-/R}(.),{\mathbb{B}_\mathrm{dR}}_{-/R}(.),{FF}_{-/R}(.),V(.)	
\end{align}
as:
\begin{align}
&\widetilde{\mathcal{C}}_{-/R}(.)(\mathcal{R}):=\underset{n}{\mathrm{homotopylimit}}~\widetilde{\mathcal{C}}_{-/R}(\mathcal{R}_n),\\
&{\mathbb{B}_e}_{-/R}(.)(\mathcal{R}):=\underset{n}{\mathrm{homotopylimit}}~{\mathbb{B}_e}_{-/R}(.)(\mathcal{R}_n),\\
&{\mathbb{B}_\mathrm{dR}^+}_{-/R}(.)(\mathcal{R}):=\underset{n}{\mathrm{homotopylimit}}~{\mathbb{B}_\mathrm{dR}^+}_{-/R}(.)(\mathcal{R}_n),\\
&{\mathbb{B}_\mathrm{dR}}_{-/R}(.)(\mathcal{R}):=\underset{n}{\mathrm{homotopylimit}}~{\mathbb{B}_\mathrm{dR}}_{-/R}(.)(\mathcal{R}_n),\\
&{{FF}}_{-/R}(.)(\mathcal{R}):=\underset{n}{\mathrm{homotopycolimit}}~{{FF}}_{-/R}(.)(\mathcal{R}_n),\\	
&V(.)(\mathcal{R}):=\underset{n}{\mathrm{homotopylimit}}~V(.)(\mathcal{R}_n).
\end{align}
\end{definition}

\

\begin{remark}
Here the homotopy colimits are taken in the corresponding colimit completions of the categories where the rings and spaces are living. Here the homotopy limits are taken in the corresponding limit completions of the categories where the rings and spaces are living. For instance, the Robba functor $\widetilde{\mathcal{C}}_{-/R}(.)$ takes value in ind-Fr\'echet rings $\mathrm{Ind}\text{Fr\'echet}_R$, we then consider the corresponding homotopy limit closure $\overline{\mathrm{Ind}\text{Fr\'echet}}^{\mathrm{homotopylimit}}_R$. For instance, the Fargues-Fontaine stack functors ${{FF}}_{-/R}(.)(-)$ take value in the preadic spaces $\mathrm{PreAdic}_R$, then we consider the corresponding homotopy colimit closure $\overline{\mathrm{PreAdic}_R}^{\mathrm{homotopycolimit}}$. 	
\end{remark}

\chapter{Generalizations}

The generalizations here are closely following the following work:

\begin{reference}
\cite{BHS}, \cite{BS}, \cite{BL3}, \cite{BL4}, \cite{Fon2}.\footnote{And we assume that $p>2$. We use the notation $k$ to represent the element $2\pi\sqrt{-1}$ in the $p$-adic complex analysis, of course over $B_\mathrm{dR}=\mathbb{C}_p[[k]][1/k]$ not the obvious $\mathbb{C}_p=\overline{\mathbb{Q}_p}^\wedge$, which is the crucial difference when we do $p$-adic analysis, see Fontaine's 1982 breakthrough \cite{Fon}.}
\end{reference}

\newpage

\section{Derived Generalized Prismatic Cohomology}

We now follow \cite{Grot1}, \cite{Grot2}, \cite{Grot3}, \cite{Grot4}, \cite{BK}, \cite{BBK}, \cite{BBBK}, \cite{BBM}, \cite{KKM}, \cite{T2}, \cite{Sch2}, \cite{BS}, \cite{BL}, \cite{Dr1}\footnote{One can consider the corresponding absolute prismatic complexes \cite{BS}, \cite{BL2}, \cite{BL}, \cite{Dr1} as well, though our presentation fix a corresponding base prism $(P,I)$ where $P/I$ is assumed to be Banach giving rise to the $p$-adic topology. And we assume the boundedness. } to revisit and discuss the corresponding derived prismatic cohomology for rings in the following:

\begin{notation}\mbox{\rm{(Rings)}}
Recall we have the following six categories on the commutative algebras in the derived sense (let $R$ be $P/I$):
\begin{align}
&\mathrm{sComm}\mathrm{Simplicial}\mathrm{Ind}\mathrm{Seminormed}_R,\\
&\mathrm{sComm}\mathrm{Simplicial}\mathrm{Ind}^m\mathrm{Seminormed}_R,\\
&\mathrm{sComm}\mathrm{Simplicial}\mathrm{Ind}\mathrm{Normed}_R,\\
&\mathrm{sComm}\mathrm{Simplicial}\mathrm{Ind}^m\mathrm{Normed}_R,\\
&\mathrm{sComm}\mathrm{Simplicial}\mathrm{Ind}\mathrm{Banach}_R,\\
&\mathrm{sComm}\mathrm{Simplicial}\mathrm{Ind}^m\mathrm{Banach}_R.	
\end{align}
	
\end{notation}

\begin{definition}
We now consider the rings:
\begin{align}
P/I\left<X_1,...,X_n\right>,n=0,1,2,...	
\end{align}
Then we take the corresponding homotopy colimit completion of these in the stable $\infty$-categories above:
\begin{align}
&\mathrm{sComm}\mathrm{Simplicial}\mathrm{Ind}\mathrm{Seminormed}_R,\\
&\mathrm{sComm}\mathrm{Simplicial}\mathrm{Ind}^m\mathrm{Seminormed}_R,\\
&\mathrm{sComm}\mathrm{Simplicial}\mathrm{Ind}\mathrm{Normed}_R,\\
&\mathrm{sComm}\mathrm{Simplicial}\mathrm{Ind}^m\mathrm{Normed}_R,\\
&\mathrm{sComm}\mathrm{Simplicial}\mathrm{Ind}\mathrm{Banach}_R,\\
&\mathrm{sComm}\mathrm{Simplicial}\mathrm{Ind}^m\mathrm{Banach}_R.	
\end{align}
The resulting $\infty$-categories will be denoted by:
\begin{align}
&\mathrm{sComm}\mathrm{Simplicial}\mathrm{Ind}\mathrm{Seminormed}^\mathrm{formalseriescolimitcomp}_R,\\
&\mathrm{sComm}\mathrm{Simplicial}\mathrm{Ind}^m\mathrm{Seminormed}^\mathrm{formalseriescolimitcomp}_R,\\
&\mathrm{sComm}\mathrm{Simplicial}\mathrm{Ind}\mathrm{Normed}^\mathrm{formalseriescolimitcomp}_R,\\
&\mathrm{sComm}\mathrm{Simplicial}\mathrm{Ind}^m\mathrm{Normed}^\mathrm{formalseriescolimitcomp}_R,\\
&\mathrm{sComm}\mathrm{Simplicial}\mathrm{Ind}\mathrm{Banach}^\mathrm{formalseriescolimitcomp}_R,\\
&\mathrm{sComm}\mathrm{Simplicial}\mathrm{Ind}^m\mathrm{Banach}^\mathrm{formalseriescolimitcomp}_R.	
\end{align}	
\end{definition}

\indent We then follow \cite{BS}, \cite{BL}, \cite{Dr1} to give the following definitions on the prismatic complexes $\Delta_{-/P}$.

\begin{definition}
Following \cite[Construction 7.6]{BS} we give the following definition. For any ring
\begin{align}
\mathcal{R}=\underset{n}{\mathrm{homotopycolimit}}\mathcal{R}_n	
\end{align}
in the $\infty$-categories:
\begin{align}
&\mathrm{sComm}\mathrm{Simplicial}\mathrm{Ind}\mathrm{Seminormed}^\mathrm{formalseriescolimitcomp}_R,\\
&\mathrm{sComm}\mathrm{Simplicial}\mathrm{Ind}^m\mathrm{Seminormed}^\mathrm{formalseriescolimitcomp}_R,\\
&\mathrm{sComm}\mathrm{Simplicial}\mathrm{Ind}\mathrm{Normed}^\mathrm{formalseriescolimitcomp}_R,\\
&\mathrm{sComm}\mathrm{Simplicial}\mathrm{Ind}^m\mathrm{Normed}^\mathrm{formalseriescolimitcomp}_R,\\
&\mathrm{sComm}\mathrm{Simplicial}\mathrm{Ind}\mathrm{Banach}^\mathrm{formalseriescolimitcomp}_R,\\
&\mathrm{sComm}\mathrm{Simplicial}\mathrm{Ind}^m\mathrm{Banach}^\mathrm{formalseriescolimitcomp}_R,	
\end{align}	
we define the corresponding generalized prismatic cohomology:
\begin{align}
\mathrm{Prism}^{\sqrt{I}}_{-/P,\mathrm{BBM},\mathrm{analytification}}(\mathcal{R})
\end{align}
as:
\begin{align}
&\mathrm{Prism}^{\sqrt{I}}_{-/P,\mathrm{BBM},\mathrm{analytification}}(\mathcal{R})\\
&:=\left[\left(\underset{n}{\mathrm{homotopycolimit}}~ (\mathrm{Prism}_{-/P,\mathrm{BBM},\mathrm{formalanalytification}}(\mathcal{R}_n)[\sqrt{I}])^\wedge_{p,\sqrt{I}}\right)^\wedge_{p,\sqrt{I}}\right]_{\mathrm{BBM},\mathrm{formalanalytification}}	
\end{align}
where the notation means we take the corresponding derived $(p,I)$-completion and derived $(p,\sqrt{I})$-completion, then we take the corresponding formal series analytification from \cite[4.2]{BBM}.\\
\end{definition}

\indent Now we consider preperfectoidization constuctions:

\begin{definition}
Following \cite{BS}, we give the following definition. For any ring
\begin{align}
\mathcal{R}=\underset{n}{\mathrm{homotopycolimit}}\mathcal{R}_n	
\end{align}
in the $\infty$-categories:
\begin{align}
&\mathrm{sComm}\mathrm{Simplicial}\mathrm{Ind}\mathrm{Seminormed}^\mathrm{formalseriescolimitcomp}_R,\\
&\mathrm{sComm}\mathrm{Simplicial}\mathrm{Ind}^m\mathrm{Seminormed}^\mathrm{formalseriescolimitcomp}_R,\\
&\mathrm{sComm}\mathrm{Simplicial}\mathrm{Ind}\mathrm{Normed}^\mathrm{formalseriescolimitcomp}_R,\\
&\mathrm{sComm}\mathrm{Simplicial}\mathrm{Ind}^m\mathrm{Normed}^\mathrm{formalseriescolimitcomp}_R,\\
&\mathrm{sComm}\mathrm{Simplicial}\mathrm{Ind}\mathrm{Banach}^\mathrm{formalseriescolimitcomp}_R,\\
&\mathrm{sComm}\mathrm{Simplicial}\mathrm{Ind}^m\mathrm{Banach}^\mathrm{formalseriescolimitcomp}_R,	
\end{align}	
we define the corresponding generalized prismatic cohomology:
\begin{align}
\mathrm{Prism}^{\sqrt{I}}_{-/P,\mathrm{BBM},\mathrm{analytification}}(\mathcal{R})
\end{align}
as:
\begin{align}
&\mathrm{Prism}^{\sqrt{I}}_{-/P,\mathrm{BBM},\mathrm{analytification}}(\mathcal{R})\\
&:=\left[\left(\underset{n}{\mathrm{homotopycolimit}}~ (\mathrm{Prism}_{-/P,\mathrm{BBM},\mathrm{formalanalytification}}(\mathcal{R}_n)[\sqrt{I}])^\wedge_{p,\sqrt{I}}\right)^\wedge_{p,\sqrt{I}}\right]_{\mathrm{BBM},\mathrm{formalanalytification}}	
\end{align}
where the notation means we take the corresponding derived $(p,I)$-completion and derived $(p,\sqrt{I})$-completion, then we take the corresponding formal series analytification from \cite[4.2]{BBM}. Then as in \cite[Definition 8.2]{BS} we put the following generalized  preperfecdtoidization of any object $R$ to be:
\begin{align}
&R^{\mathrm{preperfectoidization},\sqrt{I}}:=\\
&\underset{i}{\mathrm{homotopycolimit}}(\mathrm{Prism}^{\sqrt{I}}_{-/P,\mathrm{BBM},\mathrm{analytification}}(\mathcal{R})\rightarrow \mathrm{Fro}_*\mathrm{Prism}^{\sqrt{I}}_{-/P,\mathrm{BBM},\mathrm{analytification}}(\mathcal{R})\\
&\rightarrow \mathrm{Fro}_*\mathrm{Fro}_*\mathrm{Prism}^{\sqrt{I}}_{-/P,\mathrm{BBM},\mathrm{analytification}}(\mathcal{R})\rightarrow...)	
\end{align}
Then the generalized perfectoidization of $R$ is just defined to be:
\begin{align}
R^\mathrm{preperfectoidization}\times P/I.	
\end{align}

\end{definition}

\newpage
\section{Derived Generalized Topological Hochschild Homology}

We now follow \cite{Grot1}, \cite{Grot2}, \cite{Grot3}, \cite{Grot4}, \cite{BK}, \cite{BBK}, \cite{BBBK}, \cite{BBM}, \cite{KKM}, \cite{T2}, \cite{Sch2}, \cite{BS}, \cite{BL}, \cite{Dr1}, \cite{NS}, \cite{BMS}, \cite{B}, \cite{BHM}\footnote{Our presentation fixes a corresponding base prism $(P,I)$ where $P/I$ is assumed to be Banach giving rise to the $p$-adic topology. And we assume the boundedness. } to revisit and discuss the corresponding derived topological Hochschild homology, topological period homology and topological cyclic homology for rings in the following:

\begin{notation}\mbox{\rm{(Rings)}}
Recall we have the following six categories on the noncommutative algebras in the derived sense (let $R$ be $P/I$):
\begin{align}
&\mathrm{sNoncomm}\mathrm{Simplicial}\mathrm{Ind}\mathrm{Seminormed}_R,\\
&\mathrm{sNoncomm}\mathrm{Simplicial}\mathrm{Ind}^m\mathrm{Seminormed}_R,\\
&\mathrm{sNoncomm}\mathrm{Simplicial}\mathrm{Ind}\mathrm{Normed}_R,\\
&\mathrm{sNoncomm}\mathrm{Simplicial}\mathrm{Ind}^m\mathrm{Normed}_R,\\
&\mathrm{sNoncomm}\mathrm{Simplicial}\mathrm{Ind}\mathrm{Banach}_R,\\
&\mathrm{sNoncomm}\mathrm{Simplicial}\mathrm{Ind}^m\mathrm{Banach}_R.	
\end{align}
	
\end{notation}

\begin{definition}
We now consider the rings\footnote{$Z_1,...,Z_n$ are just assumed to be free variables.}:
\begin{align}
P/I\left<Z_1,...,Z_n\right>,n=0,1,2,...	
\end{align}
Then we take the corresponding homotopy colimit completion of these in the stable $\infty$-categories above:
\begin{align}
&\mathrm{Noncomm}\mathrm{Simplicial}\mathrm{Ind}\mathrm{Seminormed}_R,\\
&\mathrm{Noncomm}\mathrm{Simplicial}\mathrm{Ind}^m\mathrm{Seminormed}_R,\\
&\mathrm{Noncomm}\mathrm{Simplicial}\mathrm{Ind}\mathrm{Normed}_R,\\
&\mathrm{Noncomm}\mathrm{Simplicial}\mathrm{Ind}^m\mathrm{Normed}_R,\\
&\mathrm{Noncomm}\mathrm{Simplicial}\mathrm{Ind}\mathrm{Banach}_R,\\
&\mathrm{Noncomm}\mathrm{Simplicial}\mathrm{Ind}^m\mathrm{Banach}_R.	
\end{align}
The resulting $\infty$-categories will be denoted by:
\begin{align}
&\mathrm{sNoncomm}\mathrm{Simplicial}\mathrm{Ind}\mathrm{Seminormed}^\mathrm{formalseriescolimitcomp}_R,\\
&\mathrm{sNoncomm}\mathrm{Simplicial}\mathrm{Ind}^m\mathrm{Seminormed}^\mathrm{formalseriescolimitcomp}_R,\\
&\mathrm{sNoncomm}\mathrm{Simplicial}\mathrm{Ind}\mathrm{Normed}^\mathrm{formalseriescolimitcomp}_R,\\
&\mathrm{sNoncomm}\mathrm{Simplicial}\mathrm{Ind}^m\mathrm{Normed}^\mathrm{formalseriescolimitcomp}_R,\\
&\mathrm{sNoncomm}\mathrm{Simplicial}\mathrm{Ind}\mathrm{Banach}^\mathrm{formalseriescolimitcomp}_R,\\
&\mathrm{sNoncomm}\mathrm{Simplicial}\mathrm{Ind}^m\mathrm{Banach}^\mathrm{formalseriescolimitcomp}_R.	
\end{align}	
\end{definition}
\

\indent We then follow \cite[Section 2.3]{BMS}, \cite[Chapter 3]{NS} to give the following definitions on the topological Hochschild complexes, topological period complexes and topological cyclic complexes
\begin{align}
 \mathrm{THH}_{-/P,\mathrm{BBM},\mathrm{analytification}},\\
 \mathrm{TP}_{-/P,\mathrm{BBM},\mathrm{analytification}},\\
 \mathrm{TC}_{-/P,\mathrm{BBM},\mathrm{analytification}}. 
\end{align}
 All the constructions are directly applications of functors in \cite[Section 2.3]{BMS}, \cite[Chapter 3]{NS}.

\begin{definition}
Following \cite[Section 2.3]{BMS} and \cite[Chapter 3]{NS} we give the following definition. For any ring
\begin{align}
\mathcal{R}=\underset{n}{\mathrm{homotopycolimit}}\mathcal{R}_n	
\end{align}
in the $\infty$-categories:
\begin{align}
&\mathrm{sNoncomm}\mathrm{Simplicial}\mathrm{Ind}\mathrm{Seminormed}^\mathrm{formalseriescolimitcomp}_R,\\
&\mathrm{sNoncomm}\mathrm{Simplicial}\mathrm{Ind}^m\mathrm{Seminormed}^\mathrm{formalseriescolimitcomp}_R,\\
&\mathrm{sNoncomm}\mathrm{Simplicial}\mathrm{Ind}\mathrm{Normed}^\mathrm{formalseriescolimitcomp}_R,\\
&\mathrm{sNoncomm}\mathrm{Simplicial}\mathrm{Ind}^m\mathrm{Normed}^\mathrm{formalseriescolimitcomp}_R,\\
&\mathrm{sNoncomm}\mathrm{Simplicial}\mathrm{Ind}\mathrm{Banach}^\mathrm{formalseriescolimitcomp}_R,\\
&\mathrm{sNoncomm}\mathrm{Simplicial}\mathrm{Ind}^m\mathrm{Banach}^\mathrm{formalseriescolimitcomp}_R,	
\end{align}	
we define the corresponding generalized topological Hochschild complexes, generalized topological period complexes and generalized topological cyclic complexes $\mathrm{THH}^{\sqrt{I}}_{-/P}$, $\mathrm{TP}^{\sqrt{I}}_{-/P}$, $\mathrm{TC}^{\sqrt{I}}_{-/P}$:
\begin{align}
 \mathrm{THH}^{\sqrt{I}}_{-/P,\mathrm{BBM},\mathrm{analytification}}(\mathcal{R}),\\
 \mathrm{TP}^{\sqrt{I}}_{-/P,\mathrm{BBM},\mathrm{analytification}}(\mathcal{R}),\\
 \mathrm{TC}^{\sqrt{I}}_{-/P,\mathrm{BBM},\mathrm{analytification}}(\mathcal{R}). 
\end{align}
as:
\begin{align}
& \mathrm{THH}^{\sqrt{I}}_{-/P,\mathrm{BBM},\mathrm{analytification}}(\mathcal{R})\\
&:=[(\underset{n}{\mathrm{homotopycolimit}}~  \mathrm{THH}_{-/P,\mathrm{BBM},\mathrm{analytification}}(\mathcal{R}_n))^\wedge_{p}[\sqrt{I}]]_{\mathrm{BBM},\mathrm{formalanalytification}}\\
& \mathrm{TP}^{\sqrt{I}}_{-/P,\mathrm{BBM},\mathrm{analytification}}(\mathcal{R})\\
&:=[(\underset{n}{\mathrm{homotopycolimit}}~  \mathrm{TP}_{-/P,\mathrm{BBM},\mathrm{analytification}}(\mathcal{R}_n))^\wedge_{p}[\sqrt{I}]]_{\mathrm{BBM},\mathrm{formalanalytification}}\\
& \mathrm{TC}^{\sqrt{I}}_{-/P,\mathrm{BBM},\mathrm{analytification}}(\mathcal{R})\\
&:=[(\underset{n}{\mathrm{homotopycolimit}}~  \mathrm{TC}_{-/P,\mathrm{BBM},\mathrm{analytification}}(\mathcal{R}_n))^\wedge_{p}[\sqrt{I}]]_{\mathrm{BBM},\mathrm{formalanalytification}}\\	
\end{align}
where the notation means we take the corresponding algebraic topological $p$-completion, then we take the corresponding formal series analytification from \cite[4.2]{BBM} in the corresponding analogy of the commutative situation.\\
\end{definition}

\indent Now we consider preperfectoidization constuctions:

\begin{definition}
Following \cite{BS}, we give the following definition. For any ring
\begin{align}
\mathcal{R}=\underset{n}{\mathrm{homotopycolimit}}\mathcal{R}_n	
\end{align}
in the $\infty$-categories:
\begin{align}
&\mathrm{sNoncomm}\mathrm{Simplicial}\mathrm{Ind}\mathrm{Seminormed}^\mathrm{formalseriescolimitcomp}_R,\\
&\mathrm{sNoncomm}\mathrm{Simplicial}\mathrm{Ind}^m\mathrm{Seminormed}^\mathrm{formalseriescolimitcomp}_R,\\
&\mathrm{sNoncomm}\mathrm{Simplicial}\mathrm{Ind}\mathrm{Normed}^\mathrm{formalseriescolimitcomp}_R,\\
&\mathrm{sNoncomm}\mathrm{Simplicial}\mathrm{Ind}^m\mathrm{Normed}^\mathrm{formalseriescolimitcomp}_R,\\
&\mathrm{sNoncomm}\mathrm{Simplicial}\mathrm{Ind}\mathrm{Banach}^\mathrm{formalseriescolimitcomp}_R,\\
&\mathrm{sNoncomm}\mathrm{Simplicial}\mathrm{Ind}^m\mathrm{Banach}^\mathrm{formalseriescolimitcomp}_R,	
\end{align}	
we define the corresponding generalized topological Hochschild complexes, generalized topological period complexes and generalized topological cyclic complexes $\mathrm{THH}^{\sqrt{I}}_{-/P}$, $\mathrm{TP}^{\sqrt{I}}_{-/P}$, $\mathrm{TC}^{\sqrt{I}}_{-/P}$:
\begin{align}
 \mathrm{THH}^{\sqrt{I}}_{-/P,\mathrm{BBM},\mathrm{analytification}}(\mathcal{R}),\\
 \mathrm{TP}^{\sqrt{I}}_{-/P,\mathrm{BBM},\mathrm{analytification}}(\mathcal{R}),\\
 \mathrm{TC}^{\sqrt{I}}_{-/P,\mathrm{BBM},\mathrm{analytification}}(\mathcal{R}). 
\end{align}
as:
\begin{align}
& \mathrm{THH}^{\sqrt{I}}_{-/P,\mathrm{BBM},\mathrm{analytification}}(\mathcal{R})\\
&:=[(\underset{n}{\mathrm{homotopycolimit}}~  \mathrm{THH}_{-/P,\mathrm{BBM},\mathrm{analytification}}(\mathcal{R}_n))^\wedge_{p}[{\sqrt{I}}]]_{\mathrm{BBM},\mathrm{formalanalytification}}\\
& \mathrm{TP}^{\sqrt{I}}_{-/P,\mathrm{BBM},\mathrm{analytification}}(\mathcal{R})\\
&:=[(\underset{n}{\mathrm{homotopycolimit}}~  \mathrm{TP}_{-/P,\mathrm{BBM},\mathrm{analytification}}(\mathcal{R}_n))^\wedge_{p}[{\sqrt{I}}]]_{\mathrm{BBM},\mathrm{formalanalytification}}\\
& \mathrm{TC}^{\sqrt{I}}_{-/P,\mathrm{BBM},\mathrm{analytification}}(\mathcal{R})\\
&:=[(\underset{n}{\mathrm{homotopycolimit}}~  \mathrm{TC}_{-/P,\mathrm{BBM},\mathrm{analytification}}(\mathcal{R}_n))^\wedge_{p}[{\sqrt{I}}]]_{\mathrm{BBM},\mathrm{formalanalytification}}\\	
\end{align}
where the notation means we take the corresponding algebraic topological $p$-completion, then we take the corresponding formal series analytification from \cite[4.2]{BBM} in the corresponding analogy of the commutative situation. Then as in \cite[Definition 8.2]{BS} we put the following generalized preperfectoidizations of any object $R$ to be:
\begin{align}
&R^\mathrm{preperfectoidization,THH,{\sqrt{I}}}:=\\
&\underset{i}{\mathrm{homotopycolimit}}(\mathrm{THH}^{\sqrt{I}}_{-/P,\mathrm{BBM},\mathrm{analytification}}(\mathcal{R})\rightarrow \mathrm{Fro}_*\mathrm{THH}^{\sqrt{I}}_{-/P,\mathrm{BBM},\mathrm{analytification}}(\mathcal{R})\\
&\rightarrow \mathrm{Fro}_*\mathrm{Fro}_*\mathrm{THH}^{\sqrt{I}}_{-/P,\mathrm{BBM},\mathrm{analytification}}(\mathcal{R})\rightarrow...)	\\
&R^\mathrm{preperfectoidization,TP,{\sqrt{I}}}:=\\
&\underset{i}{\mathrm{homotopycolimit}}(\mathrm{TP}^{\sqrt{I}}_{-/P,\mathrm{BBM},\mathrm{analytification}}(\mathcal{R})\rightarrow \mathrm{Fro}_*\mathrm{TP}^{\sqrt{I}}_{-/P,\mathrm{BBM},\mathrm{analytification}}(\mathcal{R})\\
&\rightarrow \mathrm{Fro}_*\mathrm{Fro}_*\mathrm{TP}^{\sqrt{I}}_{-/P,\mathrm{BBM},\mathrm{analytification}}(\mathcal{R})\rightarrow...)	\\
&R^\mathrm{preperfectoidization,TC,{\sqrt{I}}}:=\\
&\underset{i}{\mathrm{homotopycolimit}}(\mathrm{TC}^{\sqrt{I}}_{-/P,\mathrm{BBM},\mathrm{analytification}}(\mathcal{R})\rightarrow \mathrm{Fro}_*\mathrm{TC}^{\sqrt{I}}_{-/P,\mathrm{BBM},\mathrm{analytification}}(\mathcal{R})\\
&\rightarrow \mathrm{Fro}_*\mathrm{Fro}_*\mathrm{TC}^{\sqrt{I}}_{-/P,\mathrm{BBM},\mathrm{analytification}}(\mathcal{R})\rightarrow...)	\\
\end{align}
Then the generalized perfectoidizations of $R$ are just defined to be:
\begin{align}
R^\mathrm{preperfectoidization,\sharp}\times P/I.	
\end{align}
Here $\sharp$ represents one of $\mathrm{THH},\mathrm{TP},\mathrm{TC}$.
\end{definition}

\newpage
\section{Derived Generalized Prismatic Cohomology for Ringed Toposes}

We now follow \cite{Grot1}, \cite{Grot2}, \cite{Grot3}, \cite{Grot4}, \cite{BK}, \cite{BBK}, \cite{BBBK}, \cite{BBM}, \cite{KKM}, \cite{T2}, \cite{Sch2}, \cite{BS}, \cite{BL}, \cite{Dr1}\footnote{One can consider the corresponding absolute prismatic complexes \cite{BS}, \cite{BL2}, \cite{BL}, \cite{Dr1} as well, though our presentation fix a corresponding base prism $(P,I)$ where $P/I$ is assumed to be Banach giving rise to the $p$-adic topology. And we assume the boundedness.} to revisit and discuss the corresponding derived prismatic cohomology for rings in the following. We first consider the following generating ringed spaces from \cite{BK}:
\begin{align}
(\mathrm{Spa}^\mathrm{BK}P/I\left<X_1,...,X_n\right>,\mathcal{R}_{\mathrm{Spa}^\mathrm{BK}P/I\left<X_1,...,X_n\right>}),n=0,1,2,...
\end{align}
\begin{definition}
We now consider the homotopy limit completion of 
\begin{align}
(\mathrm{Spa}^\mathrm{BK}P/I\left<X_1,...,X_n\right>,\mathcal{R}_{\mathrm{Spa}^\mathrm{BK}P/I\left<X_1,...,X_n\right>}),n=0,1,2,...
\end{align}
in the following $\infty$-categories:
\begin{align}
&\mathrm{Sta}^\mathrm{derivedringed,\sharp}_{\mathrm{sComm}\mathrm{Simplicial}\mathrm{Ind}\mathrm{Seminormed}_R,\mathrm{homotopyepi}},\\
&\mathrm{Sta}^\mathrm{derivedringed,\sharp}_{\mathrm{sComm}\mathrm{Simplicial}\mathrm{Ind}^m\mathrm{Seminormed}_R,\mathrm{homotopyepi}},\\
&\mathrm{Sta}^\mathrm{derivedringed,\sharp}_{\mathrm{sComm}\mathrm{Simplicial}\mathrm{Ind}\mathrm{Normed}_R,\mathrm{homotopyepi}},\\
&\mathrm{Sta}^\mathrm{derivedringed,\sharp}_{\mathrm{sComm}\mathrm{Simplicial}\mathrm{Ind}^m\mathrm{Normed}_R,\mathrm{homotopyepi}},\\
&\mathrm{Sta}^\mathrm{derivedringed,\sharp}_{\mathrm{sComm}\mathrm{Simplicial}\mathrm{Ind}\mathrm{Banach}_R,\mathrm{homotopyepi}},\\
&\mathrm{Sta}^\mathrm{derivedringed,\sharp}_{\mathrm{sComm}\mathrm{Simplicial}\mathrm{Ind}^m\mathrm{Banach}_R,\mathrm{homotopyepi}}.	
\end{align}
Here $\sharp$ represents any category in the following:
\begin{align}
&\mathrm{sComm}\mathrm{Simplicial}\mathrm{Ind}\mathrm{Seminormed}_R,\\
&\mathrm{sComm}\mathrm{Simplicial}\mathrm{Ind}^m\mathrm{Seminormed}_R,\\
&\mathrm{sComm}\mathrm{Simplicial}\mathrm{Ind}\mathrm{Normed}_R,\\
&\mathrm{sComm}\mathrm{Simplicial}\mathrm{Ind}^m\mathrm{Normed}_R,\\
&\mathrm{sComm}\mathrm{Simplicial}\mathrm{Ind}\mathrm{Banach}_R,\\
&\mathrm{sComm}\mathrm{Simplicial}\mathrm{Ind}^m\mathrm{Banach}_R.	
\end{align}	
Here $R=P/I$. The resulting sub $\infty$-categories are denoted by:
\begin{align}
&\mathrm{Proj}^\mathrm{formalspectrum}\mathrm{Sta}^\mathrm{derivedringed,\sharp}_{\mathrm{sComm}\mathrm{Simplicial}\mathrm{Ind}\mathrm{Seminormed}_R,\mathrm{homotopyepi}},\\
&\mathrm{Proj}^\mathrm{formalspectrum}\mathrm{Sta}^\mathrm{derivedringed,\sharp}_{\mathrm{sComm}\mathrm{Simplicial}\mathrm{Ind}^m\mathrm{Seminormed}_R,\mathrm{homotopyepi}},\\
&\mathrm{Proj}^\mathrm{formalspectrum}\mathrm{Sta}^\mathrm{derivedringed,\sharp}_{\mathrm{sComm}\mathrm{Simplicial}\mathrm{Ind}\mathrm{Normed}_R,\mathrm{homotopyepi}},\\
&\mathrm{Proj}^\mathrm{formalspectrum}\mathrm{Sta}^\mathrm{derivedringed,\sharp}_{\mathrm{sComm}\mathrm{Simplicial}\mathrm{Ind}^m\mathrm{Normed}_R,\mathrm{homotopyepi}},\\
&\mathrm{Proj}^\mathrm{formalspectrum}\mathrm{Sta}^\mathrm{derivedringed,\sharp}_{\mathrm{sComm}\mathrm{Simplicial}\mathrm{Ind}\mathrm{Banach}_R,\mathrm{homotopyepi}},\\
&\mathrm{Proj}^\mathrm{formalspectrum}\mathrm{Sta}^\mathrm{derivedringed,\sharp}_{\mathrm{sComm}\mathrm{Simplicial}\mathrm{Ind}^m\mathrm{Banach}_R,\mathrm{homotopyepi}}.	
\end{align}
Here $\sharp$ represents any category in the following:
\begin{align}
&\mathrm{sComm}\mathrm{Simplicial}\mathrm{Ind}\mathrm{Seminormed}_R,\\
&\mathrm{sComm}\mathrm{Simplicial}\mathrm{Ind}^m\mathrm{Seminormed}_R,\\
&\mathrm{sComm}\mathrm{Simplicial}\mathrm{Ind}\mathrm{Normed}_R,\\
&\mathrm{sComm}\mathrm{Simplicial}\mathrm{Ind}^m\mathrm{Normed}_R,\\
&\mathrm{sComm}\mathrm{Simplicial}\mathrm{Ind}\mathrm{Banach}_R,\\
&\mathrm{sComm}\mathrm{Simplicial}\mathrm{Ind}^m\mathrm{Banach}_R.	
\end{align}	
This means that any space $(\mathbb{X},\mathcal{R})$ in the full $\infty$-categories could be written as the following:
\begin{align}
(\mathbb{X},\mathcal{R})=\underset{n}{\mathrm{homotopylimit}}(\mathbb{X}_n,\mathcal{R}_n)	
\end{align}
where we have then:
\begin{align}
\mathcal{R}=\underset{n}{\mathrm{homotopycolimit}}\mathcal{R}_n	
\end{align}
as coherent sheaves over each $\mathbb{X}_n$.

\end{definition}

\indent We then follow \cite{BS}, \cite{BL}, \cite{Dr1} to give the following definitions on the prismatic cohomology presheaf $\Delta_{-/P}$.

\begin{definition}
Following \cite[Construction 7.6]{BS} we give the following definition. For any space
\begin{align}
(\mathbb{X},\mathcal{R})=\underset{n}{\mathrm{homotopylimit}}(\mathbb{X}_n,\mathcal{R}_n)	
\end{align}
in the $\infty$-categories:
\begin{align}
&\mathrm{Proj}^\mathrm{formalspectrum}\mathrm{Sta}^\mathrm{derivedringed,\sharp}_{\mathrm{sComm}\mathrm{Simplicial}\mathrm{Ind}\mathrm{Seminormed}_R,\mathrm{homotopyepi}},\\
&\mathrm{Proj}^\mathrm{formalspectrum}\mathrm{Sta}^\mathrm{derivedringed,\sharp}_{\mathrm{sComm}\mathrm{Simplicial}\mathrm{Ind}^m\mathrm{Seminormed}_R,\mathrm{homotopyepi}},\\
&\mathrm{Proj}^\mathrm{formalspectrum}\mathrm{Sta}^\mathrm{derivedringed,\sharp}_{\mathrm{sComm}\mathrm{Simplicial}\mathrm{Ind}\mathrm{Normed}_R,\mathrm{homotopyepi}},\\
&\mathrm{Proj}^\mathrm{formalspectrum}\mathrm{Sta}^\mathrm{derivedringed,\sharp}_{\mathrm{sComm}\mathrm{Simplicial}\mathrm{Ind}^m\mathrm{Normed}_R,\mathrm{homotopyepi}},\\
&\mathrm{Proj}^\mathrm{formalspectrum}\mathrm{Sta}^\mathrm{derivedringed,\sharp}_{\mathrm{sComm}\mathrm{Simplicial}\mathrm{Ind}\mathrm{Banach}_R,\mathrm{homotopyepi}},\\
&\mathrm{Proj}^\mathrm{formalspectrum}\mathrm{Sta}^\mathrm{derivedringed,\sharp}_{\mathrm{sComm}\mathrm{Simplicial}\mathrm{Ind}^m\mathrm{Banach}_R,\mathrm{homotopyepi}},	
\end{align}
we define the corresponding generalized prismatic cohomology presheaf:
\begin{align}
\mathrm{Prism}^{\sqrt{I}}_{-/P,\mathrm{BBM},\mathrm{analytification}}(\mathcal{R})
\end{align}
as:
\begin{align}
&\mathrm{Prism}^{\sqrt{I}}_{-/P,\mathrm{BBM},\mathrm{analytification}}(\mathcal{R})\\
&:=\left[\left(\underset{n}{\mathrm{homotopycolimit}}~ (\mathrm{Prism}_{-/P,\mathrm{BBM},\mathrm{formalanalytification}}(\mathcal{R}_n)[\sqrt{I}])^\wedge_{p,\sqrt{I}}\right)^\wedge_{p,\sqrt{I}}\right]_{\mathrm{BBM},\mathrm{formalanalytification}}	
\end{align}
where the notation means we take the corresponding derived $(p,\sqrt{I})$-completion, then we take the corresponding formal series analytification from \cite[4.2]{BBM}.\\

\end{definition}

\newpage
\section{Almost Mixed-Parity Robba Stacks in the Commutative Algebra Situations}

\begin{reference}
\cite{KL1}, \cite{KL2}, \cite{Sch1}, \cite{Sch}, \cite{Fon}, \cite{FF}, \cite{F1}, \cite{Ta}.
\end{reference}

Now we consider the construction from \cite{KL1} and \cite{KL2}, and apply the functors in \cite[Definition 9.3.3, Definition 9.3.5, Definition 9.3.11, Definition 9.3.9]{KL1} and \cite{KL2} to the rings and spaces in our current $\infty$-categorical context. Now let $R$ be any analytic field $\mathcal{K}$ over $\mathbb{Q}_p$\footnote{After essential deformation and descend back we have the p-adic $2\pi\sqrt{-1}$-element in the p-adic world, which is denoted in our scenario $k$.}. Recall from \cite[Definition 9.3.3, Definition 9.3.5, Definition 9.3.11, Definition 9.3.9]{KL1} we have the following functors:
\begin{align}
\widetilde{\mathcal{C}}_{-/R}(.),{\mathbb{B}_e}_{-/R}(.),{\mathbb{B}_\mathrm{dR}^+}_{-/R}(.),{\mathbb{B}_\mathrm{dR}}_{-/R}(.),{FF}_{-/R}(.)	
\end{align}
on the following rings:
\begin{align}
R\left<X_1,...,X_n\right>,n=0,1,2,...	
\end{align}
We then have the situation to promote the functors of rings and stacks to the $\infty$-categorical context as above. Here recall that from \cite[Definition 9.3.3, Definition 9.3.5, Definition 9.3.11, Definition 9.3.9]{KL1} and \cite{KL2}:

\begin{definition}
For any $R\left<X_1,...,X_n\right>$, we have by taking the global section\footnote{Here the corresponding notation $\left<\log(k)\right>$ means the corresponding formal completion with respect to the variable $\log(k)$. ${\square\square\square}{{FF}}$ is then defined by using for instance the stacks in \cite{BK} and \cite{BBBK} when we consider the corresponding foundation \cite{BK} and \cite{BBBK}, or \cite{CS2} when we consider the corresponding foundation in \cite{CS2} by using the analytic rings therein, or \cite{CS3} when we consider the corresponding foundation in \cite{CS3} by using the light analytic rings therein.}:
\begin{align}
{\square\square\square}\widetilde{\mathcal{C}}_{-/R}(.)(R\left<X_1,...,X_n\right>):=\widetilde{\mathcal{C}}_{\mathrm{Spa}R\left<X_1,...,X_n\right>/R,\text{pro\'et}}(\mathrm{Spa}R\left<X_1,...,X_n\right>/R,\text{pro\'et})[k^{1/2}]\left<\log(k)\right>\\
{\square\square\square}{\mathbb{B}_e}_{-/R}(.)(R\left<X_1,...,X_n\right>):={\mathbb{B}_e}_{\mathrm{Spa}R\left<X_1,...,X_n\right>/R,\text{pro\'et}}(\mathrm{Spa}R\left<X_1,...,X_n\right>/R,\text{pro\'et})[k^{-1/2}]\left<\log(k)\right>,\\
{\square\square\square}{\mathbb{B}_\mathrm{dR}^+}_{-/R}(.)(R\left<X_1,...,X_n\right>):={\mathbb{B}_\mathrm{dR}^+}_{\mathrm{Spa}R\left<X_1,...,X_n\right>/R,\text{pro\'et}}(\mathrm{Spa}R\left<X_1,...,X_n\right>/R,\text{pro\'et})[k^{1/2}]\left<\log(k)\right>,\\
{\square\square\square}{\mathbb{B}_\mathrm{dR}}_{-/R}(.)(R\left<X_1,...,X_n\right>):={\mathbb{B}_\mathrm{dR}}_{\mathrm{Spa}R\left<X_1,...,X_n\right>/R,\text{pro\'et}}(\mathrm{Spa}R\left<X_1,...,X_n\right>/R,\text{pro\'et})[k^{1/2}]\left<\log(k)\right>,\\
{\square\square\square}{{FF}}_{-/R}(.)(R\left<X_1,...,X_n\right>):={\square\square\square}{FF}_{\mathrm{Spa}R\left<X_1,...,X_n\right>/R,\text{pro\'et}}(\mathrm{Spa}R\left<X_1,...,X_n\right>/R,\text{pro\'et}).
\end{align}
And from \cite[Definition 9.3.3, Definition 9.3.5, Definition 9.3.11, Definition 9.3.9]{KL1} and \cite{KL2} we have the notation of $\varphi$-modules, ${\square\square\square}B$-pairs and the vector bundles over the ${\square\square\square}FF$ curves as above. We now use the notation $M$ to denote them. We then put $V(.)(R\left<X_1,...,X_n\right>):=M(\mathrm{Spa}R\left<X_1,...,X_n\right>/R,\text{pro\'et})$. For the stack ${\square\square\square}FF$, this $V$ will be a corresponding vector bundle at the end over ${\square\square\square}FF_{(R\left<X_1,...,X_n\right>)^\flat/R}$. 	
\end{definition}

\begin{definition}
Following \cite[Definition 9.3.3, Definition 9.3.5, Definition 9.3.11, Definition 9.3.9]{KL1}, \cite{KL2} we give the following definition. For any ring
\begin{align}
\mathcal{R}=\underset{n}{\mathrm{homotopycolimit}}\mathcal{R}_n	
\end{align}
in the $\infty$-categories:
\begin{align}
&\mathrm{sComm}\mathrm{Simplicial}\mathrm{Ind}\mathrm{Seminormed}^\mathrm{formalseriescolimitcomp}_R,\\
&\mathrm{sComm}\mathrm{Simplicial}\mathrm{Ind}^m\mathrm{Seminormed}^\mathrm{formalseriescolimitcomp}_R,\\
&\mathrm{sComm}\mathrm{Simplicial}\mathrm{Ind}\mathrm{Normed}^\mathrm{formalseriescolimitcomp}_R,\\
&\mathrm{sComm}\mathrm{Simplicial}\mathrm{Ind}^m\mathrm{Normed}^\mathrm{formalseriescolimitcomp}_R,\\
&\mathrm{sComm}\mathrm{Simplicial}\mathrm{Ind}\mathrm{Banach}^\mathrm{formalseriescolimitcomp}_R,\\
&\mathrm{sComm}\mathrm{Simplicial}\mathrm{Ind}^m\mathrm{Banach}^\mathrm{formalseriescolimitcomp}_R.	
\end{align}	
we define the corresponding $\infty$-functors:
\begin{align}
{\square\square\square}\widetilde{\mathcal{C}}_{-/R}(.),{\square\square\square}{\mathbb{B}_e}_{-/R}(.),{\square\square\square}{\mathbb{B}_\mathrm{dR}^+}_{-/R}(.),{\square\square\square}{\mathbb{B}_\mathrm{dR}}_{-/R}(.),{\square\square\square}{FF}_{-/R}(.),{\square\square\square}V(.)	
\end{align}
as:
\begin{align}
&{\square\square\square}\widetilde{\mathcal{C}}_{-/R}(.)(\mathcal{R}):=\underset{n}{\mathrm{homotopycolimit}}~\widetilde{\mathcal{C}}_{-/R}(\mathcal{R}_n)[k^{1/2}]\left<\log(k)\right>,\\
&{\square\square\square}{\mathbb{B}_e}_{-/R}(.)(\mathcal{R}):=\underset{n}{\mathrm{homotopycolimit}}~{\mathbb{B}_e}_{-/R}(.)(\mathcal{R}_n)[k^{-1/2}]\left<\log(k)\right>,\\
&{\square\square\square}{\mathbb{B}_\mathrm{dR}^+}_{-/R}(.)(\mathcal{R}):=\underset{n}{\mathrm{homotopycolimit}}~{\mathbb{B}_\mathrm{dR}^+}_{-/R}(.)(\mathcal{R}_n)[k^{1/2}]\left<\log(k)\right>,\\
&{\square\square\square}{\mathbb{B}_\mathrm{dR}}_{-/R}(.)(\mathcal{R}):=\underset{n}{\mathrm{homotopycolimit}}~{\mathbb{B}_\mathrm{dR}}_{-/R}(.)(\mathcal{R}_n)[k^{1/2}]\left<\log(k)\right>,\\
&{\square\square\square}{{FF}}_{-/R}(.)(\mathcal{R}):=\underset{n}{\mathrm{homotopylimit}}~{\square\square\square}{{FF}}_{-/R}(.)(\mathcal{R}_n),\\	
&{\square\square\square}V(.)(\mathcal{R}):=\underset{n}{\mathrm{homotopycolimit}}~{\square\square\square}V(.)(\mathcal{R}_n).
\end{align}
Here the homotopy colimits are taken in the corresponding colimit completions of the categories where the rings and spaces are living. Here the homotopy limits are taken in the corresponding limit completions of the categories where the rings and spaces are living.
\end{definition}

\indent Now motivated also by \cite{M} after \cite{CS1}, \cite{CS2} and \cite{CS3} we consider the condensed mathematical version of the construction above. Namely we look at the corresponding Clausen-Scholze's animated enhancement of the corresponding analytic condensed solid commutative algebras:
\begin{align}
\mathrm{AnalyticRings}^\mathrm{CS}_R.	
\end{align}
And we consider the corresponding colimit closure of the formal series. We denote the corresponding $\infty$-category as:
\begin{align}
\mathrm{AnalyticRings}^\mathrm{CS,formalcolimitclosure}_R.	
\end{align}

\begin{definition}
Following \cite[Definition 9.3.3, Definition 9.3.5, Definition 9.3.11, Definition 9.3.9]{KL1}, \cite{KL2} we give the following definition. For any ring
\begin{align}
\mathcal{R}=\underset{n}{\mathrm{homotopycolimit}}\mathcal{R}_n	
\end{align}
in the $\infty$-category:
\begin{align}
\mathrm{AnalyticRings}^\mathrm{CS,formalcolimitclosure}_R,	
\end{align}	
we define the corresponding $\infty$-functors:
\begin{align}
{\square\square\square}\widetilde{\mathcal{C}}_{-/R}(.),{\square\square\square}{\mathbb{B}_e}_{-/R}(.),{\square\square\square}{\mathbb{B}_\mathrm{dR}^+}_{-/R}(.),{\square\square\square}{\mathbb{B}_\mathrm{dR}}_{-/R}(.),{\square\square\square}{FF}_{-/R}(.),{\square\square\square}V(.)	
\end{align}
as:
\begin{align}
&{\square\square\square}\widetilde{\mathcal{C}}_{-/R}(.)(\mathcal{R}):=\underset{n}{\mathrm{homotopycolimit}}~\widetilde{\mathcal{C}}_{-/R}(\mathcal{R}_n)[k^{1/2}]\left<\log(k)\right>,\\
&{\square\square\square}{\mathbb{B}_e}_{-/R}(.)(\mathcal{R}):=\underset{n}{\mathrm{homotopycolimit}}~{\mathbb{B}_e}_{-/R}(.)(\mathcal{R}_n)[k^{-1/2}]\left<\log(k)\right>,\\
&{\square\square\square}{\mathbb{B}_\mathrm{dR}^+}_{-/R}(.)(\mathcal{R}):=\underset{n}{\mathrm{homotopycolimit}}~{\mathbb{B}_\mathrm{dR}^+}_{-/R}(.)(\mathcal{R}_n)[k^{1/2}]\left<\log(k)\right>,\\
&{\square\square\square}{\mathbb{B}_\mathrm{dR}}_{-/R}(.)(\mathcal{R}):=\underset{n}{\mathrm{homotopycolimit}}~{\mathbb{B}_\mathrm{dR}}_{-/R}(.)(\mathcal{R}_n)[k^{1/2}]\left<\log(k)\right>,\\
&{\square\square\square}{{FF}}_{-/R}(.)(\mathcal{R}):=\underset{n}{\mathrm{homotopylimit}}~{\square\square\square}{{FF}}_{-/R}(.)(\mathcal{R}_n),\\	
&{\square\square\square}V(.)(\mathcal{R}):=\underset{n}{\mathrm{homotopycolimit}}~{\square\square\square}V(.)(\mathcal{R}_n).
\end{align}
Here the homotopy colimits are taken in the corresponding colimit completions of the categories where the rings and spaces are living. Here the homotopy limits are taken in the corresponding limit completions of the categories where the rings and spaces are living.
\end{definition}

Then the following is a direct conjectural consequence of  \cite[Theorem 9.3.12]{KL1}.

\begin{conjecture}
The corresponding $\infty$-categories of $\varphi$-module functors, ${\square\square\square}B$-pair functors and vector bundles functors over ${\square\square\square}FF$ functors are equivalent over:
\begin{align}
&\mathrm{sComm}\mathrm{Simplicial}\mathrm{Ind}\mathrm{Seminormed}^\mathrm{formalseriescolimitcomp}_R,\\
&\mathrm{sComm}\mathrm{Simplicial}\mathrm{Ind}^m\mathrm{Seminormed}^\mathrm{formalseriescolimitcomp}_R,\\
&\mathrm{sComm}\mathrm{Simplicial}\mathrm{Ind}\mathrm{Normed}^\mathrm{formalseriescolimitcomp}_R,\\
&\mathrm{sComm}\mathrm{Simplicial}\mathrm{Ind}^m\mathrm{Normed}^\mathrm{formalseriescolimitcomp}_R,\\
&\mathrm{sComm}\mathrm{Simplicial}\mathrm{Ind}\mathrm{Banach}^\mathrm{formalseriescolimitcomp}_R,\\
&\mathrm{sComm}\mathrm{Simplicial}\mathrm{Ind}^m\mathrm{Banach}^\mathrm{formalseriescolimitcomp}_R,	
\end{align}
or:
\begin{align}
\mathrm{AnalyticRings}^\mathrm{CS,formalcolimitclosure}_R.	
\end{align}	
\end{conjecture}

\newpage
\section{Almost Mixed-Parity Robba Stacks in the Ringed Topos Situations}

\begin{reference}
\cite{KL1}, \cite{KL2}, \cite{Sch1}, \cite{Sch}, \cite{Fon}, \cite{FF}, \cite{F1}, \cite{Ta}.
\end{reference}

Now we consider the construction from \cite[Definition 9.3.3, Definition 9.3.5, Definition 9.3.11, Definition 9.3.9]{KL1} and \cite{KL2}, and apply the functors in \cite[Definition 9.3.3, Definition 9.3.5, Definition 9.3.11, Definition 9.3.9]{KL1} and \cite{KL2} to the rings and spaces in our current $\infty$-categorical context. Now let $R$ be any analytic field $\mathcal{K}$ over $\mathbb{Q}_p$. Recall from \cite[Definition 9.3.3, Definition 9.3.5, Definition 9.3.11, Definition 9.3.9]{KL1} we have the following functors:
\begin{align}
{\square\square\square}\widetilde{\mathcal{C}}_{-/R}(.),{\square\square\square}{\mathbb{B}_e}_{-/R}(.),{\square\square\square}{\mathbb{B}_\mathrm{dR}^+}_{-/R}(.),{\square\square\square}{\mathbb{B}_\mathrm{dR}}_{-/R}(.),{\square\square\square}{FF}_{-/R}(.)	
\end{align}
on the following ringed spaces from \cite{BK}:
\begin{align}
(\mathrm{Spa}^\mathrm{BK}R\left<X_1,...,X_n\right>,\mathcal{O}_{\mathrm{Spa}^\mathrm{BK}R\left<X_1,...,X_n\right>}),n=0,1,2,...	
\end{align}
We then have the situation to promote the functors of rings and stacks to the $\infty$-categorical context as above. Here recall that from \cite[Definition 9.3.3, Definition 9.3.5, Definition 9.3.11, Definition 9.3.9]{KL1} and \cite{KL2}:

\begin{definition}
For any $\mathrm{Spa}^\mathrm{BK}R\left<X_1,...,X_n\right>$, we have by taking the global section:
\begin{align}
{\square\square\square}\widetilde{\mathcal{C}}_{-/R}(.)(\mathcal{O}_{\mathrm{Spa}^\mathrm{BK}R\left<X_1,...,X_n\right>})&(\mathrm{Spa}^\mathrm{BK}R\left<X_1,...,X_m\right>):=\\
&\widetilde{\mathcal{C}}_{\mathrm{Spa}R\left<X_1,...,X_m\right>/R,\text{pro\'et}}(\mathrm{Spa}R\left<X_1,...,X_m\right>/R,\text{pro\'et})[k^{1/2}]\left<\log(k)\right>\\
{\square\square\square}{\mathbb{B}_e}_{-/R}(.)(\mathcal{O}_{\mathrm{Spa}^\mathrm{BK}R\left<X_1,...,X_n\right>})&(\mathrm{Spa}^\mathrm{BK}R\left<X_1,...,X_m\right>):=\\
&{\mathbb{B}_e}_{\mathrm{Spa}R\left<X_1,...,X_m\right>/R,\text{pro\'et}}(\mathrm{Spa}R\left<X_1,...,X_m\right>/R,\text{pro\'et})[k^{-1/2}]\left<\log(k)\right>,\\
{\square\square\square}{\mathbb{B}_\mathrm{dR}^+}_{-/R}(.)(\mathcal{O}_{\mathrm{Spa}^\mathrm{BK}R\left<X_1,...,X_n\right>})&(\mathrm{Spa}^\mathrm{BK}R\left<X_1,...,X_m\right>):=\\
&{\mathbb{B}_\mathrm{dR}^+}_{\mathrm{Spa}R\left<X_1,...,X_m\right>/R,\text{pro\'et}}(\mathrm{Spa}R\left<X_1,...,X_m\right>/R,\text{pro\'et})[k^{1/2}]\left<\log(k)\right>,\\
{\square\square\square}{\mathbb{B}_\mathrm{dR}}_{-/R}(.)(\mathcal{O}_{\mathrm{Spa}^\mathrm{BK}R\left<X_1,...,X_n\right>})&(\mathrm{Spa}^\mathrm{BK}R\left<X_1,...,X_m\right>):=\\
&{\mathbb{B}_\mathrm{dR}}_{\mathrm{Spa}R\left<X_1,...,X_m\right>/R,\text{pro\'et}}(\mathrm{Spa}R\left<X_1,...,X_m\right>/R,\text{pro\'et})[k^{1/2}]\left<\log(k)\right>,\\
{\square\square\square}{{FF}}_{-/R}(.)(\mathcal{O}_{\mathrm{Spa}^\mathrm{BK}R\left<X_1,...,X_n\right>})&(\mathrm{Spa}^\mathrm{BK}R\left<X_1,...,X_m\right>):=\\
&{\square\square\square}{FF}_{\mathrm{Spa}R\left<X_1,...,X_m\right>/R,\text{pro\'et}}(\mathrm{Spa}R\left<X_1,...,X_m\right>/R,\text{pro\'et}).
\end{align}
And from \cite[Definition 9.3.3, Definition 9.3.5, Definition 9.3.11, Definition 9.3.9]{KL1} and \cite{KL2} we have the notation of $\varphi$-modules, ${\square\square\square}B$-pairs and the vector bundles over the ${\square\square\square}FF$ curves as above. We now use the notation $M$ to denote them. We then put $V(.)(\mathcal{O}_{\mathrm{Spa}^\mathrm{BK}R\left<X_1,...,X_n\right>})(\mathrm{Spa}^\mathrm{BK}R\left<X_1,...,X_m\right>):=M(\mathrm{Spa}R\left<X_1,...,X_m\right>/R,\text{pro\'et})$. For the stack ${\square\square\square}FF$, this $V$ will be a corresponding vector bundle at the end over ${\square\square\square}FF_{(R\left<X_1,...,X_n\right>)^\flat/R}$. 	
\end{definition}

\begin{definition}
Following \cite[Definition 9.3.3, Definition 9.3.5, Definition 9.3.11, Definition 9.3.9]{KL1}, \cite{KL2} we give the following definition. For any space
\begin{align}
(\mathbb{X},\mathcal{R})=\underset{n}{\mathrm{homotopylimit}}(\mathbb{X}_n,\mathcal{R}_n)	
\end{align}
in the $\infty$-categories:
\begin{align}
&\mathrm{Proj}^\mathrm{formalspectrum}\mathrm{Sta}^\mathrm{derivedringed,\sharp}_{\mathrm{sComm}\mathrm{Simplicial}\mathrm{Ind}\mathrm{Seminormed}_R,\mathrm{homotopyepi}},\\
&\mathrm{Proj}^\mathrm{formalspectrum}\mathrm{Sta}^\mathrm{derivedringed,\sharp}_{\mathrm{sComm}\mathrm{Simplicial}\mathrm{Ind}^m\mathrm{Seminormed}_R,\mathrm{homotopyepi}},\\
&\mathrm{Proj}^\mathrm{formalspectrum}\mathrm{Sta}^\mathrm{derivedringed,\sharp}_{\mathrm{sComm}\mathrm{Simplicial}\mathrm{Ind}\mathrm{Normed}_R,\mathrm{homotopyepi}},\\
&\mathrm{Proj}^\mathrm{formalspectrum}\mathrm{Sta}^\mathrm{derivedringed,\sharp}_{\mathrm{sComm}\mathrm{Simplicial}\mathrm{Ind}^m\mathrm{Normed}_R,\mathrm{homotopyepi}},\\
&\mathrm{Proj}^\mathrm{formalspectrum}\mathrm{Sta}^\mathrm{derivedringed,\sharp}_{\mathrm{sComm}\mathrm{Simplicial}\mathrm{Ind}\mathrm{Banach}_R,\mathrm{homotopyepi}},\\
&\mathrm{Proj}^\mathrm{formalspectrum}\mathrm{Sta}^\mathrm{derivedringed,\sharp}_{\mathrm{sComm}\mathrm{Simplicial}\mathrm{Ind}^m\mathrm{Banach}_R,\mathrm{homotopyepi}},	
\end{align}
we define the corresponding $\infty$-functors:
\begin{align}
{\square\square}\widetilde{\mathcal{C}}_{-/R}(.),{\square\square}{\mathbb{B}_e}_{-/R}(.),{\square\square}{\mathbb{B}_\mathrm{dR}^+}_{-/R}(.),{\square\square}{\mathbb{B}_\mathrm{dR}}_{-/R}(.),{\square\square}{FF}_{-/R}(.),V(.)	
\end{align}
as:
\begin{align}
&{\square\square}\widetilde{\mathcal{C}}_{-/R}(.)(\mathcal{R}):=\underset{n}{\mathrm{homotopycolimit}}~{\square\square}\widetilde{\mathcal{C}}_{-/R}(\mathcal{R}_n),\\
&{\square\square}{\mathbb{B}_e}_{-/R}(.)(\mathcal{R}):=\underset{n}{\mathrm{homotopycolimit}}~{\square\square}{\mathbb{B}_e}_{-/R}(.)(\mathcal{R}_n),\\
&{\square\square}{\mathbb{B}_\mathrm{dR}^+}_{-/R}(.)(\mathcal{R}):=\underset{n}{\mathrm{homotopycolimit}}~{\square\square}{\mathbb{B}_\mathrm{dR}^+}_{-/R}(.)(\mathcal{R}_n),\\
&{\square\square}{\mathbb{B}_\mathrm{dR}}_{-/R}(.)(\mathcal{R}):=\underset{n}{\mathrm{homotopycolimit}}~{\square\square}{\mathbb{B}_\mathrm{dR}}_{-/R}(.)(\mathcal{R}_n),\\
&{\square\square}{{FF}}_{-/R}(.)(\mathcal{R}):=\underset{n}{\mathrm{homotopylimit}}~{\square\square}{\square\square}{{FF}}_{-/R}(.)(\mathcal{R}_n),\\	
&{\square\square}V(.)(\mathcal{R}):=\underset{n}{\mathrm{homotopycolimit}}~{\square\square}V(.)(\mathcal{R}_n).
\end{align}
Here the homotopy colimits are taken in the corresponding colimit completions of the categories where the rings and spaces are living. Here the homotopy limits are taken in the corresponding limit completions of the categories where the rings and spaces are living.
\end{definition}

\newpage
\section{Almost Mixed-Parity Robba Stacks in the Inductive System Situations}

\begin{reference}
\cite{KL1}, \cite{KL2}, \cite{Sch1}, \cite{Sch}, \cite{Fon}, \cite{FF}, \cite{F1}, \cite{Ta}.
\end{reference}

Now we consider the construction from \cite[Definition 9.3.3, Definition 9.3.5, Definition 9.3.11, Definition 9.3.9]{KL1} and \cite{KL2}, and apply the functors in \cite[Definition 9.3.3, Definition 9.3.5, Definition 9.3.11, Definition 9.3.9]{KL1} and \cite{KL2} to the rings and spaces in our current $\infty$-categorical context. Now let $R$ be any analytic field $\mathcal{K}$ over $\mathbb{Q}_p$. Recall from \cite[Definition 9.3.3, Definition 9.3.5, Definition 9.3.11, Definition 9.3.9]{KL1} we have the following functors:
\begin{align}
{\square\square\square}\widetilde{\mathcal{C}}_{-/R}(.),{\square\square\square}{\mathbb{B}_e}_{-/R}(.),{\square\square\square}{\mathbb{B}_\mathrm{dR}^+}_{-/R}(.),{\square\square\square}{\mathbb{B}_\mathrm{dR}}_{-/R}(.),{\square\square\square}{FF}_{-/R}(.)	
\end{align}
on the following ringed spaces from \cite{BK}:
\begin{align}
(\mathrm{Spa}^\mathrm{BK}R\left<X_1,...,X_n\right>,\mathcal{O}_{\mathrm{Spa}^\mathrm{BK}R\left<X_1,...,X_n\right>}),n=0,1,2,...	
\end{align}
We then have the situation to promote the functors of rings and stacks to the $\infty$-categorical context as above. Here recall that from \cite[Definition 9.3.3, Definition 9.3.5, Definition 9.3.11, Definition 9.3.9]{KL1} and \cite{KL2}:

\begin{definition}
For any $\mathrm{Spa}^\mathrm{BK}R\left<X_1,...,X_n\right>$, we have by taking the global section:
\begin{align}
{\square\square\square}\widetilde{\mathcal{C}}_{-/R}(.)(\mathcal{O}_{\mathrm{Spa}^\mathrm{BK}R\left<X_1,...,X_n\right>})&(\mathrm{Spa}^\mathrm{BK}R\left<X_1,...,X_m\right>):=\\
&\widetilde{\mathcal{C}}_{\mathrm{Spa}R\left<X_1,...,X_m\right>/R,\text{pro\'et}}(\mathrm{Spa}R\left<X_1,...,X_m\right>/R,\text{pro\'et})[k^{1/2}]\left<\log(k)\right>\\
{\square\square\square}{\mathbb{B}_e}_{-/R}(.)(\mathcal{O}_{\mathrm{Spa}^\mathrm{BK}R\left<X_1,...,X_n\right>})&(\mathrm{Spa}^\mathrm{BK}R\left<X_1,...,X_m\right>):=\\
&{\mathbb{B}_e}_{\mathrm{Spa}R\left<X_1,...,X_m\right>/R,\text{pro\'et}}(\mathrm{Spa}R\left<X_1,...,X_m\right>/R,\text{pro\'et})[k^{-1/2}]\left<\log(k)\right>,\\
{\square\square\square}{\mathbb{B}_\mathrm{dR}^+}_{-/R}(.)(\mathcal{O}_{\mathrm{Spa}^\mathrm{BK}R\left<X_1,...,X_n\right>})&(\mathrm{Spa}^\mathrm{BK}R\left<X_1,...,X_m\right>):=\\
&{\mathbb{B}_\mathrm{dR}^+}_{\mathrm{Spa}R\left<X_1,...,X_m\right>/R,\text{pro\'et}}(\mathrm{Spa}R\left<X_1,...,X_m\right>/R,\text{pro\'et})[k^{1/2}]\left<\log(k)\right>,\\
{\square\square\square}{\mathbb{B}_\mathrm{dR}}_{-/R}(.)(\mathcal{O}_{\mathrm{Spa}^\mathrm{BK}R\left<X_1,...,X_n\right>})&(\mathrm{Spa}^\mathrm{BK}R\left<X_1,...,X_m\right>):=\\
&{\mathbb{B}_\mathrm{dR}}_{\mathrm{Spa}R\left<X_1,...,X_m\right>/R,\text{pro\'et}}(\mathrm{Spa}R\left<X_1,...,X_m\right>/R,\text{pro\'et})[k^{1/2}]\left<\log(k)\right>,\\
{\square\square\square}{{FF}}_{-/R}(.)(\mathcal{O}_{\mathrm{Spa}^\mathrm{BK}R\left<X_1,...,X_n\right>})&(\mathrm{Spa}^\mathrm{BK}R\left<X_1,...,X_m\right>):=\\
&{\square\square\square}{FF}_{\mathrm{Spa}R\left<X_1,...,X_m\right>/R,\text{pro\'et}}(\mathrm{Spa}R\left<X_1,...,X_m\right>/R,\text{pro\'et}).
\end{align}
And from \cite[Definition 9.3.3, Definition 9.3.5, Definition 9.3.11, Definition 9.3.9]{KL1} and \cite{KL2} we have the notation of $\varphi$-modules, ${\square\square\square}B$-pairs and the vector bundles over the ${\square\square\square}FF$ curves as above. We now use the notation $M$ to denote them. We then put $V(.)(\mathcal{O}_{\mathrm{Spa}^\mathrm{BK}R\left<X_1,...,X_n\right>})(\mathrm{Spa}^\mathrm{BK}R\left<X_1,...,X_m\right>):=M(\mathrm{Spa}R\left<X_1,...,X_m\right>/R,\text{pro\'et})$. For the stack ${\square\square\square}FF$, this $V$ will be a corresponding vector bundle at the end over ${\square\square\square}FF_{(R\left<X_1,...,X_n\right>)^\flat/R}$. 	
\end{definition}

\begin{definition}
Following \cite[Definition 9.3.3, Definition 9.3.5, Definition 9.3.11, Definition 9.3.9]{KL1}, \cite{KL2} we give the following definition. For any space
\begin{align}
(\mathbb{X},\mathcal{R})=\underset{n}{\mathrm{homotopycolimit}}(\mathbb{X}_n,\mathcal{R}_n)	
\end{align}
in the $\infty$-categories:
\begin{align}
&\mathrm{Ind}^\mathrm{formalspectrum}\mathrm{Sta}^\mathrm{derivedringed,\sharp}_{\mathrm{sComm}\mathrm{Simplicial}\mathrm{Ind}\mathrm{Seminormed}_R,\mathrm{homotopyepi}},\\
&\mathrm{Ind}^\mathrm{formalspectrum}\mathrm{Sta}^\mathrm{derivedringed,\sharp}_{\mathrm{sComm}\mathrm{Simplicial}\mathrm{Ind}^m\mathrm{Seminormed}_R,\mathrm{homotopyepi}},\\
&\mathrm{Ind}^\mathrm{formalspectrum}\mathrm{Sta}^\mathrm{derivedringed,\sharp}_{\mathrm{sComm}\mathrm{Simplicial}\mathrm{Ind}\mathrm{Normed}_R,\mathrm{homotopyepi}},\\
&\mathrm{Ind}^\mathrm{formalspectrum}\mathrm{Sta}^\mathrm{derivedringed,\sharp}_{\mathrm{sComm}\mathrm{Simplicial}\mathrm{Ind}^m\mathrm{Normed}_R,\mathrm{homotopyepi}},\\
&\mathrm{Ind}^\mathrm{formalspectrum}\mathrm{Sta}^\mathrm{derivedringed,\sharp}_{\mathrm{sComm}\mathrm{Simplicial}\mathrm{Ind}\mathrm{Banach}_R,\mathrm{homotopyepi}},\\
&\mathrm{Ind}^\mathrm{formalspectrum}\mathrm{Sta}^\mathrm{derivedringed,\sharp}_{\mathrm{sComm}\mathrm{Simplicial}\mathrm{Ind}^m\mathrm{Banach}_R,\mathrm{homotopyepi}},	
\end{align}

we define the corresponding $\infty$-functors:
\begin{align}
{\square\square\square}\widetilde{\mathcal{C}}_{-/R}(.),{\square\square\square}{\mathbb{B}_e}_{-/R}(.),{\square\square\square}{\mathbb{B}_\mathrm{dR}^+}_{-/R}(.),{\square\square\square}{\mathbb{B}_\mathrm{dR}}_{-/R}(.),{\square\square\square}{FF}_{-/R}(.),V(.)	
\end{align}
as:
\begin{align}
&{\square\square\square}\widetilde{\mathcal{C}}_{-/R}(.)(\mathcal{R}):=\underset{n}{\mathrm{homotopylimit}}~{\square\square\square}\widetilde{\mathcal{C}}_{-/R}(\mathcal{R}_n),\\
&{\square\square\square}{\mathbb{B}_e}_{-/R}(.)(\mathcal{R}):=\underset{n}{\mathrm{homotopylimit}}~{\square\square\square}{\mathbb{B}_e}_{-/R}(.)(\mathcal{R}_n),\\
&{\square\square\square}{\mathbb{B}_\mathrm{dR}^+}_{-/R}(.)(\mathcal{R}):=\underset{n}{\mathrm{homotopylimit}}~{\square\square\square}{\mathbb{B}_\mathrm{dR}^+}_{-/R}(.)(\mathcal{R}_n),\\
&{\square\square\square}{\mathbb{B}_\mathrm{dR}}_{-/R}(.)(\mathcal{R}):=\underset{n}{\mathrm{homotopylimit}}~{\square\square\square}{\mathbb{B}_\mathrm{dR}}_{-/R}(.)(\mathcal{R}_n),\\
&{\square\square\square}{{FF}}_{-/R}(.)(\mathcal{R}):=\underset{n}{\mathrm{homotopycolimit}}~{\square\square\square}{{FF}}_{-/R}(.)(\mathcal{R}_n),\\	
&{\square\square\square}V(.)(\mathcal{R}):=\underset{n}{\mathrm{homotopylimit}}~{\square\square\square}V(.)(\mathcal{R}_n).
\end{align}
\end{definition}

\

\begin{remark}
Here the homotopy colimits are taken in the corresponding colimit completions of the categories where the rings and spaces are living. Here the homotopy limits are taken in the corresponding limit completions of the categories where the rings and spaces are living. For instance, the Robba functor ${\square\square\square}\widetilde{\mathcal{C}}_{-/R}(.)$ takes value in ind-Fr\'echet rings $\mathrm{Ind}\text{Fr\'echet}_R$, we then consider the corresponding homotopy limit closure $\overline{\mathrm{Ind}\text{Fr\'echet}}^{\mathrm{homotopylimit}}_R$. For instance, the Fargues-Fontaine stack functors ${\square\square\square}{{FF}}_{-/R}(.)(-)$ take value in the preadic spaces $\mathrm{PreAdic}_R$, then we consider the corresponding homotopy colimit closure $\overline{\mathrm{PreAdic}_R}^{\mathrm{homotopycolimit}}$. 	
\end{remark}

\newpage
\section{Mixed-Parity Robba Stacks in the Commutative Algebra Situations}

\begin{reference}
\cite{KL1}, \cite{KL2}, \cite{Sch1}, \cite{Sch}, \cite{Fon}, \cite{FF}, \cite{F1}, \cite{Ta}.
\end{reference}

Now we consider the construction from \cite{KL1} and \cite{KL2}, and apply the functors in \cite[Definition 9.3.3, Definition 9.3.5, Definition 9.3.11, Definition 9.3.9]{KL1} and \cite{KL2} to the rings and spaces in our current $\infty$-categorical context. Now let $R$ be any analytic field $\mathcal{K}$ over $\mathbb{Q}_p$\footnote{After essential deformation and descend back we have the p-adic $2\pi\sqrt{-1}$-element in the p-adic world, which is denoted in our scenario $k$.}. Recall from \cite[Definition 9.3.3, Definition 9.3.5, Definition 9.3.11, Definition 9.3.9]{KL1} we have the following functors:
\begin{align}
\widetilde{\mathcal{C}}_{-/R}(.),{\mathbb{B}_e}_{-/R}(.),{\mathbb{B}_\mathrm{dR}^+}_{-/R}(.),{\mathbb{B}_\mathrm{dR}}_{-/R}(.),{FF}_{-/R}(.)	
\end{align}
on the following rings:
\begin{align}
R\left<X_1,...,X_n\right>,n=0,1,2,...	
\end{align}
We then have the situation to promote the functors of rings and stacks to the $\infty$-categorical context as above. Here recall that from \cite[Definition 9.3.3, Definition 9.3.5, Definition 9.3.11, Definition 9.3.9]{KL1} and \cite{KL2}:

\begin{definition}
For any $R\left<X_1,...,X_n\right>$, we have by taking the global section:
\begin{align}
{\square\square}\widetilde{\mathcal{C}}_{-/R}(.)(R\left<X_1,...,X_n\right>):=\widetilde{\mathcal{C}}_{\mathrm{Spa}R\left<X_1,...,X_n\right>/R,\text{pro\'et}}(\mathrm{Spa}R\left<X_1,...,X_n\right>/R,\text{pro\'et})[k^{1/2}]\\
{\square\square}{\mathbb{B}_e}_{-/R}(.)(R\left<X_1,...,X_n\right>):={\mathbb{B}_e}_{\mathrm{Spa}R\left<X_1,...,X_n\right>/R,\text{pro\'et}}(\mathrm{Spa}R\left<X_1,...,X_n\right>/R,\text{pro\'et})[k^{-1/2}],\\
{\square\square}{\mathbb{B}_\mathrm{dR}^+}_{-/R}(.)(R\left<X_1,...,X_n\right>):={\mathbb{B}_\mathrm{dR}^+}_{\mathrm{Spa}R\left<X_1,...,X_n\right>/R,\text{pro\'et}}(\mathrm{Spa}R\left<X_1,...,X_n\right>/R,\text{pro\'et})[k^{1/2}],\\
{\square\square}{\mathbb{B}_\mathrm{dR}}_{-/R}(.)(R\left<X_1,...,X_n\right>):={\mathbb{B}_\mathrm{dR}}_{\mathrm{Spa}R\left<X_1,...,X_n\right>/R,\text{pro\'et}}(\mathrm{Spa}R\left<X_1,...,X_n\right>/R,\text{pro\'et})[k^{1/2}],\\
{\square\square}{{FF}}_{-/R}(.)(R\left<X_1,...,X_n\right>):={\square\square}{FF}_{\mathrm{Spa}R\left<X_1,...,X_n\right>/R,\text{pro\'et}}(\mathrm{Spa}R\left<X_1,...,X_n\right>/R,\text{pro\'et}).
\end{align}
And from \cite[Definition 9.3.3, Definition 9.3.5, Definition 9.3.11, Definition 9.3.9]{KL1} and \cite{KL2} we have the notation of $\varphi$-modules, ${\square\square}B$-pairs and the vector bundles over the ${\square\square}FF$ curves as above. We now use the notation $M$ to denote them. We then put $V(.)(R\left<X_1,...,X_n\right>):=M(\mathrm{Spa}R\left<X_1,...,X_n\right>/R,\text{pro\'et})$. For the stack $FF$, this $V$ will be a corresponding vector bundle at the end over $FF_{(R\left<X_1,...,X_n\right>)^\flat/R}$. 	
\end{definition}

\begin{definition}
Following \cite[Definition 9.3.3, Definition 9.3.5, Definition 9.3.11, Definition 9.3.9]{KL1}, \cite{KL2} we give the following definition. For any ring
\begin{align}
\mathcal{R}=\underset{n}{\mathrm{homotopycolimit}}\mathcal{R}_n	
\end{align}
in the $\infty$-categories:
\begin{align}
&\mathrm{sComm}\mathrm{Simplicial}\mathrm{Ind}\mathrm{Seminormed}^\mathrm{formalseriescolimitcomp}_R,\\
&\mathrm{sComm}\mathrm{Simplicial}\mathrm{Ind}^m\mathrm{Seminormed}^\mathrm{formalseriescolimitcomp}_R,\\
&\mathrm{sComm}\mathrm{Simplicial}\mathrm{Ind}\mathrm{Normed}^\mathrm{formalseriescolimitcomp}_R,\\
&\mathrm{sComm}\mathrm{Simplicial}\mathrm{Ind}^m\mathrm{Normed}^\mathrm{formalseriescolimitcomp}_R,\\
&\mathrm{sComm}\mathrm{Simplicial}\mathrm{Ind}\mathrm{Banach}^\mathrm{formalseriescolimitcomp}_R,\\
&\mathrm{sComm}\mathrm{Simplicial}\mathrm{Ind}^m\mathrm{Banach}^\mathrm{formalseriescolimitcomp}_R.	
\end{align}	
we define the corresponding $\infty$-functors:
\begin{align}
{\square\square}\widetilde{\mathcal{C}}_{-/R}(.),{\square\square}{\mathbb{B}_e}_{-/R}(.),{\square\square}{\mathbb{B}_\mathrm{dR}^+}_{-/R}(.),{\square\square}{\mathbb{B}_\mathrm{dR}}_{-/R}(.),{\square\square}{FF}_{-/R}(.),{\square\square}V(.)	
\end{align}
as:
\begin{align}
&{\square\square}\widetilde{\mathcal{C}}_{-/R}(.)(\mathcal{R}):=\underset{n}{\mathrm{homotopycolimit}}~\widetilde{\mathcal{C}}_{-/R}(\mathcal{R}_n)[k^{1/2}],\\
&{\square\square}{\mathbb{B}_e}_{-/R}(.)(\mathcal{R}):=\underset{n}{\mathrm{homotopycolimit}}~{\mathbb{B}_e}_{-/R}(.)(\mathcal{R}_n)[k^{-1/2}],\\
&{\square\square}{\mathbb{B}_\mathrm{dR}^+}_{-/R}(.)(\mathcal{R}):=\underset{n}{\mathrm{homotopycolimit}}~{\mathbb{B}_\mathrm{dR}^+}_{-/R}(.)(\mathcal{R}_n)[k^{1/2}],\\
&{\square\square}{\mathbb{B}_\mathrm{dR}}_{-/R}(.)(\mathcal{R}):=\underset{n}{\mathrm{homotopycolimit}}~{\mathbb{B}_\mathrm{dR}}_{-/R}(.)(\mathcal{R}_n)[k^{1/2}],\\
&{\square\square}{{FF}}_{-/R}(.)(\mathcal{R}):=\underset{n}{\mathrm{homotopylimit}}~{\square\square}{{FF}}_{-/R}(.)(\mathcal{R}_n),\\	
&{\square\square}V(.)(\mathcal{R}):=\underset{n}{\mathrm{homotopycolimit}}~{\square\square}V(.)(\mathcal{R}_n).
\end{align}
Here the homotopy colimits are taken in the corresponding colimit completions of the categories where the rings and spaces are living. Here the homotopy limits are taken in the corresponding limit completions of the categories where the rings and spaces are living.
\end{definition}

\indent Now motivated also by \cite{M} after \cite{CS1}, \cite{CS2} and \cite{CS3} we consider the condensed mathematical version of the construction above. Namely we look at the corresponding Clausen-Scholze's animated enhancement of the corresponding analytic condensed solid commutative algebras:
\begin{align}
\mathrm{AnalyticRings}^\mathrm{CS}_R.	
\end{align}
And we consider the corresponding colimit closure of the formal series. We denote the corresponding $\infty$-category as:
\begin{align}
\mathrm{AnalyticRings}^\mathrm{CS,formalcolimitclosure}_R.	
\end{align}

\begin{definition}
Following \cite[Definition 9.3.3, Definition 9.3.5, Definition 9.3.11, Definition 9.3.9]{KL1}, \cite{KL2} we give the following definition. For any ring
\begin{align}
\mathcal{R}=\underset{n}{\mathrm{homotopycolimit}}\mathcal{R}_n	
\end{align}
in the $\infty$-category:
\begin{align}
\mathrm{AnalyticRings}^\mathrm{CS,formalcolimitclosure}_R,	
\end{align}	
we define the corresponding $\infty$-functors:
\begin{align}
{\square\square}\widetilde{\mathcal{C}}_{-/R}(.),{\square\square}{\mathbb{B}_e}_{-/R}(.),{\square\square}{\mathbb{B}_\mathrm{dR}^+}_{-/R}(.),{\square\square}{\mathbb{B}_\mathrm{dR}}_{-/R}(.),{\square\square}{FF}_{-/R}(.),{\square\square}V(.)	
\end{align}
as:
\begin{align}
&{\square\square}\widetilde{\mathcal{C}}_{-/R}(.)(\mathcal{R}):=\underset{n}{\mathrm{homotopycolimit}}~\widetilde{\mathcal{C}}_{-/R}(\mathcal{R}_n)[k^{1/2}],\\
&{\square\square}{\mathbb{B}_e}_{-/R}(.)(\mathcal{R}):=\underset{n}{\mathrm{homotopycolimit}}~{\mathbb{B}_e}_{-/R}(.)(\mathcal{R}_n)[k^{-1/2}],\\
&{\square\square}{\mathbb{B}_\mathrm{dR}^+}_{-/R}(.)(\mathcal{R}):=\underset{n}{\mathrm{homotopycolimit}}~{\mathbb{B}_\mathrm{dR}^+}_{-/R}(.)(\mathcal{R}_n)[k^{1/2}],\\
&{\square\square}{\mathbb{B}_\mathrm{dR}}_{-/R}(.)(\mathcal{R}):=\underset{n}{\mathrm{homotopycolimit}}~{\mathbb{B}_\mathrm{dR}}_{-/R}(.)(\mathcal{R}_n)[k^{1/2}],\\
&{\square\square}{{FF}}_{-/R}(.)(\mathcal{R}):=\underset{n}{\mathrm{homotopylimit}}~{\square\square}{{FF}}_{-/R}(.)(\mathcal{R}_n),\\	
&{\square\square}V(.)(\mathcal{R}):=\underset{n}{\mathrm{homotopycolimit}}~{\square\square}V(.)(\mathcal{R}_n).
\end{align}
Here the homotopy colimits are taken in the corresponding colimit completions of the categories where the rings and spaces are living. Here the homotopy limits are taken in the corresponding limit completions of the categories where the rings and spaces are living.
\end{definition}

Then the following is a direct consequence of  \cite[Theorem 9.3.12]{KL1}.

\begin{proposition}
The corresponding $\infty$-categories of $\varphi$-module functors, ${\square\square}B$-pair functors and vector bundles functors over ${\square\square}FF$ functors are equivalent over:
\begin{align}
&\mathrm{sComm}\mathrm{Simplicial}\mathrm{Ind}\mathrm{Seminormed}^\mathrm{formalseriescolimitcomp}_R,\\
&\mathrm{sComm}\mathrm{Simplicial}\mathrm{Ind}^m\mathrm{Seminormed}^\mathrm{formalseriescolimitcomp}_R,\\
&\mathrm{sComm}\mathrm{Simplicial}\mathrm{Ind}\mathrm{Normed}^\mathrm{formalseriescolimitcomp}_R,\\
&\mathrm{sComm}\mathrm{Simplicial}\mathrm{Ind}^m\mathrm{Normed}^\mathrm{formalseriescolimitcomp}_R,\\
&\mathrm{sComm}\mathrm{Simplicial}\mathrm{Ind}\mathrm{Banach}^\mathrm{formalseriescolimitcomp}_R,\\
&\mathrm{sComm}\mathrm{Simplicial}\mathrm{Ind}^m\mathrm{Banach}^\mathrm{formalseriescolimitcomp}_R,	
\end{align}
or:
\begin{align}
\mathrm{AnalyticRings}^\mathrm{CS,formalcolimitclosure}_R.	
\end{align}	
\end{proposition}

\newpage
\section{Mixed-Parity Robba Stacks in the Ringed Topos Situations}

\begin{reference}
\cite{KL1}, \cite{KL2}, \cite{Sch1}, \cite{Sch}, \cite{Fon}, \cite{FF}, \cite{F1}, \cite{Ta}.
\end{reference}

Now we consider the construction from \cite[Definition 9.3.3, Definition 9.3.5, Definition 9.3.11, Definition 9.3.9]{KL1} and \cite{KL2}, and apply the functors in \cite[Definition 9.3.3, Definition 9.3.5, Definition 9.3.11, Definition 9.3.9]{KL1} and \cite{KL2} to the rings and spaces in our current $\infty$-categorical context. Now let $R$ be any analytic field $\mathcal{K}$ over $\mathbb{Q}_p$. Recall from \cite[Definition 9.3.3, Definition 9.3.5, Definition 9.3.11, Definition 9.3.9]{KL1} we have the following functors:
\begin{align}
{\square\square}\widetilde{\mathcal{C}}_{-/R}(.),{\square\square}{\mathbb{B}_e}_{-/R}(.),{\square\square}{\mathbb{B}_\mathrm{dR}^+}_{-/R}(.),{\square\square}{\mathbb{B}_\mathrm{dR}}_{-/R}(.),{\square\square}{FF}_{-/R}(.)	
\end{align}
on the following ringed spaces from \cite{BK}:
\begin{align}
(\mathrm{Spa}^\mathrm{BK}R\left<X_1,...,X_n\right>,\mathcal{O}_{\mathrm{Spa}^\mathrm{BK}R\left<X_1,...,X_n\right>}),n=0,1,2,...	
\end{align}
We then have the situation to promote the functors of rings and stacks to the $\infty$-categorical context as above. Here recall that from \cite[Definition 9.3.3, Definition 9.3.5, Definition 9.3.11, Definition 9.3.9]{KL1} and \cite{KL2}:

\begin{definition}
For any $\mathrm{Spa}^\mathrm{BK}R\left<X_1,...,X_n\right>$, we have by taking the global section:
\begin{align}
{\square\square}\widetilde{\mathcal{C}}_{-/R}(.)(\mathcal{O}_{\mathrm{Spa}^\mathrm{BK}R\left<X_1,...,X_n\right>})&(\mathrm{Spa}^\mathrm{BK}R\left<X_1,...,X_m\right>):=\\
&\widetilde{\mathcal{C}}_{\mathrm{Spa}R\left<X_1,...,X_m\right>/R,\text{pro\'et}}(\mathrm{Spa}R\left<X_1,...,X_m\right>/R,\text{pro\'et})[k^{1/2}]\\
{\square\square}{\mathbb{B}_e}_{-/R}(.)(\mathcal{O}_{\mathrm{Spa}^\mathrm{BK}R\left<X_1,...,X_n\right>})&(\mathrm{Spa}^\mathrm{BK}R\left<X_1,...,X_m\right>):=\\
&{\mathbb{B}_e}_{\mathrm{Spa}R\left<X_1,...,X_m\right>/R,\text{pro\'et}}(\mathrm{Spa}R\left<X_1,...,X_m\right>/R,\text{pro\'et})[k^{-1/2}],\\
{\square\square}{\mathbb{B}_\mathrm{dR}^+}_{-/R}(.)(\mathcal{O}_{\mathrm{Spa}^\mathrm{BK}R\left<X_1,...,X_n\right>})&(\mathrm{Spa}^\mathrm{BK}R\left<X_1,...,X_m\right>):=\\
&{\mathbb{B}_\mathrm{dR}^+}_{\mathrm{Spa}R\left<X_1,...,X_m\right>/R,\text{pro\'et}}(\mathrm{Spa}R\left<X_1,...,X_m\right>/R,\text{pro\'et})[k^{1/2}],\\
{\square\square}{\mathbb{B}_\mathrm{dR}}_{-/R}(.)(\mathcal{O}_{\mathrm{Spa}^\mathrm{BK}R\left<X_1,...,X_n\right>})&(\mathrm{Spa}^\mathrm{BK}R\left<X_1,...,X_m\right>):=\\
&{\mathbb{B}_\mathrm{dR}}_{\mathrm{Spa}R\left<X_1,...,X_m\right>/R,\text{pro\'et}}(\mathrm{Spa}R\left<X_1,...,X_m\right>/R,\text{pro\'et})[k^{1/2}],\\
{\square\square}{{FF}}_{-/R}(.)(\mathcal{O}_{\mathrm{Spa}^\mathrm{BK}R\left<X_1,...,X_n\right>})&(\mathrm{Spa}^\mathrm{BK}R\left<X_1,...,X_m\right>):=\\
&{\square\square}{FF}_{\mathrm{Spa}R\left<X_1,...,X_m\right>/R,\text{pro\'et}}(\mathrm{Spa}R\left<X_1,...,X_m\right>/R,\text{pro\'et}).
\end{align}
And from \cite[Definition 9.3.3, Definition 9.3.5, Definition 9.3.11, Definition 9.3.9]{KL1} and \cite{KL2} we have the notation of $\varphi$-modules, ${\square\square}B$-pairs and the vector bundles over the ${\square\square}FF$ curves as above. We now use the notation $M$ to denote them. We then put $V(.)(\mathcal{O}_{\mathrm{Spa}^\mathrm{BK}R\left<X_1,...,X_n\right>})(\mathrm{Spa}^\mathrm{BK}R\left<X_1,...,X_m\right>):=M(\mathrm{Spa}R\left<X_1,...,X_m\right>/R,\text{pro\'et})$. For the stack ${\square\square}FF$, this $V$ will be a corresponding vector bundle at the end over ${\square\square}FF_{(R\left<X_1,...,X_n\right>)^\flat/R}$. 	
\end{definition}

\begin{definition}
Following \cite[Definition 9.3.3, Definition 9.3.5, Definition 9.3.11, Definition 9.3.9]{KL1}, \cite{KL2} we give the following definition. For any space
\begin{align}
(\mathbb{X},\mathcal{R})=\underset{n}{\mathrm{homotopylimit}}(\mathbb{X}_n,\mathcal{R}_n)	
\end{align}
in the $\infty$-categories:
\begin{align}
&\mathrm{Proj}^\mathrm{formalspectrum}\mathrm{Sta}^\mathrm{derivedringed,\sharp}_{\mathrm{sComm}\mathrm{Simplicial}\mathrm{Ind}\mathrm{Seminormed}_R,\mathrm{homotopyepi}},\\
&\mathrm{Proj}^\mathrm{formalspectrum}\mathrm{Sta}^\mathrm{derivedringed,\sharp}_{\mathrm{sComm}\mathrm{Simplicial}\mathrm{Ind}^m\mathrm{Seminormed}_R,\mathrm{homotopyepi}},\\
&\mathrm{Proj}^\mathrm{formalspectrum}\mathrm{Sta}^\mathrm{derivedringed,\sharp}_{\mathrm{sComm}\mathrm{Simplicial}\mathrm{Ind}\mathrm{Normed}_R,\mathrm{homotopyepi}},\\
&\mathrm{Proj}^\mathrm{formalspectrum}\mathrm{Sta}^\mathrm{derivedringed,\sharp}_{\mathrm{sComm}\mathrm{Simplicial}\mathrm{Ind}^m\mathrm{Normed}_R,\mathrm{homotopyepi}},\\
&\mathrm{Proj}^\mathrm{formalspectrum}\mathrm{Sta}^\mathrm{derivedringed,\sharp}_{\mathrm{sComm}\mathrm{Simplicial}\mathrm{Ind}\mathrm{Banach}_R,\mathrm{homotopyepi}},\\
&\mathrm{Proj}^\mathrm{formalspectrum}\mathrm{Sta}^\mathrm{derivedringed,\sharp}_{\mathrm{sComm}\mathrm{Simplicial}\mathrm{Ind}^m\mathrm{Banach}_R,\mathrm{homotopyepi}},	
\end{align}
we define the corresponding $\infty$-functors:
\begin{align}
{\square\square}\widetilde{\mathcal{C}}_{-/R}(.),{\square\square}{\mathbb{B}_e}_{-/R}(.),{\square\square}{\mathbb{B}_\mathrm{dR}^+}_{-/R}(.),{\square\square}{\mathbb{B}_\mathrm{dR}}_{-/R}(.),{\square\square}{FF}_{-/R}(.),V(.)	
\end{align}
as:
\begin{align}
&{\square\square}\widetilde{\mathcal{C}}_{-/R}(.)(\mathcal{R}):=\underset{n}{\mathrm{homotopycolimit}}~{\square\square}\widetilde{\mathcal{C}}_{-/R}(\mathcal{R}_n),\\
&{\square\square}{\mathbb{B}_e}_{-/R}(.)(\mathcal{R}):=\underset{n}{\mathrm{homotopycolimit}}~{\square\square}{\mathbb{B}_e}_{-/R}(.)(\mathcal{R}_n),\\
&{\square\square}{\mathbb{B}_\mathrm{dR}^+}_{-/R}(.)(\mathcal{R}):=\underset{n}{\mathrm{homotopycolimit}}~{\square\square}{\mathbb{B}_\mathrm{dR}^+}_{-/R}(.)(\mathcal{R}_n),\\
&{\square\square}{\mathbb{B}_\mathrm{dR}}_{-/R}(.)(\mathcal{R}):=\underset{n}{\mathrm{homotopycolimit}}~{\square\square}{\mathbb{B}_\mathrm{dR}}_{-/R}(.)(\mathcal{R}_n),\\
&{\square\square}{{FF}}_{-/R}(.)(\mathcal{R}):=\underset{n}{\mathrm{homotopylimit}}~{\square\square}{\square\square}{{FF}}_{-/R}(.)(\mathcal{R}_n),\\	
&{\square\square}V(.)(\mathcal{R}):=\underset{n}{\mathrm{homotopycolimit}}~{\square\square}V(.)(\mathcal{R}_n).
\end{align}
Here the homotopy colimits are taken in the corresponding colimit completions of the categories where the rings and spaces are living. Here the homotopy limits are taken in the corresponding limit completions of the categories where the rings and spaces are living.
\end{definition}

\newpage
\section{Mixed-Parity Robba Stacks in the Inductive System Situations}

\begin{reference}
\cite{KL1}, \cite{KL2}, \cite{Sch1}, \cite{Sch}, \cite{Fon}, \cite{FF}, \cite{F1}, \cite{Ta}.
\end{reference}

Now we consider the construction from \cite[Definition 9.3.3, Definition 9.3.5, Definition 9.3.11, Definition 9.3.9]{KL1} and \cite{KL2}, and apply the functors in \cite[Definition 9.3.3, Definition 9.3.5, Definition 9.3.11, Definition 9.3.9]{KL1} and \cite{KL2} to the rings and spaces in our current $\infty$-categorical context. Now let $R$ be any analytic field $\mathcal{K}$ over $\mathbb{Q}_p$. Recall from \cite[Definition 9.3.3, Definition 9.3.5, Definition 9.3.11, Definition 9.3.9]{KL1} we have the following functors:
\begin{align}
{\square\square}\widetilde{\mathcal{C}}_{-/R}(.),{\square\square}{\mathbb{B}_e}_{-/R}(.),{\square\square}{\mathbb{B}_\mathrm{dR}^+}_{-/R}(.),{\square\square}{\mathbb{B}_\mathrm{dR}}_{-/R}(.),{\square\square}{FF}_{-/R}(.)	
\end{align}
on the following ringed spaces from \cite{BK}:
\begin{align}
(\mathrm{Spa}^\mathrm{BK}R\left<X_1,...,X_n\right>,\mathcal{O}_{\mathrm{Spa}^\mathrm{BK}R\left<X_1,...,X_n\right>}),n=0,1,2,...	
\end{align}
We then have the situation to promote the functors of rings and stacks to the $\infty$-categorical context as above. Here recall that from \cite[Definition 9.3.3, Definition 9.3.5, Definition 9.3.11, Definition 9.3.9]{KL1} and \cite{KL2}:

\begin{definition}
For any $\mathrm{Spa}^\mathrm{BK}R\left<X_1,...,X_n\right>$, we have by taking the global section:
\begin{align}
{\square\square}\widetilde{\mathcal{C}}_{-/R}(.)(\mathcal{O}_{\mathrm{Spa}^\mathrm{BK}R\left<X_1,...,X_n\right>})&(\mathrm{Spa}^\mathrm{BK}R\left<X_1,...,X_m\right>):=\\
&\widetilde{\mathcal{C}}_{\mathrm{Spa}R\left<X_1,...,X_m\right>/R,\text{pro\'et}}(\mathrm{Spa}R\left<X_1,...,X_m\right>/R,\text{pro\'et})[k^{1/2}]\\
{\square\square}{\mathbb{B}_e}_{-/R}(.)(\mathcal{O}_{\mathrm{Spa}^\mathrm{BK}R\left<X_1,...,X_n\right>})&(\mathrm{Spa}^\mathrm{BK}R\left<X_1,...,X_m\right>):=\\
&{\mathbb{B}_e}_{\mathrm{Spa}R\left<X_1,...,X_m\right>/R,\text{pro\'et}}(\mathrm{Spa}R\left<X_1,...,X_m\right>/R,\text{pro\'et})[k^{-1/2}],\\
{\square\square}{\mathbb{B}_\mathrm{dR}^+}_{-/R}(.)(\mathcal{O}_{\mathrm{Spa}^\mathrm{BK}R\left<X_1,...,X_n\right>})&(\mathrm{Spa}^\mathrm{BK}R\left<X_1,...,X_m\right>):=\\
&{\mathbb{B}_\mathrm{dR}^+}_{\mathrm{Spa}R\left<X_1,...,X_m\right>/R,\text{pro\'et}}(\mathrm{Spa}R\left<X_1,...,X_m\right>/R,\text{pro\'et})[k^{1/2}],\\
{\square\square}{\mathbb{B}_\mathrm{dR}}_{-/R}(.)(\mathcal{O}_{\mathrm{Spa}^\mathrm{BK}R\left<X_1,...,X_n\right>})&(\mathrm{Spa}^\mathrm{BK}R\left<X_1,...,X_m\right>):=\\
&{\mathbb{B}_\mathrm{dR}}_{\mathrm{Spa}R\left<X_1,...,X_m\right>/R,\text{pro\'et}}(\mathrm{Spa}R\left<X_1,...,X_m\right>/R,\text{pro\'et})[k^{1/2}],\\
{\square\square}{{FF}}_{-/R}(.)(\mathcal{O}_{\mathrm{Spa}^\mathrm{BK}R\left<X_1,...,X_n\right>})&(\mathrm{Spa}^\mathrm{BK}R\left<X_1,...,X_m\right>):=\\
&{\square\square}{FF}_{\mathrm{Spa}R\left<X_1,...,X_m\right>/R,\text{pro\'et}}(\mathrm{Spa}R\left<X_1,...,X_m\right>/R,\text{pro\'et}).
\end{align}
And from \cite[Definition 9.3.3, Definition 9.3.5, Definition 9.3.11, Definition 9.3.9]{KL1} and \cite{KL2} we have the notation of $\varphi$-modules, ${\square\square}B$-pairs and the vector bundles over the ${\square\square}FF$ curves as above. We now use the notation $M$ to denote them. We then put $V(.)(\mathcal{O}_{\mathrm{Spa}^\mathrm{BK}R\left<X_1,...,X_n\right>})(\mathrm{Spa}^\mathrm{BK}R\left<X_1,...,X_m\right>):=M(\mathrm{Spa}R\left<X_1,...,X_m\right>/R,\text{pro\'et})$. For the stack ${\square\square}FF$, this $V$ will be a corresponding vector bundle at the end over ${\square\square}FF_{(R\left<X_1,...,X_n\right>)^\flat/R}$. 	
\end{definition}

\begin{definition}
Following \cite[Definition 9.3.3, Definition 9.3.5, Definition 9.3.11, Definition 9.3.9]{KL1}, \cite{KL2} we give the following definition. For any space
\begin{align}
(\mathbb{X},\mathcal{R})=\underset{n}{\mathrm{homotopycolimit}}(\mathbb{X}_n,\mathcal{R}_n)	
\end{align}
in the $\infty$-categories:
\begin{align}
&\mathrm{Ind}^\mathrm{formalspectrum}\mathrm{Sta}^\mathrm{derivedringed,\sharp}_{\mathrm{sComm}\mathrm{Simplicial}\mathrm{Ind}\mathrm{Seminormed}_R,\mathrm{homotopyepi}},\\
&\mathrm{Ind}^\mathrm{formalspectrum}\mathrm{Sta}^\mathrm{derivedringed,\sharp}_{\mathrm{sComm}\mathrm{Simplicial}\mathrm{Ind}^m\mathrm{Seminormed}_R,\mathrm{homotopyepi}},\\
&\mathrm{Ind}^\mathrm{formalspectrum}\mathrm{Sta}^\mathrm{derivedringed,\sharp}_{\mathrm{sComm}\mathrm{Simplicial}\mathrm{Ind}\mathrm{Normed}_R,\mathrm{homotopyepi}},\\
&\mathrm{Ind}^\mathrm{formalspectrum}\mathrm{Sta}^\mathrm{derivedringed,\sharp}_{\mathrm{sComm}\mathrm{Simplicial}\mathrm{Ind}^m\mathrm{Normed}_R,\mathrm{homotopyepi}},\\
&\mathrm{Ind}^\mathrm{formalspectrum}\mathrm{Sta}^\mathrm{derivedringed,\sharp}_{\mathrm{sComm}\mathrm{Simplicial}\mathrm{Ind}\mathrm{Banach}_R,\mathrm{homotopyepi}},\\
&\mathrm{Ind}^\mathrm{formalspectrum}\mathrm{Sta}^\mathrm{derivedringed,\sharp}_{\mathrm{sComm}\mathrm{Simplicial}\mathrm{Ind}^m\mathrm{Banach}_R,\mathrm{homotopyepi}},	
\end{align}

we define the corresponding $\infty$-functors:
\begin{align}
{\square\square}\widetilde{\mathcal{C}}_{-/R}(.),{\square\square}{\mathbb{B}_e}_{-/R}(.),{\square\square}{\mathbb{B}_\mathrm{dR}^+}_{-/R}(.),{\square\square}{\mathbb{B}_\mathrm{dR}}_{-/R}(.),{\square\square}{FF}_{-/R}(.),V(.)	
\end{align}
as:
\begin{align}
&{\square\square}\widetilde{\mathcal{C}}_{-/R}(.)(\mathcal{R}):=\underset{n}{\mathrm{homotopylimit}}~{\square\square}\widetilde{\mathcal{C}}_{-/R}(\mathcal{R}_n),\\
&{\square\square}{\mathbb{B}_e}_{-/R}(.)(\mathcal{R}):=\underset{n}{\mathrm{homotopylimit}}~{\square\square}{\mathbb{B}_e}_{-/R}(.)(\mathcal{R}_n),\\
&{\square\square}{\mathbb{B}_\mathrm{dR}^+}_{-/R}(.)(\mathcal{R}):=\underset{n}{\mathrm{homotopylimit}}~{\square\square}{\mathbb{B}_\mathrm{dR}^+}_{-/R}(.)(\mathcal{R}_n),\\
&{\square\square}{\mathbb{B}_\mathrm{dR}}_{-/R}(.)(\mathcal{R}):=\underset{n}{\mathrm{homotopylimit}}~{\square\square}{\mathbb{B}_\mathrm{dR}}_{-/R}(.)(\mathcal{R}_n),\\
&{\square\square}{{FF}}_{-/R}(.)(\mathcal{R}):=\underset{n}{\mathrm{homotopycolimit}}~{\square\square}{{FF}}_{-/R}(.)(\mathcal{R}_n),\\	
&{\square\square}V(.)(\mathcal{R}):=\underset{n}{\mathrm{homotopylimit}}~{\square\square}V(.)(\mathcal{R}_n).
\end{align}
\end{definition}

\

\begin{remark}
Here the homotopy colimits are taken in the corresponding colimit completions of the categories where the rings and spaces are living. Here the homotopy limits are taken in the corresponding limit completions of the categories where the rings and spaces are living. For instance, the Robba functor ${\square\square}\widetilde{\mathcal{C}}_{-/R}(.)$ takes value in ind-Fr\'echet rings $\mathrm{Ind}\text{Fr\'echet}_R$, we then consider the corresponding homotopy limit closure $\overline{\mathrm{Ind}\text{Fr\'echet}}^{\mathrm{homotopylimit}}_R$. For instance, the Fargues-Fontaine stack functors ${\square\square}{{FF}}_{-/R}(.)(-)$ take value in the preadic spaces $\mathrm{PreAdic}_R$, then we consider the corresponding homotopy colimit closure $\overline{\mathrm{PreAdic}_R}^{\mathrm{homotopycolimit}}$. 	
\end{remark}

\newpage

\subsection*{Acknowledgements} 

The author thanks Professor Kedlaya for conversation on topologization and functional analytification. Idea of considering mixed-parity extension was learnt from Professor Sorensen, and we thank Professor Sorensen for the conversation on this classical $p$-adic Hodge theoretic consideration which motivates our current topologization and functional analytification for more fashionable $p$-adic cohomologizations.

\end{document}